\RequirePackage{fix-cm}
\documentclass[smallextended]{svjour3}       
\smartqed  
\usepackage{graphicx}
\usepackage{amsmath,amsfonts,amssymb}
\usepackage{mathtools}  
\usepackage{breqn}
\usepackage{tcolorbox}
\usepackage{multirow}
\usepackage{hyperref}
\usepackage{enumerate}
\usepackage{subfigure}
\usepackage[title]{appendix}
\usepackage{lineno}
\usepackage[normalem]{ulem}
\usepackage{lmodern}
\usepackage{cite}
\usepackage{float}

\hypersetup{
	colorlinks=true, 
	linkcolor=blue, 
	urlcolor=red, 
	citecolor=magenta,
	linktoc=all 
}

\usepackage{xcolor}  
\definecolor{lightyellow}{HTML}{fffcf1}

\newcommand{\reva}[1]{\color{black}#1 \color{black}}

\definecolor{lightyellow}{HTML}{fffcf1}

\usepackage{bm}
\definecolor{lightyellow}{HTML}{fffcf1}

\newcommand{\con}{\mathbf{U}}       

\newcommand{\prim}{\mathbf{W}}      
\newcommand{\vel}{\bm{v}}
\newcommand{\af}{\mathbf{f}}
\newcommand{\ent}{\mathcal{U}}      
\newcommand{\entf}{\mathcal{F}}     

\newcommand{\df}[2]{\frac{\partial #1}{\partial #2}} 
\newcommand{\dd}[2]{\frac{\mathrm d #1}{\mathrm d #2}} 

\newcommand{\half}{{\frac{1}{2}}}   
\newcommand{\avg}[1]{\bar{#1}}      

\newcommand{\evar}{\bm{\mathcal{V}}}          

\newcommand{\Uh}{{\mathbf{U}_h}}
\newcommand{\Vkh}{{\mathbf{V}^k_h}}
\usepackage{upgreek} 
\newcommand{\ph}{{\boldsymbol{\Upphi}_h}}

\newcommand{\sour}{\bm{\mathcal{S}}}
\newcommand{\nsour}{\mathcal{S}}
\newcommand{\ener}{\mathcal{E}}
\newcommand{\Diff}{\mathbf{D}}
\newcommand{\Mass}{\mathbf{M}}
\newcommand{\Stiff}{\mathbf{S}}
\newcommand{\Bdry}{\mathbf{B}}

\newcommand{\eigval}{\mathbf{\Lambda}}

\newcommand{\imh}{{i-\half}}
\newcommand{\iph}{{i+\half}}
\newcommand{\jmh}{{j-\half}}
\newcommand{\jph}{{j+\half}}
\newcommand{\intdx}{\mathrm{d} x}
\newcommand{\dx}{\Delta x}
\newcommand{\dy}{\Delta y}

\newcommand{\second}{\mathbf{O2\_ESDG}}
\newcommand{\third}{\mathbf{O3\_ESDG}}
\newcommand{\fourth}{\mathbf{O4\_ESDG}}

\begin{document}

\title{Entropy stable discontinuous Galerkin schemes for two-fluid relativistic plasma flow equations}
\titlerunning{Entropy Stable DG schemes for two-fluid relativistic plasma flow equations}

\author{Deepak Bhoriya  \and
        Biswarup Biswas \and
        Harish Kumar    \and
        Praveen Chandrashekhar 
}

\institute{D. Bhoriya (Corresponding author) \at
           {Department of Mathematics, Indian Institute of Technology Delhi, New Delhi, India}
           \\ \email{dkbhoriya@gmail.com}
           \and
           B. Biswas \at
           {Department of Mathematics, Mahindra University, Hyderabad, 500043, India}
           \\ \email{biswarup.biswas@mahindrauniversity.edu.in}
           \and
           H. Kumar \at
           {Department of Mathematics, Indian Institute of Technology Delhi, New Delhi, India}
           \\ \email{hkumar@iitd.ac.in}
           \and
           P. Chandrashekhar \at
           {Centre for Applicable Mathematics, Tata Institute of Fundamental Research, Bangalore, India}
           \\ \email{praveen@tifrbng.res.in}
}

\date{Received: date / Accepted: date}

\maketitle

\begin{abstract}
This article proposes entropy stable discontinuous Galerkin schemes (DG) for two-fluid relativistic plasma flow equations. These equations couple the flow of relativistic fluids via electromagnetic quantities evolved using Maxwell's equations. The proposed schemes are based on the Gauss-Lobatto quadrature rule, which has the summation by parts (SBP) property. We exploit the structure of the equations having the flux with three independent parts coupled via nonlinear source terms. We design entropy stable DG schemes for each flux part, coupled with the fact that the source terms do not affect entropy, resulting in an entropy stable scheme for the complete system. The proposed schemes are then tested on various test problems in one and two dimensions to demonstrate their accuracy and stability.


\end{abstract}

	\section{Introduction}
Many astrophysical phenomena like gamma-ray bursts (GRBs), pulsar winds, relativistic jets from active galactic nuclei (AGNs), and quasars have fluid moving with speeds close to the speed of light~\cite{Landau1987,Gallant1994,Mochkovitch1995,Wardle1998}. In addition, the flows consist of charged particles that require electric and magnetic fields to be considered in the mathematical model. Hence, we need to model relativistic plasma flows; one of the simplest models is relativistic magnetohydrodynamics (RMHD)~\cite{Mochkovitch1995,VanPutten1996,Blandford2019}, and several numerical methods have been developed for this model~\cite{Komissarov1999,Balsara2001,Gammie2003,Anton2010,Kim2014,Mignone2006a,DelZanna2007,Balsara2016aderweno,Wu2021,Mattia2021,Duan2020DGRMHD,Duan2022}. However, for important astrophysical phenomena, the RMHD model is not suitable~\cite{Barkov2014,Amano2016,Balsara2016}. Following the developments in modeling of non-relativistic plasma flows using multi-fluid models~\cite{Baboolal2001,Shumlak2003,Hakim2006,bond2016plasma,Kumar2012}, recently several authors have considered two-fluid relativistic plasma flow models~\cite{Zenitani2009a,Zenitani2009b,Amano2013,Amano2016,Balsara2016}.
	
The two-fluid relativistic plasma flow model equations are a system of hyperbolic balance laws. Due to the nonlinear flux, solutions can develop discontinuities, even when the initial conditions are smooth and weak solutions need to be considered. However, often for hyperbolic PDEs, the weak solutions are not unique~\cite{GodlewskiEdwige1991,LeVeque2002} and an additional condition, in the form of entropy stability, is imposed to select the suitable solution. This ensures the uniqueness of the solutions in the case of scalar conservation laws~\cite{GodlewskiEdwige1991}. For nonlinear systems, recently, it has been shown that even the entropy criterion is not enough to guarantee uniqueness~\cite{Chiodaroli2015}. However, the entropy criteria provides a nonlinear stability estimate for the solution. So, it is desirable to have a numerical scheme that preserves such a stability estimate at the discrete level, leading to stable numerical schemes.
	
Recently, there have been considerable developments in designing higher-order entropy stable finite difference schemes~\cite{Tadmor1987, Tadmor2003, Fjordholm2012, Fjordholm2013, Ismail2009, Chandrashekar2013,Chandrashekar2016,Kumar2012, Sen2018}, and  the Discontinuous Galerkin (DG) methods in \cite{Carpenter2014,Gassner2016a,Gassner2016,Friedrich2019}. In \cite{Chen2017}, authors have presented a framework for entropy stable DG schemes based on the SBP property. This framework has been applied to various hyperbolic systems in \cite{Liu2018,Duan2021,Biswas2021}.
	
In addition to all the challenges in designing stable numerical methods for hyperbolic PDEs, there are several additional difficulties in developing a suitable numerical method for two-fluid relativistic plasma flow model equations. First, to compute primitive variables from the conserved variables, one needs to solve a nonlinear equation. Secondly, one needs to preserve the divergence constraints on magnetic and electric fields. Also, it is highly desirable to have a positivity-preserving and entropy stable scheme. However, it is difficult to design higher-order schemes satisfying all these properties, and often only one or few of these two nonlinear stability properties can be respected in a numerical scheme.
	
For relativistic hydrodynamics (RHD) several numerical schemes have been designed~\cite{Landau1987,Anile1989,Wilson2010}. An Eulerian explicit finite difference scheme is designed in ~\cite{Wilson1979}.  In~\cite{Falle1996}, the authors propose a Van-Leer based scheme. High-order schemes using piece-wise parabolic reconstruction are proposed in~\cite{Mignone2005,Marti1996}.  In~\cite{Ryu2006}, authors have described the total variation diminishing schemes. An adaptive mesh refinement (AMR)  based finite difference WENO scheme is presented in~\cite{Zhang2006}. More recently,  bound-preserving discontinuous Galerkin methods~\cite{Qin2016},  and physical-constraint-preserving central DG schemes~\cite{Wu2016} are also proposed. For entropy stability,  entropy stable finite difference schemes~\cite{Bhoriya2020} and entropy stable adaptive moving mesh schemes~\cite{Duan2021adaptivemeshRHD} have been proposed. Also recently, in~\cite{Biswas2022}, authors designed entropy stable DG schemes for RHD using the framework in \cite{Chen2017}.
	
For the relativistic plasma flow model based on RMHD equations, \cite{Komissarov1999} introduced a Godunov type scheme. An overview of grid-based numerical methods used in relativistic hydrodynamics and magnetohydrodynamics is presented in~\cite{Marti2015}. A total variation diminishing scheme is designed in~\cite{Balsara2001}. An ADER-WENO formulation is addressed in~\cite{Balsara2016aderweno}. A finite difference entropy stable scheme is designed in~\cite{Wu2020}, while a DG version of the entropy stable schemes, based on the nodal formulation, is given in~\cite{Duan2020DGRMHD}. Recently, the authors in~\cite{Duan2022} extended the entropy stable finite difference schemes to use adaptive mesh refinements. A discontinuous Galerkin scheme that preserves the physical constraints is presented in~\cite{Wu2021}.
	
For the two-fluid relativistic plasma flow model, in \cite{Zenitani2009a,Zenitani2009b}, authors have simulated the relativistic magnetic reconnection problem. Simulations of the standing shock front are performed in \cite{Amano2013}, which is used to study the importance of an electric field-dominated regime in strongly magnetized relativistic plasma flows. A third-order accurate numerical scheme was proposed in \cite{Barkov2014}. The issues of divergence-free magnetic field and constraint-preserving electric field are addressed in \cite{Amano2016,Balsara2016}. More recently, in \cite{bhoriya_high-order_2022}, authors have proposed high-order entropy stable finite difference numerical schemes. In addition, an implicit-explicit (IMEX) time stepping is also proposed to overcome the stiff source terms.
	
In this article, we design high-order discontinuous Galerkin entropy stable numerical schemes for the two-fluid relativistic plasma flow equations. We proceed as follows:
\begin{itemize}
\item The flux vector of two-fluid relativistic plasma flow equations consists of three sets of independent fluxes, one each for the relativistic fluid (ion or electron) and one linear flux for Maxwell's equations. The fluid flux for ion (and electron) is exactly the RHD flux with ion (electron) fluid variables. These three parts are coupled via source terms only. Hence, we use RHD schemes proposed in \cite{Biswas2022} and \cite{Bhoriya2020} to discretize the flux.
		
\item The source terms do not affect entropy~\cite{bhoriya_high-order_2022}. Hence entropy evolution depends on the suitable discretization of the flux. We then use the framework of \cite{Chen2017} and couple it with the three-dimensional version of entropy conservative numerical flux in~\cite{Bhoriya2020,bhoriya_high-order_2022}. 
\end{itemize}

The rest of the article is organized as follows: In Section~\ref{sec:equations}, we present the multi-dimensional, two-fluid relativistic plasma flow equations. We also discuss the hyperbolicity of the two-fluid relativistic plasma system in two dimensions and present the entropy framework for the complete system. In Section~\ref{sec:DG_SCHEME}, we describe the entropy stable Discontinuous Galerkin methods in one and two dimensions. We begin by defining the Galerkin projection of the hyperbolic conservation law with integral approximations using the Gauss-Lobatto quadrature rule. The section also includes the construction of entropy conservative fluxes for the two-fluid relativistic plasma flow. We use the Local Lax-Friedrichs flux at the cell interfaces to obtain entropy stable schemes. The time discretization is presented in Section~\ref{sec:full_disc}. In Section~\ref{sec:num_test}, we present several one and two-dimensional numerical test cases.

	\section{Two-fluid relativistic plasma flow model equations} \label{sec:equations}
	
	The two-fluid relativistic plasma flow equations consist of electron and ion species. These species evolve in the presence of electric and magnetic fields. We use the subscript $\alpha$ to denote the ion and electron species, i.e., $\alpha \in \{i,e\}$, where `$i$' refers to the ion species and `$e$' refers to the electron species. Following ~\cite{Amano2016,Balsara2016}, we consider the following set of equations:
	\begin{subequations}\label{eq:TFRP_sys}
		\begin{align}
			%
			\frac{\partial D_\alpha}{\partial t} + \nabla \cdot (D_\alpha \vel_\alpha) & = 0, \label{density}
			\\
			%
			\frac{\partial \bm{M}_\alpha}{\partial t} + \nabla \cdot (\bm{M}_\alpha \vel_\alpha  + p_\alpha \mathbf{I}) & = r_\alpha \Gamma_\alpha \rho_\alpha (\bm{E}+\vel_\alpha \times \bm{B}), \label{momentum} 
			\\
			%
			\frac{\partial \mathcal{E}_\alpha}{\partial t} + \nabla \cdot ((\mathcal{E}_\alpha+p_\alpha)\vel_\alpha) & = r_\alpha \Gamma_\alpha \rho_\alpha (\vel_\alpha \cdot \bm{E}) ,	\label{energy}
			\\
			%
			\dfrac{\partial \bm{B}}{\partial t} +  \nabla \times \bm{E}
			+ \kappa \nabla \psi  &= 0, \label{max1}
			\\
			%
			\dfrac{\partial \bm{E}}{\partial t} - \nabla \times \bm{B}
			+ \chi \nabla \phi & = -\bm{j} , \label{max2}
			\\
			%
			 \dfrac{\partial \phi} {\partial t} + \chi\nabla \cdot \bm{E}
			& = \chi\rho_c, \label{max3}
			\\
			%
			\dfrac{\partial \psi} {\partial t} + \kappa\nabla \cdot \bm{B}
			& = 0, \label{max4}
			%
		\end{align}
	\end{subequations}
        %
        %
	where $D_\alpha, \bm{M}_\alpha,$ and $\mathcal{E}_\alpha$ are the conserved quantities, denoting mass density, momentum density, and total energy density, respectively. These conservative quantities can be expressed as
	\begin{align*}
		D_\alpha           & = \rho_\alpha \Gamma_\alpha,                \\
		\bm{M}_\alpha      & = \rho_\alpha h_\alpha \Gamma^2_\alpha \bm{v}_\alpha, \\
		\mathcal{E}_\alpha & = \rho_\alpha h_\alpha \Gamma^2_\alpha - p_\alpha,
	\end{align*}
	where, $\rho_\alpha, \bm{v}_\alpha, p_\alpha, \Gamma_\alpha$ and $h_\alpha$ are the proper mass density, velocity, isotropic gas pressure, Lorentz factor, and specific enthalpy, respectively. The speed of light is set to unity throughout the article; consequently, the Lorentz factor $\Gamma_\alpha$ is given by
	\begin{equation*}
		\Gamma_\alpha = \dfrac{1}{\sqrt{1 - |\bm{v}_\alpha|^2}}, \qquad \text{ with} \qquad \bm{v}_\alpha = \left(v_{x_\alpha},v_{y_\alpha}, v_{z_\alpha}\right).
	\end{equation*}
	The magnetic field vector is denoted by $\bm{B}=(B_x, B_y, B_z)$, and electric field vector is denoted by $\bm{E}=(E_x, E_y, E_z)$. The charge to mass ratio is denoted by $r_\alpha=\dfrac{q_\alpha}{m_\alpha},$ where $q_\alpha$ is the particle charge and $m_\alpha$ is the particle mass. The total charge density $\rho_c$ and the current density vector $\bm{j}$ are given by
	\begin{equation*}
		\rho_c = r_i \rho_i \Gamma_i  + r_e \rho_e \Gamma_e,
		\qquad
		\bm{j} = r_i \rho_i \Gamma_i \bm {v}_i + r_e \rho_e \Gamma_e \bm{v}_e.
	\end{equation*}
	We denote $\gamma_\alpha$ as the ratio of specific heats. Using the ideal equation of state, we introduce the enthalpy $h_\alpha$ as follows
	\begin{equation} \label{EOS}
		h_\alpha = 1+ \frac{\gamma_\alpha}{\gamma_\alpha - 1} \frac{p_\alpha}{\rho_\alpha}. 
	\end{equation}
	The polytropic index is given by $n_\alpha=k_\alpha-1$, and the sound speed is denoted by $c_\alpha$. The expression for $c_\alpha$ is given by
	\begin{equation*}
		c_\alpha^2=\frac{k_\alpha p_\alpha }{n_\alpha \rho_\alpha h_\alpha} \qquad
		\text{where} \qquad
		k_\alpha = \dfrac{\gamma_\alpha}{\gamma_\alpha-1}.
	\end{equation*}
To overcome the magnetic and electric field constraints, in \eqref{max1}-\eqref{max4}, we consider {\em perfectly hyperbolic formulation} of the Maxwell's equations from \cite{Munz2000}. We further remark that more appropriate methods to deal with these constraints are presented in \cite{Amano2013,Balsara2016}, however, they need a staggered evolution of the conservative variables.
	
	\subsection{Hyperbolicity} \label{subsec:hyp}
	Let us introduce the vector of conservative variables $\mathbf{U}$ as
	\begin{equation*}
		\mathbf{U}=(\mathbf{U}_i, \mathbf{U}_e, \mathbf{U}_m )^\top,
	\end{equation*}
	where
	\begin{equation*}
		\mathbf{U}_i = (D_i, \bm{M}_i, \mathcal{E}_i), \qquad
		\mathbf{U}_e = (D_e, \bm{M}_e, \mathcal{E}_e), \qquad
		\mathbf{U}_m = (\bm{B}, \ \bm{E},\phi,\psi).
	\end{equation*}
	In this article, we consider the two-dimensional case; the extension to three dimensions is straightforward and does not result in additional theoretical or implementation issues.  We use the notations $\mathbf{f}^x$, $\mathbf{f}^y$ for the $x-$ and $y-$directional fluxes, respectively. The fluxes $\mathbf{f}^x$ and $\mathbf{f}^y$ consist of three independent fluxes, two for each fluid and linear Maxwell's flux i.e. 
	\begin{equation*}
		\mathbf{f}^x = ( \mathbf{f}^x_i, \mathbf{f}^x_e, \mathbf{f}^x_m )^\top,
		\qquad
		\mathbf{f}^y = ( \mathbf{f}^y_i, \mathbf{f}^y_e, \mathbf{f}^y_m )^\top,
	\end{equation*}
	where the explicit expressions for $\mathbf{f}^d_\alpha$ ($\alpha\in \{i,e\},d\in\{x,y\}$) and $\mathbf{f}^d_m$  are given by
	\begin{align*}
		\mathbf{f}_\alpha^d & = \mathbf{f}_\alpha^d(\mathbf{U}_\alpha) =
		\begin{cases}
			(D_\alpha v_{x_\alpha}, M_{x_\alpha} v_{x_\alpha} + p_\alpha, M_{y_\alpha} v_{x_\alpha}, M_{z_\alpha} v_{x_\alpha}, M_{x_\alpha}) & \text{if } d=x, \\
			(D_\alpha v_{y_\alpha}, M_{x_\alpha} v_{y_\alpha} , M_{y_\alpha} v_{y_\alpha}+ p_\alpha, M_{z_\alpha} v_{y_\alpha}, M_{y_\alpha}) & \text{if } d=y,
		\end{cases}
		\intertext{and}
		\mathbf{f}_m^d      & = \mathbf{f}_m^d(\mathbf{U}_m) =
		\begin{cases}
			(\kappa \psi, - E_z,  E_y, \chi \phi,  B_z, - B_y, \chi E_x,\kappa B_x) & \text{if } d=x, \\
			( E_z, \kappa \psi, - E_x, - B_z, \chi \phi,  B_x, \chi E_y,\kappa B_y) & \text{if } d=y.
		\end{cases}, \label{exact_max_flux}
	\end{align*}
	Similarly, the source terms can be written in a convenient form by introducing the source vector $\sour= (\sour_i, \sour_e, \sour_m)^\top$, where the components have the following explicit expressions.
	\begin{align*}
		\sour_\alpha (\mathbf{U}_\alpha,\mathbf{U}_m)=   
		& 
		(0, r_\alpha D_\alpha(E_x + v_{y_\alpha} B_z - v_{z_\alpha}B_y), r_\alpha D_\alpha(E_y + v_{z_\alpha} B_x - v_{x_\alpha}B_z), \\
		&
		r_\alpha D_\alpha(E_z + v_{x_\alpha} B_y - v_{y_\alpha}B_x), 
		r_\alpha D_\alpha(v_{x_\alpha}E_x + v_{y_\alpha}E_y + v_{z_\alpha}E_z)), 
		\\
		\sour_m (\mathbf{U}_i,\mathbf{U}_e)= 
		& 
		(0, 0, 0, -j_x, -j_y, -j_z,\chi\rho_c,0).
	\end{align*}
	%
	Using the above notations, the system~\eqref{eq:TFRP_sys}, for the two-dimensional case, can be written in the balance law form as
	\begin{equation}
		\frac{\partial \mathbf{U}}{\partial t}+\frac{\partial \mathbf{f}^x}{\partial x} +\frac{\partial \mathbf{f}^y}{\partial y}= \sour \label{conservedform},
	\end{equation}
	and in quasi-linear form as
	\begin{equation} \label{conservedform_A}
		\frac{\partial \mathbf{U}}{\partial t}+ \mathbf{A}^x\frac{\partial \mathbf{U}}{\partial x } + \mathbf{A}^y\frac{\partial \mathbf{U}}{\partial y }= \sour \qquad 
		\text{where} \qquad
		\mathbf{A}^x=\frac{\partial \mathbf{f}^x}{\partial \mathbf{U}}, 
		\quad
		\mathbf{A}^y=\frac{\partial \mathbf{f}^y}{\partial \mathbf{U}}.
	\end{equation}
	The matrices $\mathbf{A}^x$ and $\mathbf{A}^y$ are the Jacobian matrices of the flux vectors with respect to the conservative vectors. We focus on the $x-$directional case, having the Jacobian matrix $\mathbf{A}^x$.
	
	Let us also introduce the vector of {\em primitive variables}, $\prim$=$(\prim_i, \prim_e, \prim_m)^\top$, where $\prim_i = (\rho_i, \vel_i, p_i)$, $\prim_e = (\rho_e, \vel_e, p_e)$ and $\prim_m=(\bm{B,E},\phi,\psi)$. Unlike the non-relativistic case, the conversion from conservative to primitive variables is not straightforward due to the presence of the Lorentz factor. We follow the \reva{numerical root-solver based} algorithm described in~\cite{Schneider1993,Bhoriya2020} to extract the primitive variables from the conservative variables. \reva{For the sake of completeness, we give a brief description of the algorithm in Appendix~\ref{ap:cons2prim}.}  Using the chain rule, the matrix $\mathbf{A}^x$  can be written as,
	\begin{equation*}
		\mathbf{A}^x= \frac{\partial \mathbf{f}^x}{\partial \mathbf{W}} \frac{\partial \mathbf{W}}{\partial \mathbf{U}}.
	\end{equation*}
	The eigenvalues of matrix $\mathbf{A}^x$ and $\frac{\partial \mathbf{f}^x}{\partial \mathbf{W}}$ are same since they are related by a similarity transformation, and can be shown to be
	\begin{align*}
		\eigval^x
		=
		\biggl\{
		& \frac{(1-c_i^2)v_{x_i}-(c_i/\Gamma_i) \sqrt{Q_i^x}}{1-c_i^2 |\vel_i|^2},\
		v_{x_i}, \ v_{x_i}, \ v_{x_i},
		\frac{(1-c_i^2)v_{x_i}+(c_i/\Gamma_i) \sqrt{Q_i^x}}{1-c_i^2 |\vel_i|^2},
		\nonumber
		\\
		& \frac{(1-c_e^2)v_{x_e}-(c_e/\Gamma_e) \sqrt{Q_e^x}}{1-c_e^2 |\vel_e|^2},\
		v_{x_e}, \ v_{x_e}, \ v_{x_e},
		\frac{(1-c_e^2)v_{x_e}+(c_e/\Gamma_e) \sqrt{Q_e^x}}{1-c_e^2 |\vel_e|^2},
		\nonumber
		\\
		& -\chi, \ -\kappa, \ -1, \ -1, \ 1, \ 1, \ \kappa, \ \chi
		\biggr\},
	\end{align*}
	where, $Q_\alpha^x=1-v_{x_{\alpha}}^2-c_\alpha^2 (v_{y_\alpha}^2+v_{z_\alpha}^2), \ \alpha \in \{ i, \ e \}$. The domain of admissible solutions is
	\begin{equation*}
		\Omega = \{ \con \in \mathbb{R}^{18} : \ \rho_i, \rho_e>0, \ p_i,p_e>0, \ |\vel_i|, |\vel_e| < 1 \}.
	\end{equation*}
	If $\con \in \Omega$ and $\gamma_\alpha \in (1,2]$ then $c_\alpha<1$ and $Q_\alpha^x > 0$ and all the eigenvalues are real. A complete set of right eigenvectors for the Jacobian  matrix $\mathbf{A}^x$ is given in Appendix~\eqref{ap:right_eig}. We can now state the following result:
	\begin{lemma}
		The conservative system~\eqref{conservedform} is hyperbolic for the states $\con \in \Omega$, with real set of eigenvalues and a complete set of right eigenvectors. 
	\end{lemma}
	\subsection{Entropy}
	\label{subsec:entropy}
	A strictly convex function $\ent(\con)$ is said to be an entropy function if there exist smooth functions $\entf^d(\con)$, ($d \in \{x,y\}$) called entropy fluxes satisfying $(\entf^d(\con))^\prime$=$\ent^\prime(\con) (\mathbf{f}^d(\con))^\prime$. Following \cite{bhoriya_high-order_2022,Bhoriya2020}, we define the following entropy function $\ent_\alpha$ and associated entropy flux $\entf^d_\alpha$, 
	\begin{equation*}
		\mathcal{U}_\alpha=-\frac{\rho_\alpha \Gamma_\alpha s_\alpha}{\gamma_\alpha -1} \qquad \text{ and } \qquad \mathcal{F}^d_\alpha=-\frac{\rho_\alpha \Gamma_\alpha s_\alpha v_{d_\alpha}}{\gamma_\alpha -1}, \ \alpha \in \{i,e\}, 
	\end{equation*}
	where $s_\alpha = \ln(p_\alpha \rho_\alpha^{-\gamma_\alpha})$. The pair $(\mathcal{U}_\alpha,\mathcal{F}^d)$  is called an {\em entropy-entropy flux pair}.
    For the electromagnetic part, we consider quadratic electromagnetic energy as the entropy function
	\begin{equation*}
		\mathcal{U}_m=\frac{|\bm{B}|^2+|\phi|^2}{2} +\frac{|\bm{E}|^2 + |\psi|^2}{2} .
	\end{equation*}
    For simplicity, let us consider the $x-$directional case. Then we have the following result from \cite{bhoriya_high-order_2022}.
	\begin{proposition} \label{prop:entropy}
		The smooth solutions of \eqref{conservedform} satisfy the entropy equality, 
		\begin{equation*}
			\partial_t s_\alpha+v_{x_\alpha} \partial_x s_\alpha=0, \ \ \alpha \in \{i,e\}. 
		\end{equation*}
		As a consequence, for smooth functions $H(s_\alpha)$ of $s_\alpha$ we have
		\begin{equation}
			\partial_t (\rho_\alpha \Gamma_\alpha H(s_\alpha))+ \partial_x (\rho_\alpha \Gamma_\alpha v_{x_\alpha} H(s_\alpha))=0. \label{h_s}
		\end{equation} 
		In particular, we have the following entropy equality for the smooth solutions,
		\begin{equation}
			\partial_t \mathcal{U}_\alpha+ \partial_x \mathcal{F}^x_\alpha=0. \label{entropy_pair_equality}
		\end{equation}
	\end{proposition}
	
	\begin{remark}
		The entropy equality~\eqref{entropy_pair_equality} in Proposition~\eqref{prop:entropy} is replaced with the entropy inequality
		\begin{equation}
			\partial_t \mathcal{U}_\alpha + \partial_x \mathcal{F}^x_\alpha \le 0 \label{ent_inq}
		\end{equation}   
		for non-smooth solutions. 
	\end{remark}
	\begin{remark} From the proof of the above result in \cite{bhoriya_high-order_2022}, we also note that the entropy evolution is not affected by the source terms. As a corollary, we have $\evar^\top \sour = 0$ where $\evar=\frac{\partial \ent}{\partial \con}$ is the entropy variable vector for the system~\eqref{eq:TFRP_sys}. 
        \label{remark:ent_dot_source}
	\end{remark}
	We now aim to design DG schemes that satisfy the estimate \eqref{ent_inq} at the discrete level which will be called an entropy stable scheme.
	\section{Entropy stable discontinuous Galerkin  schemes} \label{sec:DG_SCHEME}

	This section proposes entropy stable DG schemes for the two-fluid relativistic plasma in one and two dimensions. To simplify the discussion, we first consider the one-dimensional case.
	
	\subsection{One-dimensional schemes}\label{sec:1d_schemes}
	
	We consider the one-dimensional case of the hyperbolic system~\eqref{eq:TFRP_sys}, given by,
	\begin{equation}
		\frac{\partial \mathbf{U}}{\partial t}+\frac{\partial \mathbf{f}^x}{\partial x} = \sour \label{conservedform_1d}.
	\end{equation}
     The computational domain is taken to be $I^x=(x_{min},x_{max})$ which is discretized using $(N_x+1)$ points 
	\begin{equation*}
		x_{min} = x_{\half} < x_{\frac{3}{2}} < \dots < x_{N_x+\half} = x_{max},
		\quad
		N_x\in \mathbb{N}. 
	\end{equation*}
	We split the domain $I^x$ into computational cells $\left[x_{i-\half}, x_{i+\half}\right],$ $i \in \{1,2,\dots,N_x\},$ denoted by $I^x_i$ with the spatial step size $\dx_i$, given by $\dx_i = x_{i+\half}-x_{i-\half}$. For an interval $I^x_i$, we adopt the notation $\mathcal{P}^k(I_i^x)$ for the space of one-dimensional polynomials of degree at most $k$ on $I_i^x$. We now define the discrete finite element space as,
	\begin{equation}
		\Vkh = \{ \mathbf{U}_h : \mathbf{U}_h |_{I^x_i} \in [\mathcal{P}^k(I^x_i)]^{18}, \ 1 \le i \le N_x\},
		\label{element_space}
	\end{equation}
	where the subscript $h$ denotes the discrete space. In DG schemes, we find a solution, $\Uh \in \Vkh$, of the hyperbolic system~\eqref{conservedform_1d}, which satisfies,
	\begin{align}
		\int\limits_{I^x_i} \df{\Uh^\top}{t} \ph \intdx
		-
		\int\limits_{I^x_i} \mathbf{f}^x(\Uh)^\top \df{\ph}{x} \intdx
		= 
		&- 
		\left(
		\hat{\mathbf{f}}_{\iph}^\top \ph(x^-_{\iph})
		-
		\hat{\mathbf{f}}_{\imh}^\top \ph(x^+_{\imh})
		\right)
		\nonumber
		\\
		+
		\int\limits_{I^x_i} \sour^\top(\Uh) \ph \intdx
		&
		\qquad 
		\forall \ph \in \Vkh, 1 \le i \le N_x, 
		\label{dg_weak}
	\end{align}
	where,
	$
	\hat{\af}_{\iph} := \hat{\af} \left( \Uh(x^-_{\iph}) , \Uh(x^+_{\iph})\right)
	$
	is an interface flux. A simple application of the integration by parts on the weak form~\eqref{dg_weak} leads to the strong form
	\begin{align}
		\small
		\int\limits_{I^x_i}
		\left(
		\df{\Uh}{t} + \df{\af^x(\Uh)}{x}
		\right)^\top
		\ph \intdx
		=
		& -
		\Bigg[
		\left(\hat{\mathbf{f}}_{\iph} - \mathbf{f}^x(\Uh (x^-_{\iph})) \right)^\top \ph(x^-_{\iph})
		\nonumber
		\\
		& \qquad
		-
		\left(\hat{\mathbf{f}}_{\imh} - \mathbf{f}^x(\Uh (x^+_{\imh})) \right)^\top \ph(x^+_{\imh})
		\Bigg]
		\nonumber
		\\
		& +
		\int\limits_{I^x_i} \sour^\top(\Uh) \ph \intdx
		\qquad \forall \ph \in \Vkh, 1 \le i \le N_x.
		\label{dg_strong}
	\end{align}
	To further simplify the discussion, we first consider the case of scalar conservation law with source term and the scheme for each component of the conservation law looks like that of a scalar problem. The integrals in the weak formulation~\eqref{dg_weak} are approximated using the Legendre–Gauss–Lobatto quadrature rule~\cite{MiltoAbramowitzStegun1972} over the reference interval $I=[-1,1]$. The $k+1$ quadrature points in the interval $[-1,1]$ are denoted by
	\[
	-1 = \xi_0 < \xi_1 < \cdots < \xi_k = 1.
	\]
	The corresponding weights are denoted by the set $\omega=$$\{\omega_0, \omega_1,\dots,\omega_k\}$;  Table~\eqref{table:lobatto} shows the quadrature points and corresponding weights for different values of $k$~\cite{MiltoAbramowitzStegun1972}.
	\begin{table}[h]
		\centering
		\begin{tabular}{ |c|c|c| } 
			\hline
			$k$ & Quadrature points ($\xi_j$) & Weights ($\omega_j)$ \\ \hline
			\multirow{2}{*}{1} & $-1$ & $1$ \\
			& $\ \ 1$ & $1$ \\ \hline
			\multirow{3}{*}{2} & $-1$ & $1/3$ \\
			&$\ \ 0$ & $4/3$ \\
			&$\ \ 1$ & $1/3$ \\ \hline
			\multirow{4}{*}{3} & $-1$ 			& $1/6$ \\
			&$-1/\sqrt{5}$  & $5/6$ \\
			&$\ \ 1/\sqrt{5}$ & $5/6$ \\ 
			&$\ \ 1$ 			& $1/6$ \\
			\hline
		\end{tabular}
		\caption[h]{Gauss-Lobatto nodes and weights for $k=1,2,3$.}
		\label{table:lobatto}
	\end{table}
	Now, using the change of variables from cell $I^x_i$ to the  reference cell $I$, 
	\begin{equation*}
		x_i(\xi) = \dfrac{x_{\imh} + x_{\iph}}{2} + \dfrac{\dx_i}{2} \xi,
	\end{equation*}
	the formulation~\eqref{dg_weak}, for the scalar conservation law with source, can be written in an equivalent form as,
	\begin{align}
		\dfrac{\dx_i}{2}\int\limits_{I} \df{{U}_h^i}{t} \phi_h \mathrm{d} \xi
		-
		\int\limits_{I} {f}^x({U}_h^i) \df{\phi_h}{\xi} \mathrm{d} \xi
		=& - \left(
		\hat{{f}}_{\iph} \phi_h(x_i(1))
		-
		\hat{{f}}_{\imh} \phi_h(x_i(-1))
		\right)
		\nonumber
		\\
		&+
		\dfrac{\dx_i}{2}\int\limits_{I} {\nsour}({U}_h^i) \phi_h \mathrm{d} \xi.
		\label{scalar_form}
	\end{align}
	Let us now consider the following Lagrangian nodal basis set $\mathcal{B}$ for the space $\Vkh$,
	\begin{equation*}
		\mathcal{B}=\{L_m(\xi)\}_{m = 0}^{k} \quad 
		\text{ with }
		\quad
		L_m(\xi) = \prod_{j = 0, j \neq m }^k \dfrac{\xi - \xi_j}{\xi_m - \xi_j}, 
	\end{equation*}
	on the reference cell $I$.  For each $j,m \in \{0,1,2,\dots,k\}$, we have the interpolation property of Lagrange polynomials
	\begin{equation}
		L_m(\xi_j) = \delta_{jm} = \begin{cases}
			1 & \text{if } j = m,  
			\\
			0 & \text{otherwise}. 
		\end{cases}	\label{basis}
	\end{equation}
 The solution in each element is approximated in terms of the Lagrange polynomials as basis functions
 \[
			U^i_h(x(\xi)) = \sum_{j=0}^{k} U_j^i L_j(\xi), \qquad U_j^i = U_h(x_i(\xi_j))
\] 
where the coefficients $U_j^i$ are the values of the solution $U_h$ at the nodes inside the cell. Furthermore, the following approximations are used for the integrands in integral equation~\eqref{scalar_form}
	\begin{equation}
		\begin{rcases}
			{f}^x(x(\xi)) =  \sum_{j=0}^{k} {f}^x(U^i_j)   L_j(\xi),
			\\
			\nsour(x(\xi)) = \sum_{j=0}^{k} \nsour(U^i_j) L_j(\xi),
		\end{rcases}.
	\end{equation} 
	which consist of interpolation of flux and source using the same nodal values as the solution. We can now replace the continuous inner product with the discrete inner product of  $f$ and  $g$ on the space $\Vkh$, defined as,
	\begin{equation}
		\langle f, g \rangle_\omega := \sum_{j = 0}^{k} \omega_j f(\xi_j) g(\xi_j).
		\label{ip}
	\end{equation}
	To obtain compact expressions for the scheme, we introduce the following notations:
	\begin{itemize}
		\item The difference matrix $\Diff =\{D_{jm}\}_{0 \le j,m \le k}$ with $D_{jm}=L^\prime_m(\xi_j)$. 
		\item The mass matrix $\Mass =\{M_{jm}\}_{0 \le j,m \le k}$ with $M_{jm}= \langle L_j, L_m \rangle_\omega$.  
		\item The stiffness matrix $\Stiff =\{S_{jm}\}_{0 \le j,m \le k}$ with $S_{jm}= \langle L_j, L_m^\prime \rangle_\omega$.  
		\item The boundary matrix $\Bdry =\text{diag}[\tau_0, \tau_1,\dots, \tau_k]$ where $\tau_{j}=  
		\begin{cases}
			-1	&	\text{if } j=0\\
			1	&	\text{if } j=k\\
			0	&	\text{otherwise}
		\end{cases}
		$  
	\end{itemize}
	Using the interpolation property of Lagrange polynomials, we obtain $M_{jm} = \omega_j \delta_{jm}$, i.e., the mass matrix is diagonal. Moreover, the matrices satisfy the discrete analog of the integration by parts, known as Summation By Parts (SBP) properties~\cite{Chen2017,Carpenter2014}, given by
	\begin{align}
		\Stiff=\Mass \Diff, \qquad \Mass \Diff + \Diff^\top \Mass = \Stiff + \Stiff^\top = \Bdry. \label{SBP}
	\end{align}
	As a result, we have the following row sum and column sum property~\cite{Chen2017} for the mass matrix $\Mass$ and the stiffness matrix $\Stiff$
	\begin{enumerate}
		\item Row sum property: For $0 \le j \le k$
		\begin{align}
			\begin{rcases}
				\sum_{m=0}^{k} D_{jm} = 0 \\
				\sum_{m=0}^{k} S_{jm} = 0
			\end{rcases}.
			\label{row_sum}
		\end{align}
		\item Column sum property: For $0 \le m \le k$
		\begin{align}
			\sum_{j=0}^{k} S_{jm} = \tau_m.
			\label{column_sum}
		\end{align}
	\end{enumerate}
	Now choosing $\phi_h = L_j, \ 0 \le j \le k,$ as the test functions, we write the integral Eqn.~\eqref{scalar_form} as,
	\begin{align}
		\dfrac{\dx_i}{2} \Mass \frac{d\vec{U^i}}{dt} 
		-
		\Stiff^\top {\vec{f}}^{x^i}
		= - \Bdry  \vec{f}^{x^i}_* 
		+
		\dfrac{\dx_i}{2} \Mass \vec{\nsour^i} 
		\label{compact_weak_form}
	\end{align}
	where, we have used the following vector notations
	\begin{align}
		\begin{rcases}
			\vec{U^i} &= [U^i_0,U^i_1,\dots,U^i_k]^\top = \left[U_h(x_i(\xi_0)), U_h(x_i(\xi_1)), \dots, U_h(x_i(\xi_k))\right]^\top,
			\\
			\vec{f}^{x^i} &= [f^{x^i}_0,f^{x^i}_1,\dots,f^{x^i}_k]^\top = \left[f^x(U^i_0), f^x(U^i_1), \dots, f^x(U^i_k)\right]^\top,
			\\
			\vec{f}^{x^i}_* &= [f^{x^i}_{*,0},f^{x^i}_{*,1},\dots,f^{x^i}_{*,k}]^\top = \left[\hat{f}_{\imh}, 0, \dots,0, \hat{f}_{\iph}\right]^\top,
			\\
			\vec{\nsour^i} &= [\nsour^i_0, \nsour^i_1, \dots, \nsour^i_k]^\top = \left[\nsour(U^i_0), \nsour(U^i_1), \dots, \nsour(U^i_k)\right]^\top,
		\end{rcases}.
		\label{vector_notations}
	\end{align}
	The discretized weak form, which is given by Eqn.~\eqref{compact_weak_form}, can be further simplified using the SBP properties~\eqref{SBP} to give the following form
	\begin{align}
		\dfrac{\dx_i}{2} \frac{d\vec{U^i}}{dt} 
		+
		\Diff {\vec{f}}^{x^i}
		=  \Mass^{-1} \Bdry  (\vec{f}^{x^i} -\vec{f}^{x^i}_* )
		+
		\dfrac{\dx_i}{2}  \vec{\nsour^i} 
		\label{compact_strong_form}
	\end{align}
 Suppressing the cell index $i$, we can write the above set of equations in system form as
	\begin{align}
		\dfrac{\dx}{2} \frac{d{\con_j}}{dt} 
		+
		\sum_{m=0}^{k} D_{jm} {\mathbf{f}}^{x}_m
		=  \frac{\tau_j}{\omega_j} ({\mathbf{f}}^{x}_j - {\mathbf{f}}^{x}_{*,j} )
		+
		\dfrac{\dx}{2}  {\sour_j}, \qquad 0 \le j \le k. 
		\label{element_system}
	\end{align}
	where $\con_j$ is a vector of length 18 at each node $j$. For the system \eqref{conservedform_1d}, for the vectors $\con_i$, $\con_e$ and $\con_m$, scheme~\eqref{element_system} is,
	\begin{align}
		\dfrac{\dx}{2} \frac{d{\con_{\alpha,j}}}{dt} 
		+
		\sum_{m=0}^{k} D_{jm} {\mathbf{f}}^{x}_{\alpha,j}
		=  \frac{\tau_j}{\omega_j}  ({\mathbf{f}}^{x}_{\alpha,j} - {\mathbf{f}}^{x}_{\alpha,*,j} )
		+
		\dfrac{\dx}{2}  {\sour_{\alpha,m}},
		\nonumber
		\\
		\qquad \alpha\in\{i,e,m\},
		 \ 0 \le j \le k. 
		\label{element_fluid}
	\end{align}
	In general, the solutions of the system~\eqref{conservedform_1d} obtained using the scheme~\eqref{element_fluid} may not provably satisfy the entropy condition at the discrete level.  To ensure satisfaction of entropy condition, we have to design suitable numerical fluxes in the scheme~\eqref{element_fluid}; we first design the entropy conservative schemes, followed by the entropy stable schemes. 
	
	We proceed by introducing the entropy variables, denoted by $\bm{\mathcal{V}}_\alpha(\mathbf{U}_\alpha)$, and entropy potentials, denoted by $\psi_\alpha^x(\con_\alpha)$.
	\begin{align}
		\evar_\alpha(\con_\alpha)
		=
		\df{\ent_\alpha}{{\con_\alpha}},  
		\quad 
		\psi_\alpha^x(\con_\alpha)=\evar_\alpha^\top(\con_\alpha) \cdot \mathbf{f}_\alpha^x(\con_\alpha)-\entf_\alpha^d(\con_\alpha), \qquad \alpha \in \{i,e,m\}.
		\label{entropy_defn_general}
	\end{align}
	Using the entropy variables $\evar=$$(\bm{\mathcal{V}}_i,\bm{\mathcal{V}}_e,\bm{\mathcal{V}}_m)^\top$ and entropy potentials $\psi^x = (\psi^x_i, \psi^x_e, \psi^x_m)^\top$, we introduce the following definitions for two point numerical fluxes which form an essential ingredient of entropy consistent schemes to be developed next.
	\begin{definition}
		Let subscripts $L$ and $R$ denote two admissible states. We say a numerical flux $\mathbf{F}^x(\con_L,\con_R)$ is consistent, if
		\begin{align}
			\mathbf{F}^x(\con,\con)=\mathbf{f}^x(\con) \qquad \forall \con
            \label{def_consistent_flux}
		\end{align}
  and symmetric if
  \[
				\mathbf{F}^x(\con_L,\con_R) =  \mathbf{F}^x(\con_R,\con_L), \qquad \forall \con_L, \con_R
    \]
	\end{definition}
	\begin{definition}
		\label{def:EC_ES}
		Let subscripts $L$ and $R$ denote any two admissible states. We say a consistent numerical flux $\mathbf{F}^x(\con_L,\con_R)$ is
		\begin{itemize}
			\item Entropy conservative if it is symmetric and 
			\begin{align}
				( \evar_R - \evar_L)^\top
				\cdot
				\mathbf{F}^x(\con_L,\con_R)
				=
				\psi^x_R - \psi^x_L.
				\label{def_EC}
			\end{align}
			\item Entropy stable if
			\begin{align}
				( \evar_R - \evar_L)^\top
				\cdot
				\mathbf{F}^x(\con_L,\con_R)
				\le
				\psi^x_R - \psi^x_L.
				\label{def_ES}
			\end{align}
		\end{itemize} 
	\end{definition}
 
	\subsubsection{Entropy conservative schemes}
	We adopt the approach given in~\cite{Chen2017} and use a symmetric entropy conservative flux, denoted by $\mathbf{F}^x_{EC} = (\mathbf{F}^x_{i,EC},  \mathbf{F}^x_{e,EC},  \mathbf{F}^x_{m,EC})^\top$, in the scheme~\eqref{element_fluid} to obtain the modified scheme as 
	\begin{align}
		\dfrac{\dx}{2} \frac{d{\con_{j}}}{dt} 
		+
		2 \sum_{m=0}^{k} D_{jm} {\mathbf{F}_{EC}^x} (\con_j,\con_m)
		= \frac{\tau_j}{\omega_j} ({\mathbf{f}}^{x}_{j} - {\mathbf{f}}^{x}_{*,j} )
		+
		\dfrac{\dx}{2}  {\sour_{j}}, \qquad  \ 0 \le j \le k. 
		\label{scheme_1D}
	\end{align}
	For the one-dimensional case, we note that the discrete integrals of $\con$ and $\ent$ in the cell $\Tilde{I}$ of size $\dx$ are given by
	\begin{equation}
		\int_{\tilde{I}} \mathbf{U} \mathrm{d}x \approx \frac{\dx}{2} \sum_{j = 0}^k \omega_j \mathbf{U}_j  
		\qquad
		\text{ and }
		\qquad 
		\int_{\Tilde{I}} \mathbf{\mathcal{U}} \mathrm{d}x \approx \frac{\dx}{2} \sum_{j = 0}^k \omega_j \mathbf{\mathcal{U}}_j. 
		\label{int_1D}
	\end{equation}
	\begin{theorem}
		If a smooth numerical flux ${\mathbf{F}_{EC}^{x}} (\con_L,\con_R)$ is consistent, symmetric and entropy conservative in the sense of Definition~\ref{def:EC_ES}, then the scheme~\eqref{scheme_1D} is conservative, $k^{th}$ order accurate and entropy conservative in the discrete sense. 
		\label{thm:conservative}
	\end{theorem}
	\begin{proof}
		The proof of conservation and accuracy follows from \cite{Chen2017}. Here, we present the proof of entropy conservation.

		\paragraph{Entropy Conservation:} We show that the time rate of change of the discrete integral of the entropy variable in a domain $I^x$ is equal to zero.

		\begin{align*}
			\df{}{t}
			\left(
			\sum_{j = 0}^k \dfrac{\dx}{2} \omega_j \mathcal{U}_j
			\right)
			& =
			\sum_{j = 0}^k \dfrac{\dx}{2} \omega_j \evar_j^\top \frac{d\con_j}{dt} 
			\\
			& =
			-2 \sum_{j = 0}^k \sum_{m=0}^{k} S_{jm} \evar_j^\top{\mathbf{F}_{EC}^{x}} (\con_j,\con_m)
			+ \sum_{j = 0}^k {\tau_j} \evar_j^\top ({\mathbf{f}}^{x}_{j} - {\mathbf{f}}^{x}_{*,j} )
			\\
			& \qquad
			+	\sum_{j = 0}^k \dfrac{\dx}{2} \omega_j \evar_j^\top \sour_j 
			\\
			& =
			- \sum_{j = 0}^k \sum_{m=0}^{k} (S_{jm} + S_{jm}) \evar_j^\top{\mathbf{F}_{EC}^{x}} (\con_j,\con_m)
			+ \sum_{j = 0}^k {\tau_j} \evar_j^\top ({\mathbf{f}}^{x}_{j} - {\mathbf{f}}^{x}_{*,j} )\tag*{(Using $\evar_j^\top\sour_{j}=0$, remark~\ref{remark:ent_dot_source})}
			\\
			& =
			- \sum_{j = 0}^k \sum_{m=0}^{k} ( B_{jm} - S_{mj} + S_{jm}) \evar_j^\top{\mathbf{F}_{EC}^{x}} (\con_j,\con_m)
			\\
			& \qquad
			+ \sum_{j = 0}^k {\tau_j} \evar_j^\top ({\mathbf{f}}^{x}_{j} - {\mathbf{f}}^{x}_{*,j} )
			\tag*{(Using SBP property~\eqref{SBP})}
			\\
			\\
			& =
			- \sum_{j = 0}^k  \tau_j \evar_j^\top \mathbf{f}_j^x 
			- \sum_{j = 0}^k \sum_{m=0}^{k} S_{jm} (\evar_j - \evar_m)^\top {\mathbf{F}_{EC}^{x}} (\con_j,\con_m)
			\\
			& \qquad
			+ \sum_{j = 0}^k {\tau_j} \evar_j^\top ({\mathbf{f}}^{x}_{j} - {\mathbf{f}}^{x}_{*,j} )
			\\
			& =
			- \sum_{j = 0}^k \sum_{m=0}^{k} S_{jm} (\psi_j - \psi_m) 
			- \sum_{j = 0}^k {\tau_j} \evar_j^\top {\mathbf{f}}^{x}_{*,j} 
			\tag*{(Since $\mathbf{F}^x_{EC}$ is entropy conservative)}
			\\
			& =
			\sum_{j = 0}^k  \tau_j \psi_j 
			- \sum_{j = 0}^k {\tau_j} \evar_j^\top {\mathbf{f}}^{x}_{*,j} 
			\tag*{(Using Eqn.~\eqref{column_sum})}
			\\
			& =
			( \psi_k - \evar_k^\top {\mathbf{f}}^{x}_{*,k})
			-
			( \psi_0 - \evar_0^\top {\mathbf{f}}^{x}_{*,0}).
		\end{align*}
  Putting back the cell index $i$, we have shown that
\begin{align*}
  \frac{d}{dt}\int_{I^x_i} \mathcal{U} d x = & (\psi_{i+1/2}^- - (\evar_{i+1/2}^-)^\top \mathbf{F}^x_{EC}(\con_{i+1/2}^-, \con_{i+1/2}^+))
  \\
  & - (\psi_{i-1/2}^+ 
  - (\evar_{i-1/2}^+)^\top \mathbf{F}^x_{EC}(\con_{i-1/2}^-, \con_{i-1/2}^+))
\end{align*}
  If the flux $\mathbf{F}_{EC}$ satisfies the entropy conservation property \eqref{def_EC}, then we can rewrite the above equation as
  \[
  \frac{d}{dt}\int_{I^x_i} \mathcal{U} d x = -\mathcal{F}^x_{i+1/2} + \mathcal{F}^x_{i-1/2}
  \]
  where
  \[
  \mathcal{F}^x_{i+1/2} = \frac{1}{2}( \evar_{i+1/2}^- + \evar_{i+1/2}^+)^\top \mathbf{F}^x_{EC}(\con_{i+1/2}^-, \con_{i+1/2}^+) - \frac{1}{2}( \psi_{i+1/2}^- + \psi_{i+1/2}^+)
  \]
  is a consistent entropy flux. Hence, the scheme~\eqref{scheme_1D} is locally entropy conservative. The local entropy equations can be summed over all the elements to obtain global entropy conservation.  Therefore, the change in total entropy in the whole domain depends on the entropy fluxes at the boundaries. 
\end{proof}

	\subsubsection{Entropy conservative flux for the fluid components}
	To derive entropy conservative fluxes $\mathbf{F}^x_{i,EC}$ and $\mathbf{F}^x_{e,EC}$, satisfying Eqn.~\eqref{def_EC}, we follow \cite{bhoriya_high-order_2022,Bhoriya2020}. The expressions for {\em entropy variables} $\evar_i$ and $\evar_e$ are given by, 
	\begin{align*}
		\evar_\alpha=\begin{pmatrix}
			\dfrac{\gamma_\alpha - s_\alpha}{\gamma_\alpha - 1} +{\beta_\alpha} \\
			{v_{\alpha_x} \Gamma_\alpha \beta_\alpha }                        \\
			{v_{\alpha_y} \Gamma_\alpha \beta_\alpha }                        \\
			{v_{\alpha_z} \Gamma_\alpha \beta_\alpha }                        \\
			-{\Gamma_\alpha \beta_\alpha}
		\end{pmatrix}, \quad \beta_\alpha=\frac{\rho_\alpha}{p_\alpha}
		\qquad \text{ and } \qquad
		\psi_\alpha^x=\rho_\alpha \Gamma_\alpha v_{\alpha_x}.
	\end{align*}
	%
	Now we introduce the notations $[\![a]\!]$ and $\bar{a}$ for the jump and average of any variable $a$
	\begin{align*}
		[\![a]\!]_{\iph,j}=a_{i+1,j}-a_{i,j},  \qquad \bar{a}_{\iph,j}=\frac{1}{2}(a_{i+1,j}+a_{i,j}),
		\\
		[\![a]\!]_{i,\jph}=a_{i,j+1}-a_{i,j}, \qquad \bar{a}_{i,\jph}=\frac{1}{2}(a_{i,j+1}+a_{i,j}).
	\end{align*}
	For a strictly positive quantity $a$, we follow~\cite{Ismail2009} to introduce the logarithmic average of $a$ as
	$$
	a^{\ln}=\frac{[\![a]\!]}{[\![\log a]\!]}.
	$$ 
	This logarithmic average definition may not be numerically well-posed whenever the left and right states approach each other; a stable approximation given in~\cite{Ismail2009} will be used in the computations.
	
	
	For $\alpha \in \{i,e\}$, Eqn.~\eqref{def_EC} is a single algebraic equation in five unknown variables, 
	$\mathbf{F}^x_{\alpha,EC}$ = $[F^x_{\alpha,EC,1}$, $F^x_{\alpha,EC,2}$, $F^x_{\alpha,EC,3}$, $F^x_{\alpha,EC,4}$, $ F^x_{\alpha,EC,5}]^{\top}$. There can be many fluxes which satisfy this equation, but all of them will lead to entropy conservative schemes; in this work we use the expressions given in~\cite{bhoriya_high-order_2022}, which can be summarized as follows:

	\begin{align}
		\begin{rcases}
			F_{\alpha,EC,1}^x & = \rho_\alpha^{\ln} \avg{M}_{x_\alpha} 
			\\ 
			F_{\alpha,EC,2}^x & = \frac{1}{\avg{\beta_\alpha}} \left( \frac{\avg{\beta_\alpha} \avg{M}_{x_\alpha}}{\avg{\Gamma}_\alpha} F_{\alpha,EC,5}^x + \avg{\rho_\alpha} \right)  
			\\
			F_{\alpha,EC,3}^x & = \frac{\avg{M}_{y_\alpha}}{\avg{\Gamma}_\alpha} F_{\alpha,EC,5}^x
			\\
			F_{\alpha,EC,4}^x & = \frac{\avg{M}_{z_\alpha}}{\avg{\Gamma}_\alpha} F_{\alpha,EC,5}^x
			\\
			F_{\alpha,EC,5}^x & = \frac{
				-\avg{\Gamma}_\alpha{\big( L_{\beta_\alpha} {\rho_\alpha^\ln}\avg{M}_{x_\alpha} +\frac{\avg{M}_{x_\alpha}\avg{\rho_\alpha}}{\avg{\beta_\alpha}}}\big)
			} 
			{
				\big(\avg{M}_{x_\alpha}^2 + \avg{M}_{y_\alpha}^2 + \avg{M}_{z_\alpha}^2 -{(\avg{\Gamma}_\alpha})^2  \big)
			} 
		\end{rcases} \text{ for } \alpha \in \{i,e\}.
		\label{flux_fluid_EC_x}
	\end{align}
	Here, $L_{\beta_\alpha}=\left(\frac{1}{\gamma_\alpha-1} \frac{1}{\beta^{ln}_\alpha}+1\right)$, $M_{x_\alpha}=\Gamma_\alpha v_{x_\alpha}$, $M_{y_\alpha}=\Gamma_\alpha v_{y_\alpha}$, $M_{z_\alpha}=\Gamma_\alpha v_{z_\alpha}$ and $\beta_\alpha=\frac{\rho_\alpha}{p_\alpha}$. 
	
	\subsubsection{Entropy conservative flux for the electromagnetic components}
	The Maxwells' Eqns.~\eqref{max1}-\eqref{max4} have a linear flux. Therefore an entropy conservative flux for the Eqns.~\eqref{max1}-\eqref{max4} can be taken as
	\begin{equation}
		\mathbf{F}_{m,EC}^x(\con_L,\con_R)  = \dfrac{\mathbf{f}_m^x(\con_L) + \mathbf{f}_m^x(\con_R)}{2},   \label{flux_maxwell_EC_x}
	\end{equation}
        which is a central flux (\cite{Kumar2012}).

	\begin{theorem}
		From Theorem~\eqref{thm:conservative}, we observe that the scheme~\eqref{scheme_1D} with flux $\mathbf{F}_{EC}^x$ as in Eqns.~\eqref{flux_fluid_EC_x}-\eqref{flux_maxwell_EC_x} is entropy conservative. 
	\end{theorem}

	\subsubsection{Entropy stable Schemes}
	Once we have the entropy conservative scheme, we modify the element boundary fluxes using the following theorem to obtain the one-dimensional entropy stable schemes for the complete one-dimensional system~\eqref{conservedform_1d}.
	\begin{theorem}
		If the element interface flux $\mathbf{\hat f}_{i+1/2}$ is entropy stable in the sense of definition \ref{def:EC_ES}, then the scheme~\eqref{scheme_1D} is entropy stable.
	\end{theorem}
	\begin{proof}
		Assume compactly supported or periodic boundary conditions. The total entropy in the $i$'th cell changes according to
		\begin{align}
			\df{}{t}
			\left(
			\sum_{j = 0}^k \dfrac{\dx}{2} \omega_j \ent_j^i
			\right)\nonumber
			& =
			( \psi_k^i - {(\evar_k^i)}^\top \hat{\mathbf{f}}_{\iph})
			-
			( \psi_0^{i} - {(\evar_0^{i})}^\top \hat{\mathbf{f}}_{\imh})\nonumber
			\\
			& \le 0 \nonumber 
		\end{align}
		i.e., the scheme~\eqref{scheme_1D} with entropy stable fluxes at the interfaces is entropy stable. 
	\end{proof}
	
	    In \cite{Chen2017}, the Local Lax–Friedrichs (LLF) flux is proven to be entropy stable when the wave speed in the LLF flux is taken as a suitable approximation of the average wave speed. Therefore, we choose the Local Lax-Friedrichs fluxes at the interfaces to dissipate the total entropy. Collectively, we have the following theorem. 
	
	\begin{theorem}
		The scheme~\eqref{scheme_1D} with the interior fluxes $\mathbf{F}_{EC}^x$ as in Eqns.~\eqref{flux_fluid_EC_x},~\eqref{flux_maxwell_EC_x} and interface fluxes as local Lax-Friedrichs fluxes is entropy stable.
	\end{theorem}
	
	\subsection{Two-dimensional schemes}
	We now describe the two-dimensional numerical schemes. Consider a two-dimensional computational domain of the form $D_{xy}=I^x \times I^y$ with $I^x=(x_{min},x_{max})$ and $I^y=(y_{min},y_{max})$. For $M,N\in \mathbb{N}$, the domain $D_{xy}$ is discretized as follows
	\begin{align}
		\begin{rcases}
			x_{min} = x_{\half} < x_{\frac{3}{2}} < \dots < x_{M+\half} = x_{max}
			\\
			y_{min} = y_{\half} < y_{\frac{3}{2}} < \dots < y_{N+\half} = y_{max} 
		\end{rcases}
	\end{align}
	The computational subcells $\left[x_{i-\half}, x_{i+\half}\right] \times \left[y_{j-\half}, y_{j+\half}\right], i\in \{1,2,\dots,N_x\},j \in \{1,2,\dots,N_y\},$ are denoted by $I_{ij}$. The step size $\dx_i$ is given by $\dx_i = x_{i+\half}-x_{i-\half}$ and $\dy_j$ is given by $\dy_j = y_{j+\half}-y_{j-\half}$. We consider the $(k+1)$ Gauss-Lobatto-Legendre (GLL) points in both directions, with corresponding weights defined by $\{\omega_l\}_{l = 0}^k$. The change of variable formulas between the cell $I_{ij}$ and reference element $I = [-1,1]^2$ is given by
	\begin{equation}
		x = x_i(\xi) = \dfrac{x_{\imh} + x_{\iph}}{2} + \dfrac{\dx_i}{2} \xi \quad  \text{and} \quad
		y = y_j(\eta) = \dfrac{y_{\jmh} + y_{\jph}}{2} + \dfrac{\dy_j}{2} \eta.
	\end{equation}
	The solution inside the cell $I_{ij}$ is approximated by a tensor product of polynomials of the form
	\[
	\con_h(x_i(\xi), y_j(\eta)) = \sum_{q=0}^k \sum_{p=0}^k \con_{p,q} L_p(\xi) L_q(\eta), \qquad \con_{p,q} = \con_h(x_i(\xi_p), y_j(\xi_q))
	\]
    where $\con_{p,q}$ are the solution values at the GLL points, which depend on time $t$, but this is not explicitly indicated in the notation.  Accordingly, the discrete integrals of $\mathbf{U}$ and $\mathbf{\mathcal{U}}$ in the cell $I_{ij}$ are approximated as follows
	\begin{gather*}
		\int_{I_{ij}} \con  \mathrm{d}x \mathrm{d}y \approx 
		\frac{\dx_i \dy_j}{4} \sum_{q = 0}^k \sum_{p = 0}^k \omega_p \omega_q \con_{p,q}
		\intertext{and}
		\int_{I_{ij}} \ent \mathrm{d}x \mathrm{d}y \approx 
		\frac{\dx_i \dy_j}{4} \sum_{q = 0}^k \sum_{p = 0}^k\omega_p \omega_q \ent(\con_{p,q}).
	\end{gather*}
	We proceed similarly to the one-dimensional case, to obtain the two-dimensional entropy stable schemes for a single element. The semi-discrete scheme for the cell $I_{ij}$ is given by   
	\begin{align}
		\dd{{\con_{{p,q}}}}{t} 
		&=
		-\dfrac{2}{\dx_i} 
		\left(
		2 \sum_{m=0}^{k} D_{pm} {\mathbf{F}_{EC}^x} (\con_{p,q},\con_{m,q})
		-\frac{\tau_p}{\omega_p} ({\mathbf{f}}^{x}_{{p,q}} - {\mathbf{f}}^{x}_{*,{p,q}} )
		\right)\nonumber
		\\
		&
		\quad
		-\dfrac{2}{\dy_j} 
		\left(
		2 \sum_{m=0}^{k} D_{qm} {\mathbf{F}_{EC}^y} (\con_{p,q},\con_{p,m})
		-\frac{\tau_q}{\omega_q} ({\mathbf{f}}^{y}_{{p,q}} - {\mathbf{f}}^{y}_{*,{p,q}} )
		\right)
		\nonumber
		\\
		&
		\quad
		+
		{\sour_{{p,q}}}, \qquad  \ 0 \le {p,q} \le k. 
		\label{scheme_2D}
	\end{align}
	Here, we have used the following notations 
	\begin{align}
		\begin{rcases}
			\mathbf{f}^x_{p,q} = \mathbf{f}^x(\con_{p,q})
			, \
			\mathbf{f}^y_{p,q} = \mathbf{f}^y(\con_{p,q})
			, \
			\sour_{p,q} = \sour(\con_{p,q})
			\\
			\mathbf{f}^x_{*,p,q} = \begin{cases} 
				\mathbf{f}^x_{\imh,q} = \mathbf{f}^x(\con(x^-_\imh,y_j(\xi_q)),\con(x^+_\imh,y_j(\xi_q))) & \text{ if } p = 0
				\\
				0 & \text{ if } 0 < p < k
				\\
				\mathbf{f}^x_{\iph,q} = \mathbf{f}^x(\con(x^-_\iph,y_j(\xi_q)),\con(x^+_\iph,y_j(\xi_q))) & \text{ if } p = k 
			\end{cases} 
			\\
			\mathbf{f}^y_{*,p,q} = \begin{cases} 
				\mathbf{f}^y_{p,\jmh} = \mathbf{f}^y(\con(x_i(\xi_p),y^-_\imh),\con(x_i(\xi_p),y^+_\imh)) & \text{ if } q = 0
				\\
				0 & \text{ if } 0 < q < k
				\\
				\mathbf{f}^y_{p,\jph} = \mathbf{f}^y(\con(x_i(\xi_p),y^-_\iph),\con(x_i(\xi_p),y^+_\iph)) & \text{ if } q = k
			\end{cases} 
		\end{rcases}.
		\label{notations_2D}
	\end{align}
	The $x-$directional entropy conservative flux $\mathbf{F}^x_{EC} = (\mathbf{F}^x_{i,EC},\mathbf{F}^x_{e,EC},\mathbf{F}^x_{m,EC})^\top$ is given by Eqn.~\eqref{flux_fluid_EC_x}, \eqref{flux_maxwell_EC_x}. Following~\cite{bhoriya_high-order_2022}, the $y-$directional entropy conservative flux $\mathbf{F}^y_{EC} = (\mathbf{F}^y_{i,EC},\mathbf{F}^y_{e,EC},\mathbf{F}^y_{m,EC})^\top$ is given by 
	\begin{align}
		\begin{rcases}
			F_{\alpha,EC,1}^y & = \rho_\alpha^{\ln} \avg{M}_{y_\alpha} 
			\\ 
			F_{\alpha,EC,2}^y & = \frac{\avg{M}_{x_\alpha}}{\avg{\Gamma}_\alpha} F_{\alpha,EC,5}^y
			\\
			F_{\alpha,EC,3}^y & = \frac{1}{\avg{\beta}_\alpha} \left( \frac{\avg{\beta}_\alpha \avg{M}_{y_\alpha}}{\avg{\Gamma}_\alpha} F_{\alpha,EC,5}^y + \avg{\rho}_\alpha \right)  
			\\
			F_{\alpha,EC,4}^y & = \frac{\avg{M}_{z_\alpha}}{\avg{\Gamma}_\alpha} F_{\alpha,EC,5}^y
			\\
			F_{\alpha,EC,5}^y & = \frac{
				-\avg{\Gamma}_\alpha{\big( L_{\beta_\alpha} {\rho}_\alpha^{\ln}\avg{M}_{y_\alpha} +\frac{\avg{M}_{y_\alpha}\avg{\rho}_\alpha}{\avg{\beta}_\alpha}}\big)
			} 
			{
				\big(\avg{M}_{x_\alpha}^2 + \avg{M}_{y_\alpha}^2 + \avg{M}_{z_\alpha}^2 -{(\avg{\Gamma}_\alpha)}^2  \big)
			} 
		\end{rcases} \text{ for } \alpha \in \{i,e\},
		\label{flux_fluid_EC_y}
	\end{align}
	and
	\begin{align}
		\hspace{-2cm}\mathbf{F}_{m,EC}^y(\con_L,\con_R) = \dfrac{\mathbf{f}_m^y(\con_L) + \mathbf{f}_m^y(\con_R)}{2}.   \label{flux_maxwell_EC_y}
	\end{align}
	We can now state the following result:
	\begin{theorem}
		The two-dimensional scheme defined by~\eqref{scheme_2D} with the conservative flux as in Eqns.~\eqref{flux_fluid_EC_x}, \eqref{flux_maxwell_EC_x},\eqref{flux_fluid_EC_y},\eqref{flux_maxwell_EC_y} and interface fluxes as the Local Lax-Friedrichs fluxes is entropy stable.
	\end{theorem}
	\begin{proof}
		The proof is similar to the proof of one-dimensional case.
	\end{proof}
	
	\section{Time discretization} \label{sec:full_disc}
	
	We use explicit strong stability preserving Runge-Kutta (SSP-RK) methods for the time-update of ODEs given in Eqns.~\eqref{scheme_1D} and \eqref{scheme_2D}. The analytical details of the presented time schemes can be found in~\cite{Gottlieb2001}. For clarity of notations, we ignore the nodal subscripts $p$ and $q$. Consequently, the update Eqns.~\eqref{scheme_1D},\eqref{scheme_2D} can be written in a compact form as 
	\begin{equation}
		\frac{d \mathbf{U}(t) }{dt}
		=
		\mathcal{L}(\mathbf{U}(t)), 
		\label{fullydiscrete}
	\end{equation}
	where $\mathcal{L}(\mathbf{U}(t))$ denotes the right hand side (RHS) of the scheme given by Eqn.~\eqref{scheme_1D} or Eqn.~\eqref{scheme_2D}. The initial solution at time $t = t^n$ is denoted by $\mathbf{U}^n$, the time step is $\Delta t = t^{n+1} - t^n$ and the updated solution at time $t=t^{n+1}$ is denoted by $\mathbf{U}^{n+1}$. The second and third-order accurate SSP-RK schemes to get the updated solution $\con^{n+1}$ are given by,
	\begin{enumerate}
		\item Let $\con^{(0)} \ = \ \con^n$.
		\item Let $r=2, 3$ for the second and third order schemes, respectively. For $i=1,\dots,r$, compute
		\begin{equation}
			\con^{(i)} \
			= \
			\sum_{m=0}^{i-1} \left[ \alpha_{im}\con^{(m)}
			+
			\beta_{im}\Delta t \mathcal{L}(\con^{(m)}) \right]
			\label{ssprk}
		\end{equation}
		where, the coefficients $\alpha_{im}$ and $\beta_{im}$ are given in Table~\eqref{table:ssp}.
		\item The updated solution is given by $\con^{n+1} \ =  \ \con^{(r)}$.
	\end{enumerate}

	\begin{table}[h]
		\caption[h]{Coefficients for Runge-Kutta time stepping:}
		\centering
		\begin{tabular}{ l|ccc|ccc }
			\hline
			Order              & \multicolumn{3}{|c|}{$\alpha_{im}$} & \multicolumn{3}{|c}{$\beta_{im}$}                       \\
			\hline
			\multirow{2}{*}{2} & 1     &     &     & 1 &     &     \\
			& 1/2   & 1/2 &     & 0 & 1/2 &     \\
			\hline
			\multirow{3}{*}{3} & 1     &     &     & 1 &     &     \\
			& 3/4   & 1/4 &     & 0 & 1/4 &     \\
			& 1/3   & 0   & 2/3 & 0 & 0   & 2/3 \\
			\hline
		\end{tabular}
		\label{table:ssp}
	\end{table}
	To achieve the fourth-order accuracy in the time direction, we use the fourth-order accurate RK-SSP from~\cite{spiteri2002new} which is given by
	\begin{subequations}
		\begin{align*}
			\mathbf{U}^{(1)} &= \textbf{U}^n + 0.39175222700392 \;(\Delta t) \mathcal{L}(\mathbf{U}^n), \\
			\mathbf{U}^{(2)} &= 0.44437049406734 \;\textbf{U}^n + 0.55562950593266\; \mathbf{U}^{(1)} 
			    \\ 
				&\quad 
				+0.36841059262959\;  (\Delta t) \mathcal{L}(\mathbf{U}^{(1)}), \\
			\mathbf{U}^{(3)} &= 0.62010185138540\; \textbf{U}^n + 0.37989814861460\; \mathbf{U}^{(2)}
			    \\ 
				&\quad 
			    +0.25189177424738\;  (\Delta t) \mathcal{L}(\mathbf{U}^{(2)}), \\
			\mathbf{U}^{(4)} &= 0.17807995410773\; \textbf{U}^n + 0.82192004589227\; \mathbf{U}^{(3)}
			    \\ 
				&\quad 
			    + 0.54497475021237\;  (\Delta t)  \mathcal{L}(\mathbf{U}^{(3)}), \\
			\textbf{U}^{n+1} &= 0.00683325884039 \;\textbf{U}^n + 0.51723167208978\; \mathbf{U}^{(2)}
			    \\ 
				&\quad 
			    + 0.12759831133288 \;\mathbf{U}^{(3)}
			    + 0.34833675773694\; \mathbf{U}^{(4)}
			    \\ 
				&\quad 
				+ 0.08460416338212 \; (\Delta t) \mathcal{L}(\mathbf{U}^{(3)})+  0.22600748319395 \; (\Delta t) \mathcal{L}(\mathbf{U}^{(4)}).
		\end{align*}
	\end{subequations}
	To obtain the second-order accurate scheme, we take $k = 1$ and integrate it with the second-order SSP-RK scheme; we denote the scheme by $\second$. For the third-order accurate scheme, we choose $k=2$ and use the third-order explicit time stepping; the resultant scheme is denoted by $\third$. Similarly, for the fourth-order scheme, we consider $k = 3$ and use the fourth-order SSP Runge-Kutta scheme; we denote the obtained scheme by $\fourth$.

	\section{Numerical results} \label{sec:num_test}
	
	This section presents several numerical tests for the two-fluid relativistic plasma flows in one and two dimensions. We use the notations $N_x$ for the number of cells in the $x$-direction and $N_y$ for the number of cells in the $y$-direction. The specific heat constants for both the species (ion and electron) are assumed to be equal, i.e., $\gamma_i=\gamma_e=\gamma$ with $\gamma=4/3$, unless specified otherwise. In all the test cases, we set initial potentials $\phi$ and $\psi$ to be 0. The potential speeds are set to $\chi=\kappa=1$. For the one-dimensional ($x$-directional) test cases, the time step $\Delta t$ is determined as,
	\begin{gather}
		{\Delta t} =  \text{CFL} \cdot \min \left\{\dfrac{\Delta^i x}{\Lambda_{max}^x(\mathbf{U}_i)} : 1 \le i \le N_x\right\}, 
		\label{dt_1d}
	\end{gather}
	where,
	\begin{gather}
		\Delta^i x = (\xi_1 - \xi_0) \Delta x_i \quad\text{ and }\quad
		\Lambda_{max}^x(\mathbf{U}_i) = \max\{|\Lambda^x_l(\mathbf{U}_i^p)|: 0 \le p \le k, 1 \le l \le 18\}.
		\nonumber
	\end{gather}
	The annotations $\Delta^i x$ and $\Lambda_{max}^x(\mathbf{U}_i)$ are analogs of the local step size and maximum local speed, respectively. For the one-dimensional test cases, we use the phrase \textit{zone size} for the uniform step size $\Delta x = (x_{\iph} - x_\imh)$. In the two-dimensional case, the CFL formula in Eqn.~\eqref{dt_1d} generalized as follows
	\begin{gather*}
		{\Delta t} = \text{CFL} \cdot \min \left\{ \left(\dfrac{\Delta^{i} x}{\Lambda_{max}^x(\mathbf{U}_{i,j})} + \dfrac{\Delta^{j} y}{\Lambda_{max}^y(\mathbf{U}_{i,j})}\right) : 1 \le i \le N_x, \ 1 \le j \le N_y \right\}
	\end{gather*}
	where,
	$$
	\Delta^{j} y =(\xi_1 - \xi_0) \Delta y_j,,
	$$
	$$
	\Lambda_{max}^x(\mathbf{U}_{i,j}) = \max \{|\Lambda^x_l(\mathbf{U}^{p,q}_{i,j})| : 0 \le p,q \le k, 1 \le l \le 18 \},
	$$ 
	and 
	$$\Lambda_{max}^y(\mathbf{U}_{i,j}) = \max\{|\Lambda^y_l(\mathbf{U}^{p,q}_{i,j})|: 0 \le p,q \le k, 1 \le l \le 18\}.
	$$
	We choose CFL=0.15 for all the test problems, which leads to stable computations for all the schemes used in this work. \reva{For all the test cases, the parameters used and the grid sizes are such that the time step is restricted by the hyperbolic CFL condition and not by the source terms. Hence explicit time stepping schemes are used in the current work.}
 
	In some test cases, we need to consider resistivity effects; the momentum densities~(Eqns.~\eqref{momentum}) and the energy densities~(Eqns.~\eqref{energy}) are modified to include resistive terms. We follow \cite{Amano2016,Balsara2016} to write the modified equations as
	\begin{subequations}
		\begin{align}
			\frac{\partial \bm{M}_\alpha}{\partial t} + \nabla \cdot (\bm{M}_\alpha \bm{v}_\alpha  + p_\alpha \bm{I}) & = r_\alpha \Gamma_\alpha \rho_\alpha (\bm{E}+\bm{v}_\alpha \times \bm{B}) + \bm{R}_\alpha  , \label{resist_momentum} \\
			\frac{\partial \ener_\alpha}{\partial t} + \nabla \cdot ((\ener_\alpha+p_\alpha)\bm{v}_\alpha)    & = r_\alpha \Gamma_\alpha \rho_\alpha (\bm{v}_\alpha \cdot \bm{E}) + R_\alpha^0,	\label{resist_energy}
		\end{align}
	\end{subequations}
	where $\alpha \in \{i,e\}$. Following~\cite{Amano2016}, an anti-symmetric relationship, $(\bm{R}_e, {R}_e^0)=(-\bm{R}_i, -R_i^0)$, has been assumed for the conservation of the total momentum and energy density. The expressions for the terms $\bm{R}_i$ and $R_i^0$ are given by
	\begin{align*}
		& \bm{R}_i =  -\eta \dfrac{\omega_p^2}{r_i-r_e} (\bm{j}- \rho_0 \bm{\Phi}), \qquad
		{R}_i^0 =  -\eta \dfrac{\omega_p^2}{r_i-r_e} (\rho_c- \rho_0 {\Lambda}),
		\\
		& \omega_p^2 = r_i^2 \rho_i + r_e^2 \rho_e, \qquad
		\bm{\Phi} = \dfrac{r_i \rho_i \Gamma_i \vel_i + r_e \rho_e \Gamma_e \vel_e}{\omega_p^2},
		\\
		& \Lambda = \dfrac{r_i^2 \rho_i \Gamma_i  + r_e^2 \rho_e \Gamma_e }{\omega_p^2}, \qquad
		\rho_0 = \Lambda \rho_c - \bm{j} \cdot \bm{\Phi},
	\end{align*}
	and $\eta$ is the finite resistivity constant. In the above, $\omega_p$ denotes the total plasma frequency. The update \eqref{ssprk} is modified to include these terms consistently.
	In addition, We scale Maxwell's source by a factor of $4\pi$ for the test problems taken from~\cite{Balsara2016}, to have consistent results. 
    \begin{remark}
High order schemes generate numerical oscillations around discontinuities, and some form of non-linear limiting is required to control these oscillations. For the one-dimensional test cases, we use the slope limiter given in~\cite{Hesthaven2008}, and for the two-dimensional case, we use the methodology of~\cite{Gallego-Valencia2016} to define the two-dimensional slope limiter. We use a bound preserving limiter of~\cite{Qin2016} to keep the solution in the physical domain.
     \end{remark}	
 
	\subsection{One-dimensional test cases}
	In this section, we present several one-dimensional test cases. We compare the accuracy and robustness of the second ($\second$), third  ($\third$), and fourth-order ($\fourth$) accurate schemes.
	\subsubsection{Order of accuracy test} \label{test:1d_smooth}
 A problem with a known exact solution can be constructed using the forced solution approach from~\cite{Kumar2012},  by introducing a forcing source term to generalize the smooth test case given in~\cite{Bhoriya2020}. The resulting test problem can also be found in~\cite{bhoriya_high-order_2022}. We append extra source terms (denoted by $\bm{\mathcal{R}}(x,t)$) to the system and consider the updated system,
	\begin{equation*}
		\frac{\partial \mathbf{U}}{\partial t}+\frac{\partial \mathbf{f}^x}{\partial x} = \sour + \bm{\mathcal{R}}(x,t)
	\end{equation*}
	with
	\begin{align*}
		\bm{\mathcal{R}}(x,t)=\Big(\mathbf{0}_{13},
		   -\dfrac{1}{\sqrt{3}}(2+\sin(2 \pi (x-0.5t))),
		    0,
			-3\pi\cos(2 \pi (x-0.5t)),
			\\
			\dfrac{2 \chi}{\sqrt{3}}  (2+\sin(2 \pi (x-0.5t))) ,
			0 \Big)^\top.
	\end{align*}
	The exact solution is given by fluid densities $\rho_i = \rho_e = 2+\sin(2 \pi (x-0.5t))$, where we take ion and electron velocities as $v_{x_i}=v_{x_e} = 0.5$. The initial pressures are set as $p_i=p_e=1$. The magnetic field $\bm{B}$ has non-zero $y$-component given by $B_y=2 \sin(2 \pi (x-0.5t))$ and the electric field vector $\bm{E}$ has non-zero $z$-component given by $E_z = -\sin(2\pi (x-0.5t))$. All remaining variables are set to zero. The computations are performed on the one-dimensional domain $I=[0,1]$ with the periodic boundary conditions. We take unequal mass for the fluids to have a non-trivial source in the electric-field equation; specifically,  we choose $r_i = 1$ and $r_e = -2$. The adiabatic index is $\gamma=5/3$ and we integrate till time $t=2$ using the schemes $\second$, $\third$ and $\fourth$.
	
	\begin{table}[h]
		\centering
		\begin{tabular}{|c|c|c|c|c|c|c|}
			\hline
			$N_x$ & \multicolumn{2}{|c}{$\second$ } &
			\multicolumn{2}{|c}{$\third$} &
			\multicolumn{2}{|c|}{$\fourth$} \\
			\hline
			-- & $L^1$ error & Order  & $L^1$ error & Order  & $L^1$ error & Order   \\
			\hline
			16 & 1.74431e-01 &   --   & 3.75871e-03 &   --   & 6.22973e-05 &   --   \\
			32 & 4.32234e-02 & 2.0128 & 4.96314e-04 & 2.9209 & 3.53976e-06 & 4.1374 \\
			64 & 1.08096e-02 & 1.9995 & 6.26802e-05 & 2.9852 & 2.16246e-07 & 4.0329 \\
			128 & 2.69730e-03 & 2.0027 & 7.83468e-06 & 3.0001 & 1.37147e-08 & 3.9789 \\
			256 & 6.73101e-04 & 2.0026 & 9.78976e-07 & 3.0005 & 8.33461e-10 & 4.0405 \\
			512 & 1.68113e-04 & 2.0014 & 1.22665e-07 & 2.9965 & 6.11604e-11 & 3.7684 \\
			\hline
		\end{tabular}
		\caption{\nameref{test:1d_smooth}: $L^1$ errors and order of convergence for the $\second$, $\third$ and $\fourth$ schemes.}
		\label{table:order}
	\end{table}
	\begin{figure}[h]
		\begin{center}
			\subfigure{
				\includegraphics[width=3.4in, height=2.7in]{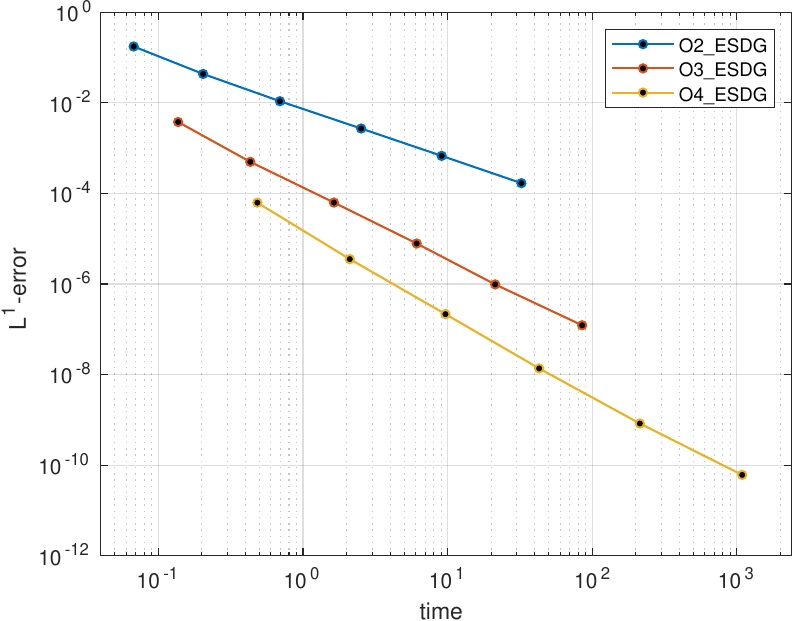}
			}
			\caption{\nameref{test:1d_smooth}: Plot of $L^1$-error with respect to CPU time. }
			\label{fig:cputime}
		\end{center}
	\end{figure}
	
	The resultant $ L^1$ errors of ion density and the order of convergence are presented in Table~\ref{table:order}. We observe that error norm of all schemes converge at the theoretical order of accuracy, $O(\Delta x^{k+1})$.  We have also presented ion density $L^1$-error versus CPU time for each scheme in Figure~\ref{fig:cputime}. We also notice that for a given error threshold, the $\fourth$ scheme is most economical, followed by the schemes $\third$ and $\second$.

	\reva{
	\subsubsection{Circularly polarized superluminal waves} \label{test:CP_smooth}
    In~\cite{Barkov2014}, a nonlinear harmonic solution for the two-fluid relativistic plasma system has been provided. We consider the same problem here which involves polarized superluminal waves in one dimension. The availability of the analytical solution allows us to perform an accuracy study of the schemes for the two-fluid system. The background solution consists of uniform and equal densities and pressures for electrons and positrons. Furthermore, the background state has vanishing velocities and electromagnetic fields. These vector quantities are perturbed by a small oscillation in such a way that the densities, pressures, and velocity vectors along the x-direction are not changed. The exact solution at time $t$ is given by 
	\begin{align*}
		(\rho_i, v_{x_i},v_{y_i}, v_{z_i},p_i )
		&=
		(1,0,V_0 \cos(kx-\omega t), -V_0 \sin(kx - \omega t), 1/4)
		\\
		(\rho_e, v_{x_e},v_{y_e}, v_{z_e},p_e )
		&=
		(\rho_i,-v_{x_i}, -v_{y_i}, -v_{z_i}, p_i)
		\\
		(B_x,B_y,B_z) &= (0,B_0 \cos(kx-\omega t), - B_0 \sin(kx-\omega t))
		\\
		(E_x,E_y,E_z) &= (0, E_0 \sin(kx-\omega t), E_0 \cos(kx-\omega t)).
	\end{align*}
	In the above, we set $k=2\pi$, $\omega=2\pi\sqrt{2} $, $V_0 = -1/\sqrt{5}$, $B_0=1$ and $E_0 = -\omega/k$. We note that the wave phase speed is $\sqrt{2}$. The specific heat ratio is set to $\gamma=4/3$. The ion charge to mass ratio is given by $r_i = 2\pi$ and the electron charge to mass ratio is given by $r_e=-r_i$.  The computations are performed on the one-dimensional domain $I=[0,1]$ with periodic boundary conditions. We integrate till time $t=1/\sqrt{2}$ using the schemes $\second$, $\third$ and $\fourth$.
	
	\begin{table}[h]
		\centering
		\begin{tabular}{|c|c|c|c|c|c|c|}
			\hline
			$N_x$ & \multicolumn{2}{|c}{$\second$ } &
			\multicolumn{2}{|c}{$\third$} &
			\multicolumn{2}{|c|}{$\fourth$} \\
			\hline
			-- & $L^1$ error & Order  & $L^1$ error & Order  & $L^1$ error & Order   \\
			\hline
			16  & 9.961068e-02 &    --  & 3.024582e-03 &   --   & 1.001248e-04 &   --   \\
			32  & 2.723106e-02 & 1.8710 & 3.719219e-04 & 3.0236 & 6.232853e-06 & 4.0057 \\
			64  & 7.005851e-03 & 1.9586 & 4.644056e-05 & 3.0015 & 3.887373e-07 & 4.0030 \\
			128 & 1.768490e-03 & 1.9860 & 5.799790e-06 & 3.0013 & 2.443314e-08 & 3.9918 \\
			256 & 4.435591e-04 & 1.9953 & 7.250341e-07 & 2.9998 & 1.572122e-09 & 3.9580 \\
			\hline
		\end{tabular}
		\caption{\nameref{test:CP_smooth}: $L^1$ errors and order of convergence for the $B_y$ variable using the $\second$, $\third$ and $\fourth$ schemes.}
		\label{table:CP_By_order}
	\end{table}
	\begin{table}[h]
		\centering
		\begin{tabular}{|c|c|c|c|c|c|c|}
			\hline
			$N_x$ & \multicolumn{2}{|c}{$\second$ } &
			\multicolumn{2}{|c}{$\third$} &
			\multicolumn{2}{|c|}{$\fourth$} \\
			\hline
			-- & $L^1$ error & Order  & $L^1$ error & Order  & $L^1$ error & Order   \\
			\hline
            16  & 1.168646e-01 &   --   & 3.094881e-03 &    --  & 1.360579e-04 &    --  \\
            32  & 3.240937e-02 & 1.8503 & 3.810366e-04 & 3.0218 & 8.499166e-06 & 4.0007 \\
            64  & 8.389545e-03 & 1.9497 & 4.725174e-05 & 3.0114 & 5.306491e-07 & 4.0014 \\
            128 & 2.121465e-03 & 1.9835 & 5.881873e-06 & 3.0060 & 3.364629e-08 & 3.9792 \\
            256 & 5.324741e-04 & 1.9942 & 7.339285e-07 & 3.0025 & 2.123211e-09 & 3.9861 \\
			\hline
		\end{tabular}
		\caption{\nameref{test:CP_smooth}: $L^1$ errors and order of convergence for the $E_y$ variable using the $\second$, $\third$ and $\fourth$ schemes.}
		\label{table:CP_Ey_order}
	\end{table}
	The resultant $L^1$ errors and convergence rates for the $B_y$ and $E_y$ are given in Table~\ref{table:CP_By_order} and \ref{table:CP_Ey_order}, respectively.  We observe that errors for all the schemes converge at the theoretical rate.}
	
	\subsubsection{Relativistic Brio-Wu test problem with finite plasma skin depth} \label{test:1d_brio}
	This is a Riemann problem whose initial data is as given in~\cite{Balsara2016} and we also scale the Maxwell's source terms as in this reference. The computational domain is $[-0.5,0.5]$ with Neumann boundary conditions. The initial discontinuity at $x=0$ separates the initial states,
	\begin{align*}
		\mathbf{W_L}=\begin{pmatrix}
			\rho_i \\ p_i \\ \rho_e \\ p_e \\ B_x \\ B_y
		\end{pmatrix}_{L}
		=
		\begin{pmatrix}
			0.5 \\ 0.5 \\ 0.5 \\ 0.5 \\ \sqrt{\pi} \\ \sqrt{4 \pi}
		\end{pmatrix},
		\qquad
		\mathbf{W_R}=\begin{pmatrix}
			\rho_i \\ p_i \\ \rho_e \\ p_e \\ B_x \\ B_y
		\end{pmatrix}_{R}
		=
		\begin{pmatrix}
			0.0625 \\ 0.05 \\ 0.0625 \\ 0.05 \\ \sqrt{\pi} \\ -\sqrt{4 \pi}
		\end{pmatrix}.
	\end{align*}
	All the remaining variables are set to zero. The charge to mass ratios has been taken as $r_i=-r_e=10^3/\sqrt{4 \pi}$; correspondingly, the plasma skin depth $d_p$ is given by $d_p=10^{-3}/\sqrt{\rho_i}$. We expect the solution to be close to the relativistic Magnetohydrodynamics solution whenever the zone size $\Delta x$ is larger than the plasma skin depth $d_p$. We take the specific heat ratio as $\gamma = 2$. The computations are performed till time $t = 0.4$ with $400$, $1600$ and $6400$ cells, and using $\second$, $\third$ and $\fourth$ schemes.
	
	To compare with the results of~\cite{Balsara2016}, in Figure~\ref{fig:brio_rho_o2}, we first plot the total-density ($\rho_i+\rho_e$) with $400$, $1600$ and $6400$ cells using the second-order scheme $\second$. To compare the solution with the RMHD solution, we have computed the ``RMHD" solution using $r_i=-r_e=10^4/\sqrt{4 \pi}$ with $1600$ cells using $\second$ scheme. For the zone size of $\Delta x = 0.0025$ (using $400$ cells), the plasma skin depth remains dominant in the entire domain; hence the solution is close to the RMHD solution. At the finer grids, $1600$ and $6400$, the zone size is smaller than the plasma skin depth $d_p$; hence, we observe the two-fluid dispersive effects in the waves. The results for the different schemes, at 400 cells, are compared in Figure~\ref{fig:brio_rho_all}.  We observe that on the same mesh, higher-order schemes capture more dispersive effects than the second-order scheme $\second$. The evolution of the total entropy in the domain is shown in Figure~\ref{fig:brio_ent_tot}; we observe that the second-order scheme is the most entropy dissipative, followed by the third and fourth-order schemes. \reva{In Figure~\ref{fig:brio_hireso}, we have plotted the total density for the fourth order scheme on highly resolved meshes of $50000$ and $100000$ cells and compared them with the total density for the second order scheme on $6400$ cells. A zoomed profile around point $x=-0.1$ is also plotted in Figure~\ref{fig:brio_hireso_zoomed}. We observe that the oscillations shown in Figure~\ref{fig:brio_rho_o2} are consistent with the high-resolution results of the fourth order scheme in Figure~\ref{fig:brio_hireso}.}
	
	\begin{figure}[!htbp]
		\begin{center}
			\subfigure[Plot of total density ($\rho_i+\rho_e$) using $\second$ scheme]{
				\includegraphics[width=2.2in, height=1.6in]{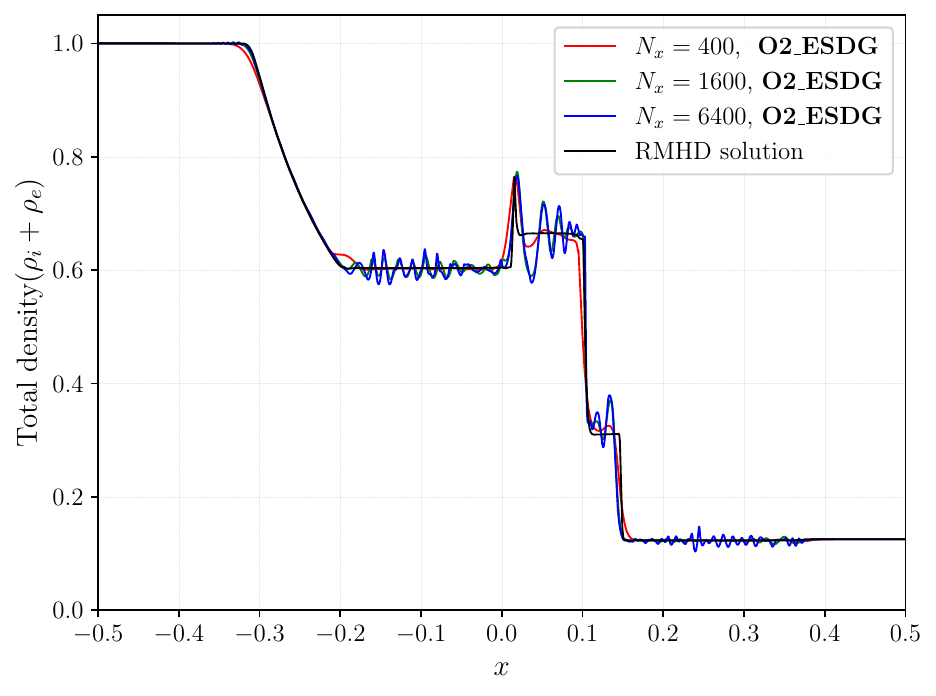}
				\label{fig:brio_rho_o2}}
			\subfigure[Plot of total density ($\rho_i+\rho_e$) with $N_x$=400 using $\second$, $\third$ and $\fourth$ schemes]{
				\includegraphics[width=2.2in, height=1.6in]{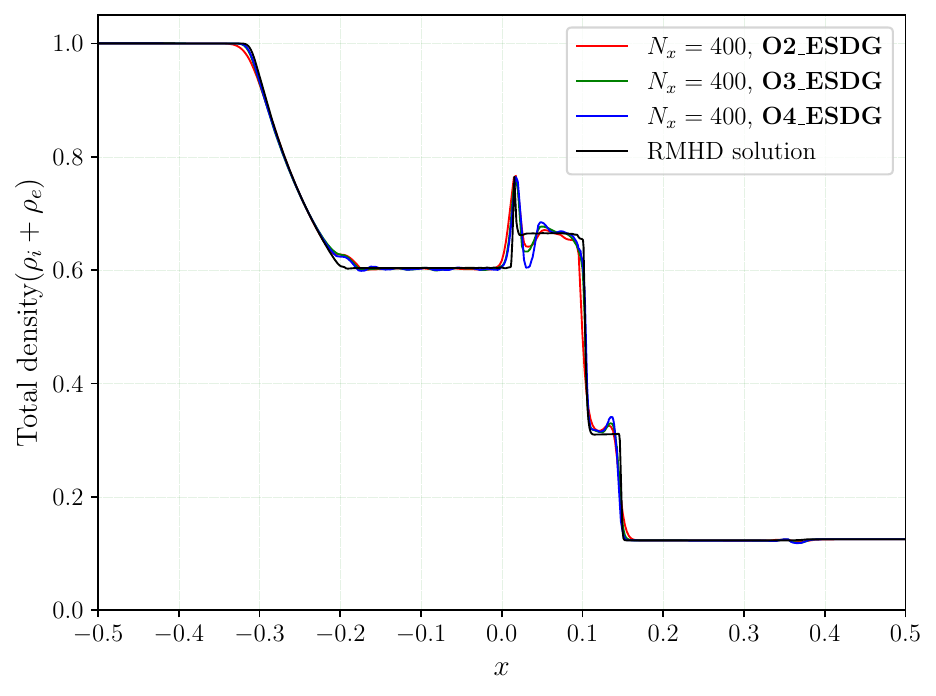}
				\label{fig:brio_rho_all}}
			\subfigure[Evolution of total entropy with $N_x$=400 cells for $\second$, $\third$ and $\fourth$ schemes]{
				\includegraphics[width=2.2in, height=1.6in]{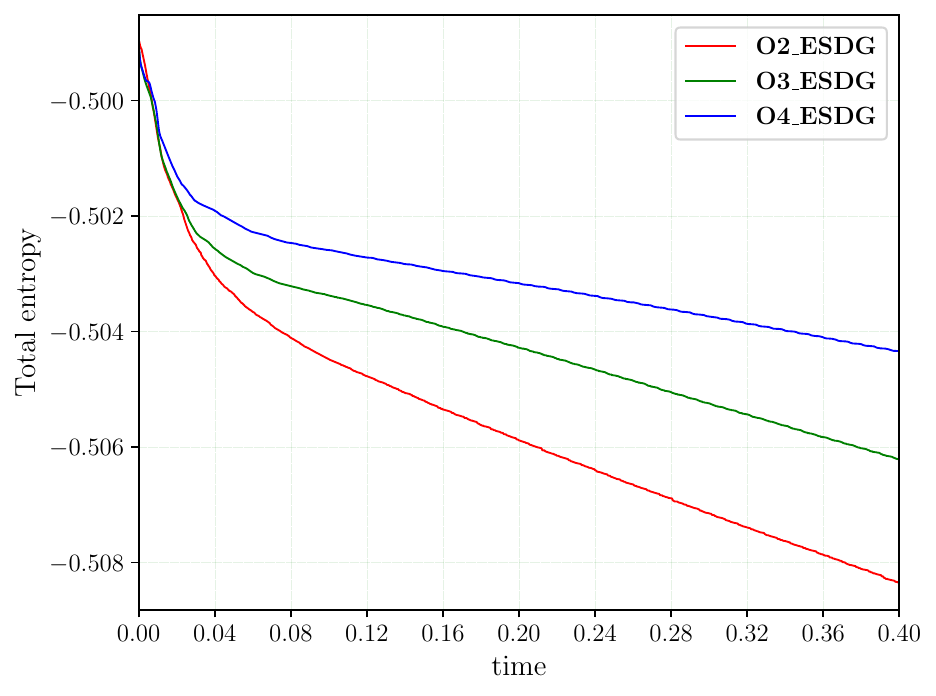}
				\label{fig:brio_ent_tot}}
            \reva{
			\subfigure[Plots of total density ($\rho_i+\rho_e$) at high resolutions]{
				\includegraphics[width=2.2in, height=1.6in]{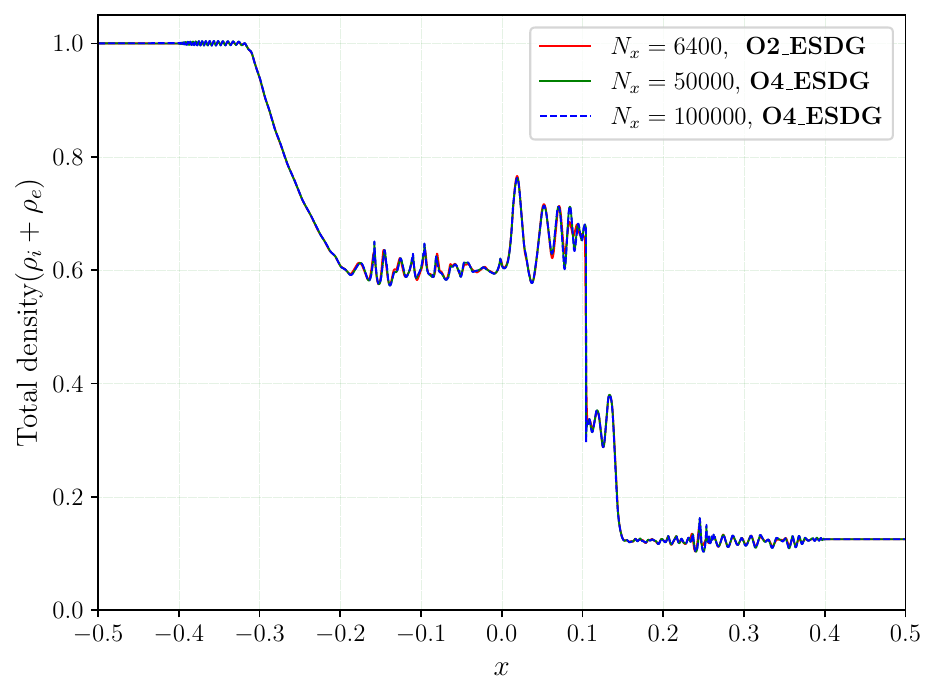}
				\label{fig:brio_hireso}}
			\subfigure[Plots of total density ($\rho_i+\rho_e$) at high resolutions (Zoomed view)]{
				\includegraphics[width=2.2in, height=1.6in]{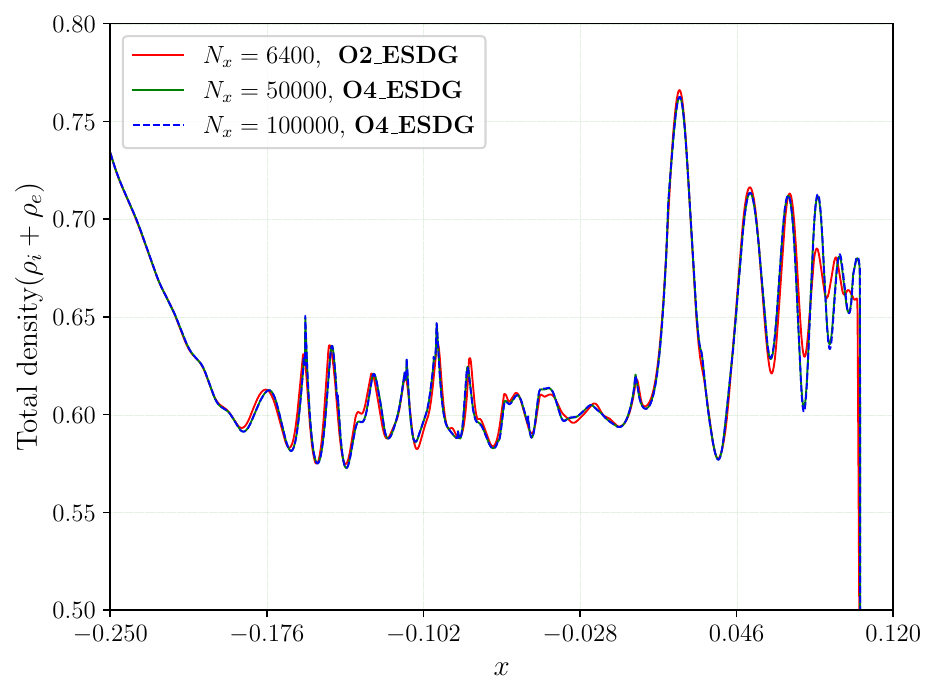}
				\label{fig:brio_hireso_zoomed}}
			}
   \caption{\nameref{test:1d_brio}: Plots for the total density and evolution of total entropy for the $\second$, $\third$ and $\fourth$ schemes.}
			\label{fig:brio}
		\end{center}
	\end{figure}
	
	\subsubsection{Self-similar current sheet with finite resistivity} \label{test:1d_sheet}
	The test case is a two-fluid extension of the relativistic magneto-hydrodynamics (RMHD) test considered in~\cite{Komissarov2007}; the two-fluid extension is described in~\cite{Amano2016}. Initially, only the $y$-component of the magnetic field has non-zero variation and the evolution of $B_y$ is governed by the diffusion equation 
	\begin{equation} \label{current_diff}
		\dfrac{\partial B_y}{\partial t} - D \dfrac{\partial^2 B_y}{\partial x^2} =0.
	\end{equation}
	The diffusion coefficient $D$ is related to the resistivity $\eta$ by the equation $D=\eta c^2$. In the resistive RMHD regime~\cite{Komissarov2007}, the expression for an exact self-similar solution of the partial differential equation is given by
	\begin{equation*}
		B_y(x,t) = B_0 \textit{  erf}\left(\dfrac{x}{2 \sqrt{Dt}} \right),
	\end{equation*}
	where \textit{erf} is the error function. The initial profile for the $B_y$-component is obtained by setting the time $t=1$, $B_0=1$, and $\eta =0.01$. The computational domain is $[-1.5,1.5]$ with Neumann boundary conditions. The charge-to-mass ratios are taken as $r_i=-r_e=10^3$. The ion and electron densities are equal to 0.5, and the ion and electron pressures are set to 25. The $z$-velocities for the ion and electron fluids are given by
	\begin{equation}
		v_{z_i} = - v_{z_e}
		=
		\dfrac{B_0}{r_i \rho_i \sqrt{\pi D}} \exp\left(-\dfrac{x^2}{4D}\right).
	\end{equation}
	All remaining variables are set to zero. We take adiabatic index $\gamma=4/3$, and integrate from initial time $t=1$ to final time $t=9$ using the schemes $\second$, $\third$ and $\fourth$. 
	\begin{figure}[!htbp]
		
		\begin{center}
			\subfigure{
				\includegraphics[width=2.5in, height=1.8in]{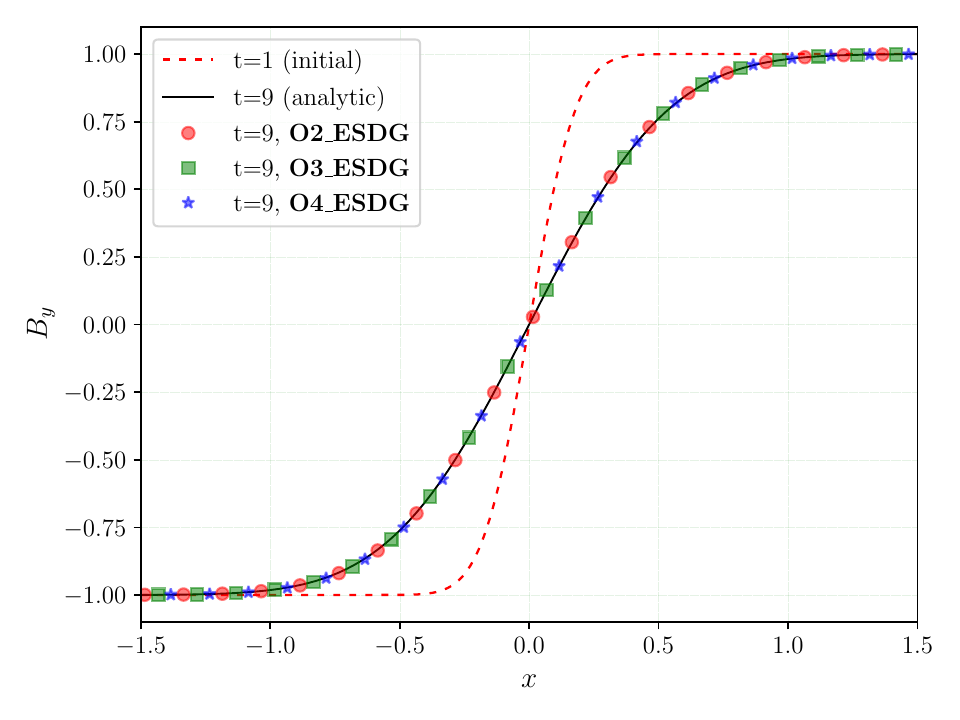}
			}
			\caption{\nameref{test:1d_sheet}: Comparison of the $B_y$ profiles using $\second$, $\third$ and $\fourth$ schemes.}
			\label{fig:sheet}
		\end{center}
	\end{figure}
	The numerical solution using 400 cells is shown in Figure~\ref{fig:sheet}. The figure also includes the analytical solutions at time $t=1$ and $t=9$. We observe that all the proposed schemes capture the analytical RMHD solution accurately.
	
	\subsection{Two-dimensional test cases}
	As the second-order scheme is quite diffusive compared to the third and fourth-order schemes unless very large grids are used which is more expensive in multi-dimensions, we only present results obtained using the third ($\third$) and fourth-order ($\fourth$) schemes. 
	
	\subsubsection{Relativistic Orzag-Tang test problem} \label{test:2d_ot}
	The test was first proposed for magnetohydrodynamics in~\cite{Orszag1979} and a two-fluid relativistic extension was considered in~\cite{Balsara2016}. We follow \cite{Balsara2016} and consider the computational domain of $[0,1] \times [0,1]$ with the periodic boundary conditions. The initial data is given by
	\begin{align*}
		\begin{pmatrix*}[c]
			\rho_i \\ v_{x_i} \\ v_{y_i} \\ p_i
		\end{pmatrix*} =
		\begin{pmatrix*}[c]
			\rho_e \\v_{x_e} \\ v_{y_e} \\ p_e
		\end{pmatrix*} =
		\begin{pmatrix*}[c]
			\frac{25}{72 \pi} \\ - \frac{\sin(2 \pi y)}{2} \\ \frac{\sin(2 \pi x)}{2} \\ \frac{5}{24 \pi}
		\end{pmatrix*}, \qquad
		\begin{pmatrix*}[c]
			B_x \\ B_y
		\end{pmatrix*} =
		\begin{pmatrix*}[c]
			-\sin(2 \pi y) \\  \sin (4 \pi x)
		\end{pmatrix*}
	\end{align*}
	The electric field components are given by the vector cross-product $-\vel_i \times \bm{B}$. All the remaining variables are set to zero. The charge-to-mass ratios are set to $r_i$=$-r_e=10^3/\sqrt{4 \pi}$  and the adiabatic index is assumed to be $\gamma = 5/3$. We run the simulation till the final time $t = 1$ using $200\times 200$ cells. 
	
	\begin{figure}[!htbp]
		\begin{center}
			\subfigure[Plot of total density ($\rho_i +\rho_e$)]{
				\includegraphics[width=2.2in, height=1.9in]{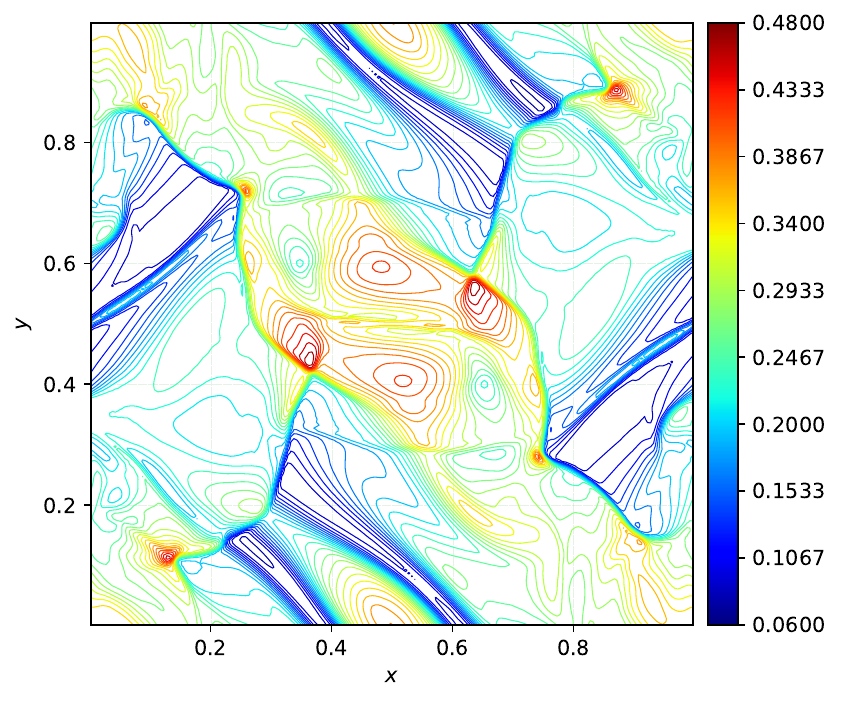}
				\label{fig:ot_o3_totrho}}
			\subfigure[Plot of total pressure ($p_{i} + p_e$)]{
				\includegraphics[width=2.2in, height=1.9in]{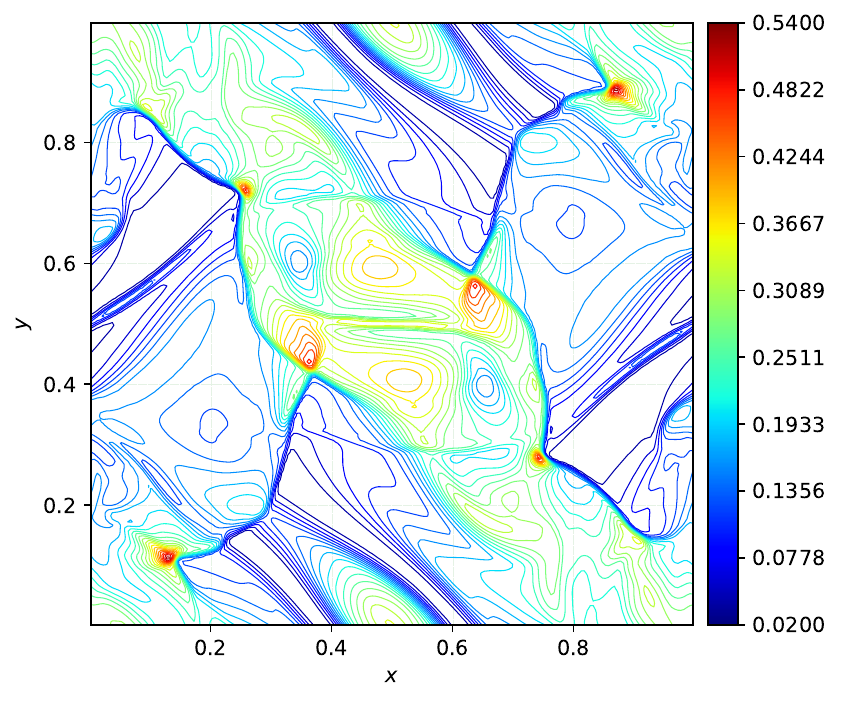}
				\label{fig:ot_o3_totp}}
			\subfigure[Plot of ion Lorentz factor ($\Gamma_i$)]{
				\includegraphics[width=2.2in, height=1.9in]{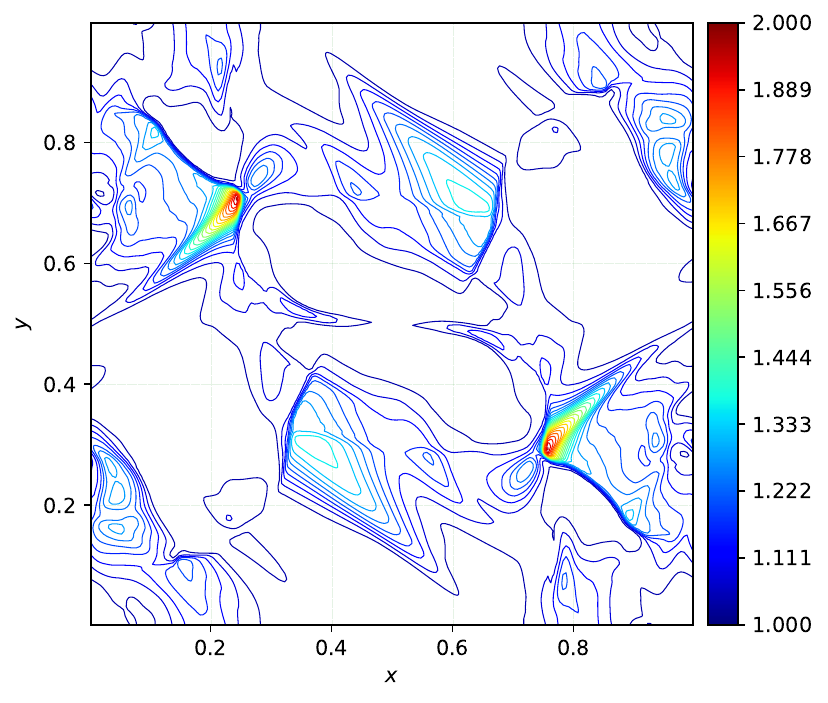}
				\label{fig:ot_o3_lorentz}}
			\subfigure[Plot of magnitude of the magnetic field ($\dfrac{|\bm{B}|^2}{2}$)]{
				\includegraphics[width=2.2in, height=1.9in]{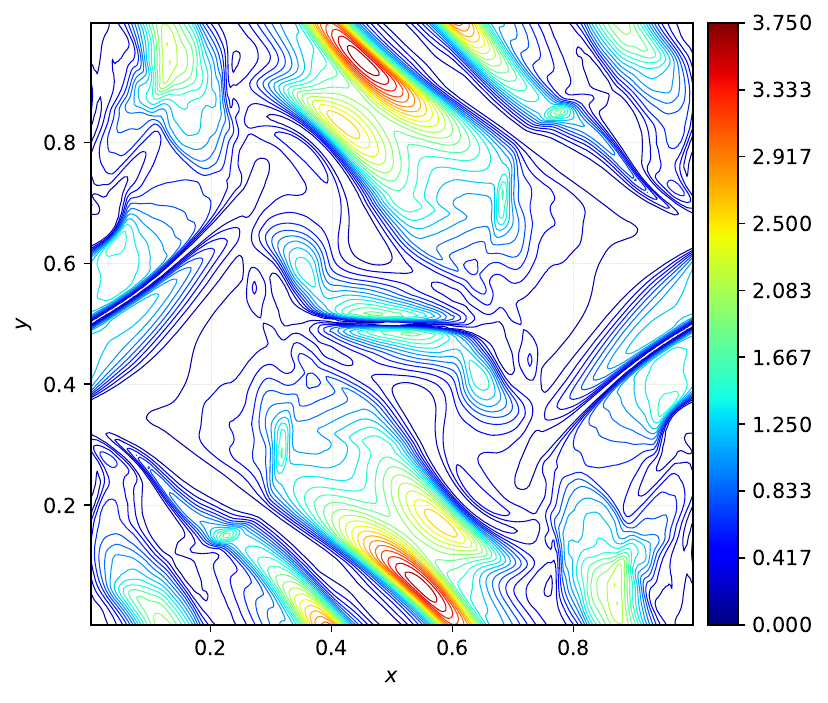}
				\label{fig:ot_o3_magBby2}}
			\caption{\nameref{test:2d_ot}: Plots of total density, total pressure, ion Lorentz factor, and magnitude of magnetic field using the $\third$ scheme with $200\times200$ cells. }
			\label{fig:ot_o3}
		\end{center}
	\end{figure}
	
	\begin{figure}[!htbp]
		\begin{center}
			\subfigure[Plot of total density ($\rho_i +\rho_e$)]{
				\includegraphics[width=2.2in, height=1.9in]{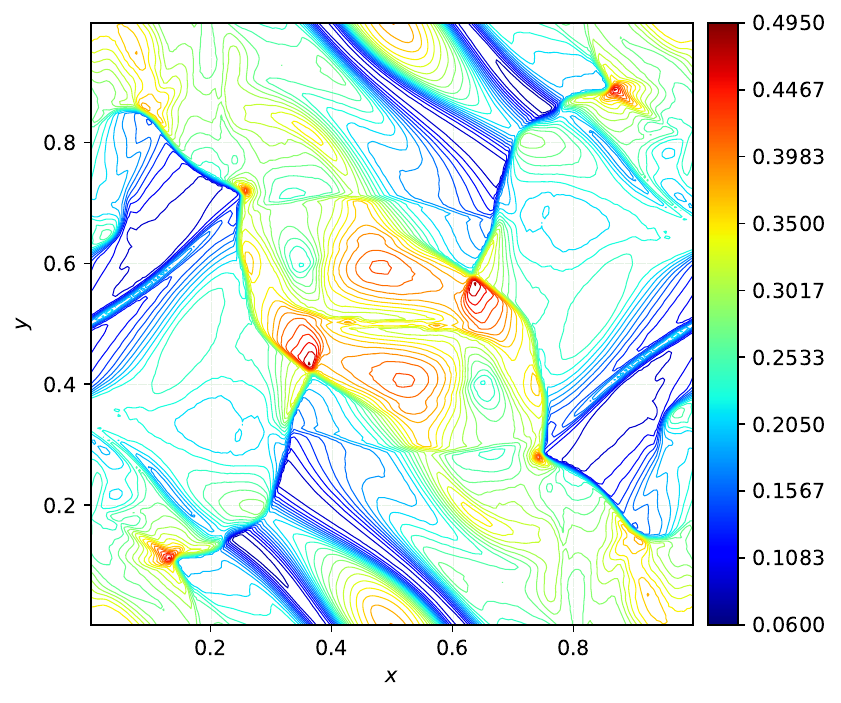}
				\label{fig:ot_o4_totrho}}
			\subfigure[Plot of total pressure ($p_{i} + p_e$)]{
				\includegraphics[width=2.2in, height=1.9in]{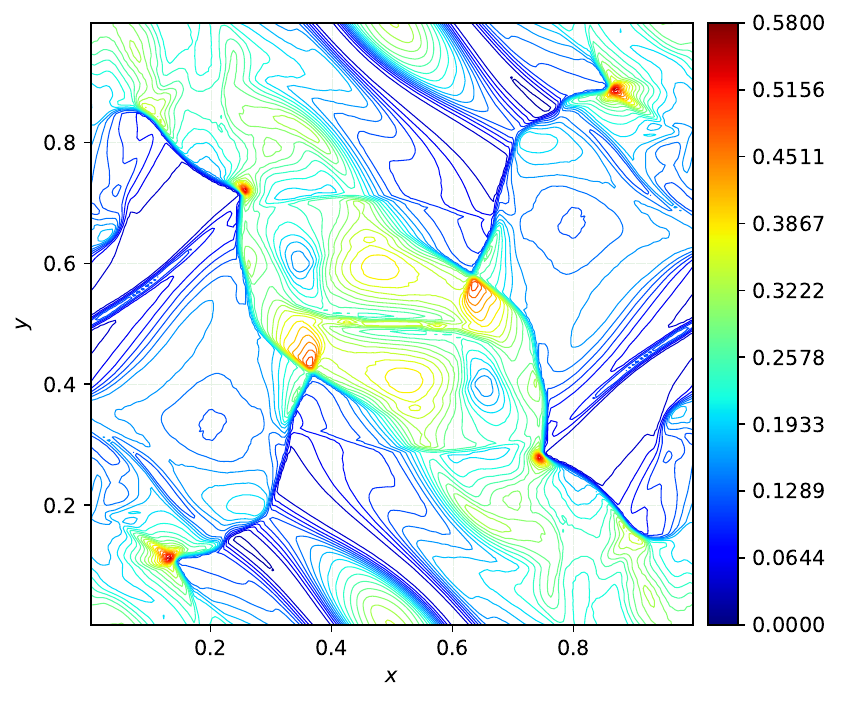}
				\label{fig:ot_o4_totp}}
			\subfigure[Plot of ion Lorentz factor ($\Gamma_i$)]{
				\includegraphics[width=2.2in, height=1.9in]{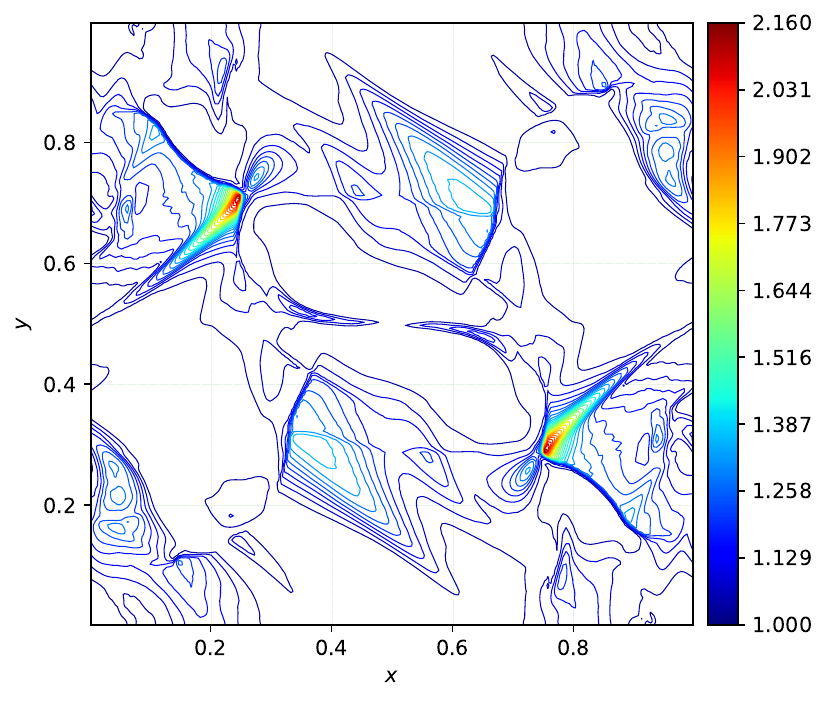}
				\label{fig:ot_o4_lorentz}}
			\subfigure[Plot of magnitude of the magnetic field ($\dfrac{|\bm{B}|^2}{2}$)]{
				\includegraphics[width=2.2in, height=1.9in]{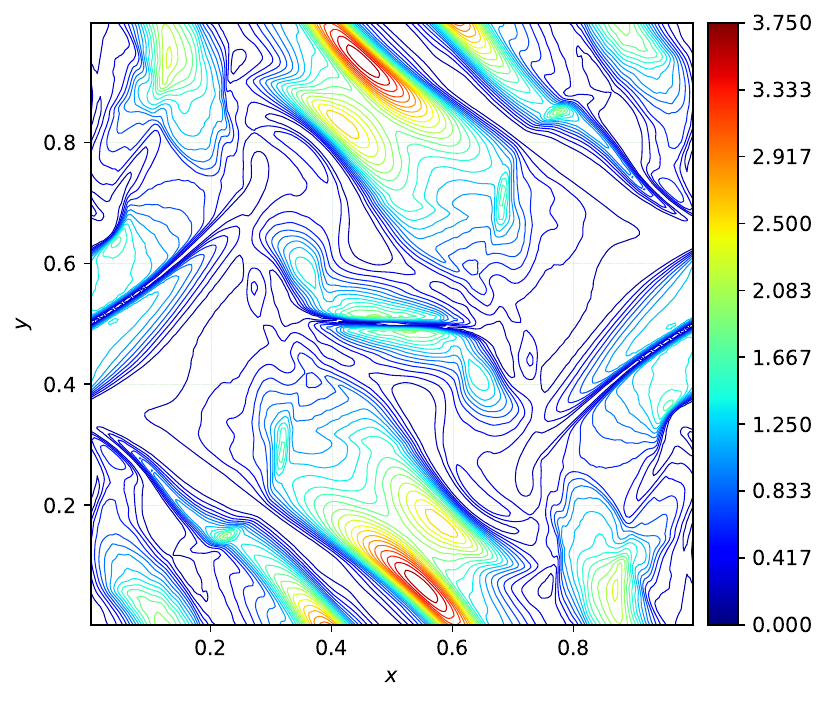}
				\label{fig:ot_o4_magBby2}}
			\caption{\nameref{test:2d_ot}: Plots of total density, total pressure, ion Lorentz factor, and magnitude of magnetic field using the $\fourth$ scheme with $200\times200$ cells.}
			\label{fig:ot_o4}
		\end{center}
	\end{figure}
	Numerical results for the $\third$ and $\fourth$ schemes are shown in Figures ~\ref{fig:ot_o3} and \ref{fig:ot_o4}, respectively. We plot the total density, total pressure, ion Lorentz factor, and magnitude of the magnetic field. We observe that both schemes are able to capture the  two-dimensional  discontinuities. Furthermore, the solutions are highly refined and comparable to those presented in~\cite{Balsara2016}. 	
	\begin{figure}[!htbp]
		\begin{center}
			\includegraphics[width=2.5in, height=1.8in]{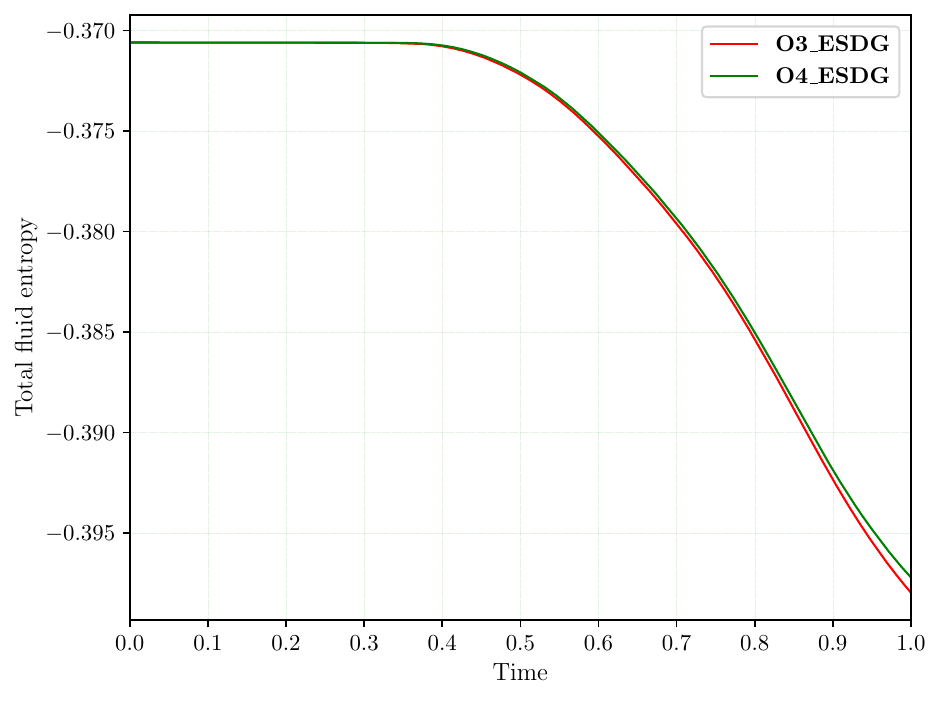}
			\caption{\nameref{test:2d_ot}: Total entropy evolution of the $\third$ and $\fourth$ schemes with $200\times200$ cells.}
			\label{fig:ot_entropy}
		\end{center}
	\end{figure}
	
	In Figure \ref{fig:ot_entropy}, we plot the time evolution of the total fluid entropy  for the  $\third$ and $\fourth$ schemes. Initially, we do not see significant decay in the entropies as discontinuities are not present in the solution. However, as the solutions develop discontinuities, we see sharp decay in the entropies after time $t=0.4$. We also observe that the third-order scheme decays slightly more entropy than the fourth-order scheme; however, the difference is very small.

	\subsubsection{Relativistic two-fluid blast problem} \label{test:2d_blast}
	We consider a two-fluid relativistic strong cylindrical blast problem from~\cite{Amano2016} which is a relativistic extension of the strong cylindrical explosion problem for MHD described in~\cite{Komissarov1999}. We consider a square computational domain of size $[-6,6]\times[-6,6]$ with the Neumann boundary conditions. To describe the initial conditions appropriately, let us denote $\rho_{in}=10^{-2}$, $p_{in}=1$, $\rho_{out}=10^{-4}$ and $p_{out}=5 \times 10^{-4}$. We also define radial distance $r$ from the origin $(0,0)$ by $r=\sqrt{x^2+y^2}$. In the interior of the disc of radius $0.8$, i.e., $r<0.8$, we set $\rho_i=\rho_e=0.5 \times \rho_{in}$ and $p_i=p_e=p_{in}$, while, on the outside of the disc of radius 1, i.e., $r>1$, we set $\rho_i=\rho_e=0.5 \times \rho_{out}$ and $p_i=p_e=p_{out}$. In the intermediate region, i.e., $0.8 \le r \le 1$, the densities and pressures are defined by a linear profile such that the densities and pressures match at the values $r=0.8$ and $r=1$. The $x$-component of the magnetic field is defined as $B_x=B_0$, where $B_0$ is a parameter that describes the magnetized strength of the medium. We set all remaining variables to zero. The charge-to-mass ratios are given by $r_i =-r_e=10^3$ and we set $\gamma=4/3$. Simulations are performed on a grid of $200 \times 200$ cells using $\third$ and $\fourth$ schemes and we evolve the solution till the final time $t=4$. 
	
	We first consider the weakly magnetized medium by taking $B_0=0.1$.  In Figures~\ref{fig:blast_o3_0_1} and \ref{fig:blast_o4_0_1}, we have presented results for $\third$ and $\fourth$ schemes, respectively.  We have plotted $\log(\rho_i+\rho_e)$, $\log(p_i+p_e)$, ion-Lorentz factor $\Gamma_i$ and magnitude of the magnetic field $\dfrac{|\bm{B}|^2}{2}$. Both schemes can capture outward-moving shock wave in great detail, and the results are consistent with those presented in \cite{Balsara2016}.
	
	Next, to observe the effect of the higher magnetic field, we consider stronger magnetized medium by taking $B_0=1.0$. The numerical results for $\third$ and $\fourth$ schemes are plotted in Figures \ref{fig:blast_o3_1_0}, and \ref{fig:blast_o4_1_0}, respectively. In this case, we observe that the solution profile is changed significantly from the weakly magnetized case. The solution is more aligned along the $x$-axis. Both the schemes produced highly detailed solutions, with the $\fourth$ scheme having more details. 
	
	\begin{figure}[!htbp]
		\begin{center}
			\subfigure[Plot of $\log(\rho_i+\rho_e)$]{
				\includegraphics[width=2.2in, height=1.9in]{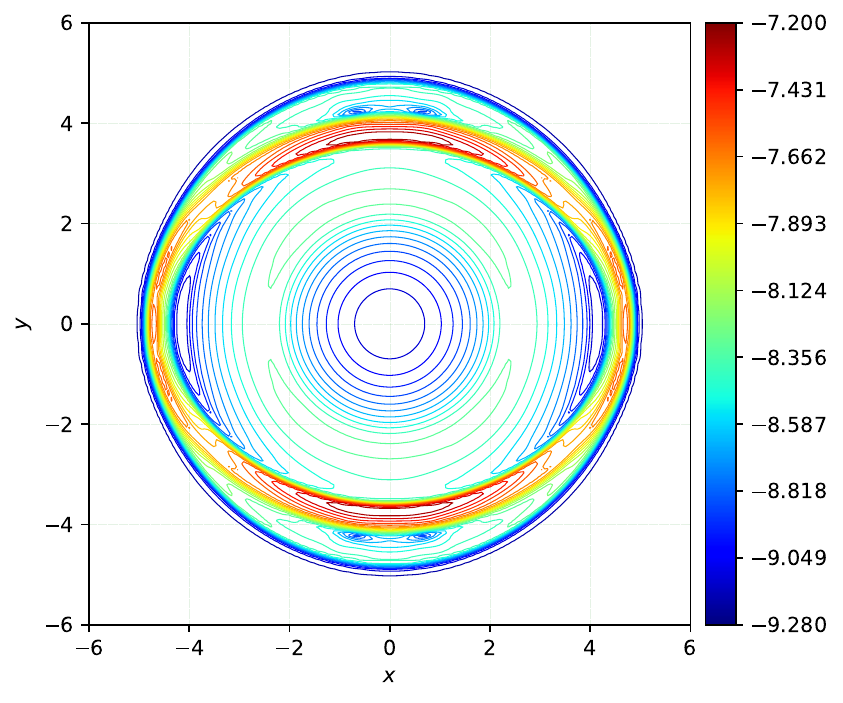}
			}
			\subfigure[Plot of $\log(p_i+p_e)$]{
				\includegraphics[width=2.2in, height=1.9in]{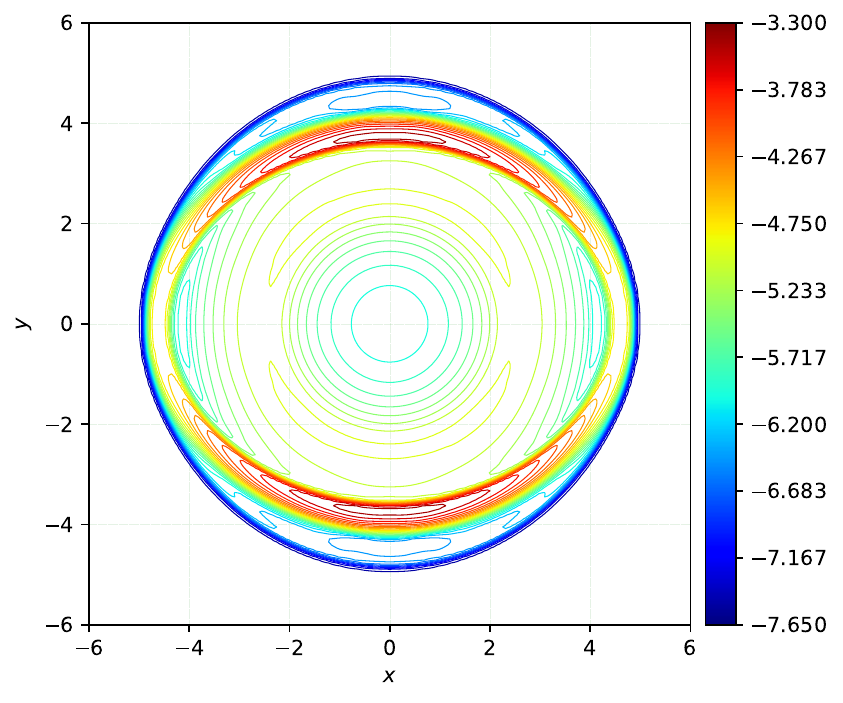}
			}
			\subfigure[Plot of ion Lorentz factor ($\Gamma_i$)]{
				\includegraphics[width=2.2in, height=1.9in]{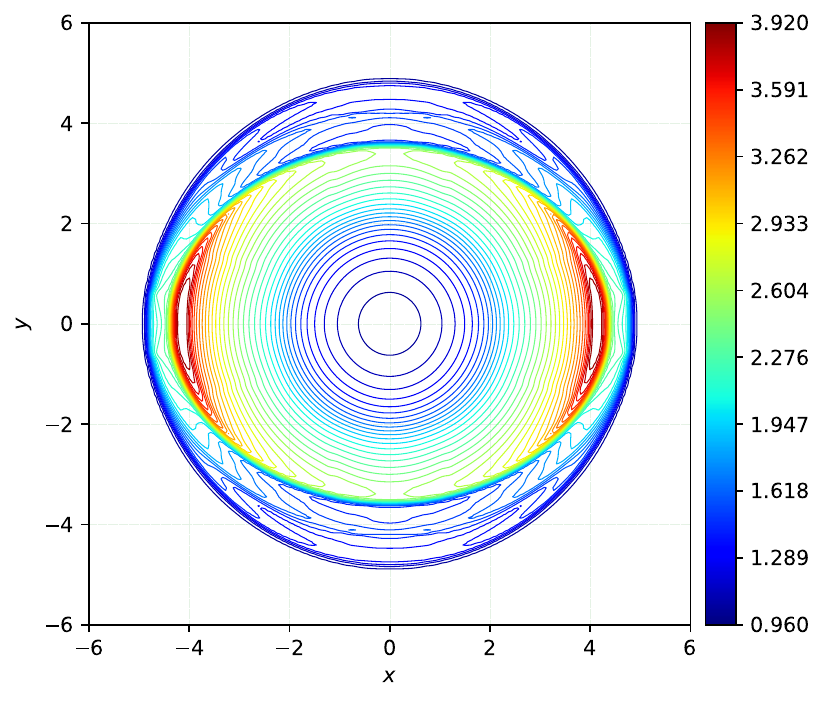}
			}
			\subfigure[Plot of magnitude of the magnetic field ($\dfrac{|\bm{B}|^2}{2}$)]{
				\includegraphics[width=2.2in, height=1.9in]{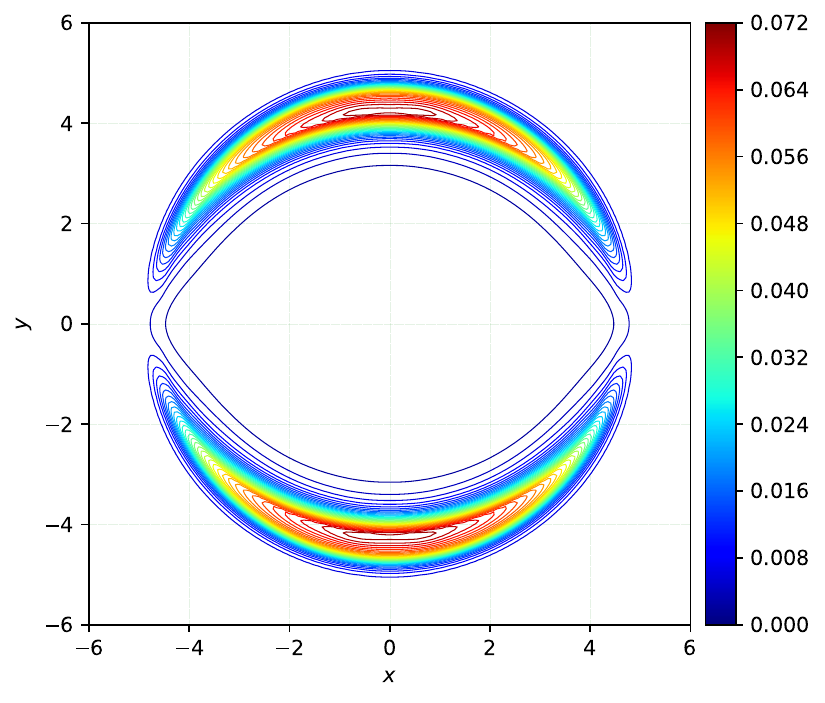}
			}
			\caption{\nameref{test:2d_blast}: Plots for the weakly magnetized medium $B_0=0.1$, using  $\third$ scheme with $200\times200$ cells.}
			\label{fig:blast_o3_0_1}
		\end{center}
	\end{figure}
	
	\begin{figure}[!htbp]
		\begin{center}
			\subfigure[Plot of $\log(\rho_i+\rho_e)$]{
				\includegraphics[width=2.2in, height=1.9in]{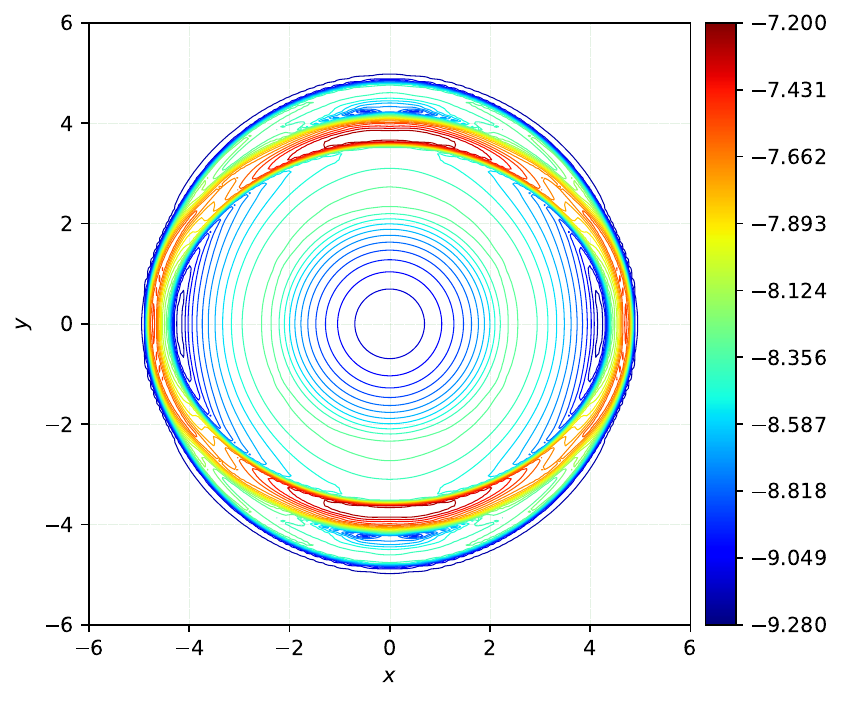}
			}
			\subfigure[Plot of $\log(p_i+p_e)$]{
				\includegraphics[width=2.2in, height=1.9in]{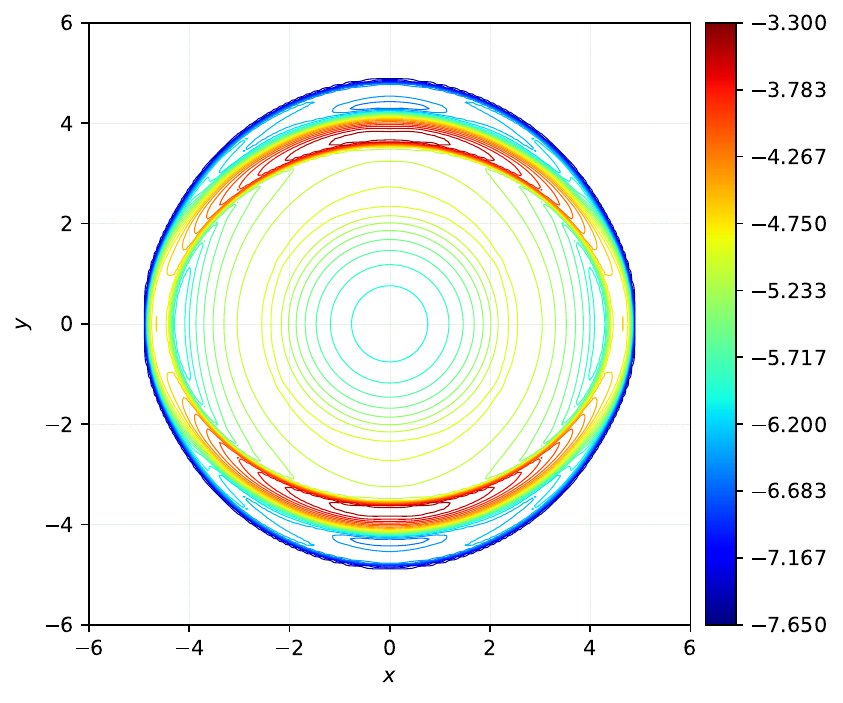}
			}
			\subfigure[Plot of ion Lorentz factor ($\Gamma_i$)]{
				\includegraphics[width=2.2in, height=1.9in]{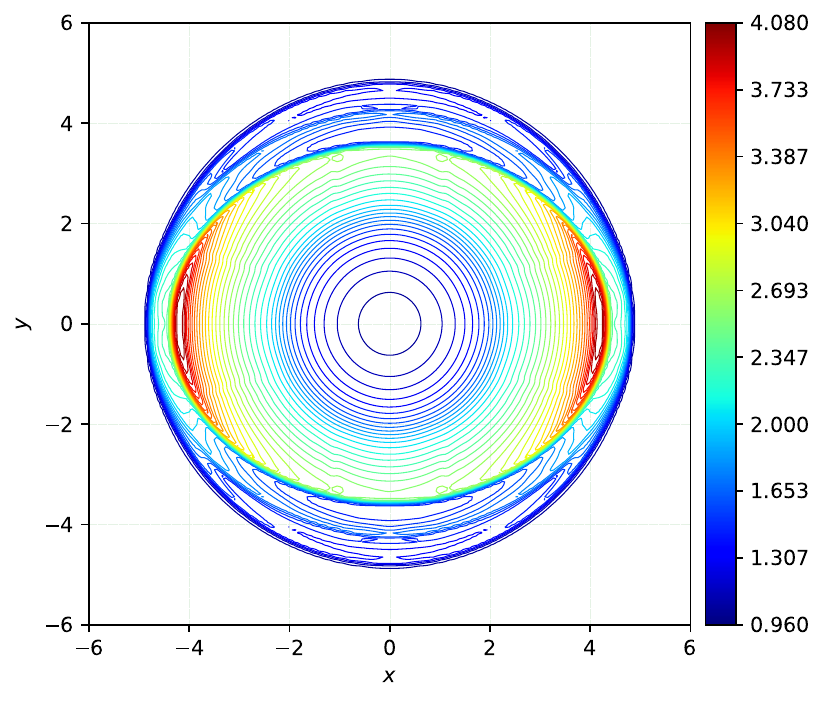}
			}
			\subfigure[Plot of magnitude of the magnetic field ($\dfrac{|\bm{B}|^2}{2}$)]{
				\includegraphics[width=2.2in, height=1.9in]{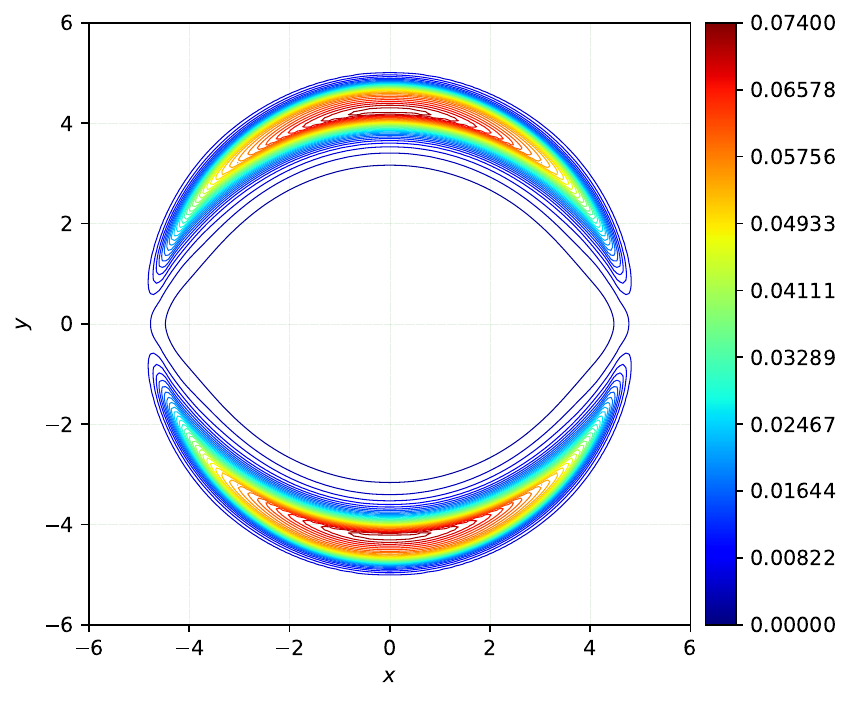}
			}
			\caption{\nameref{test:2d_blast}: Plots for the weakly magnetized medium $B_0=0.1$, using scheme $\fourth$ with $200\times200$ cells. }
			\label{fig:blast_o4_0_1}
		\end{center}
	\end{figure}
	
	\begin{figure}[!htbp]
		\begin{center}
			\subfigure[Plot of $\log(\rho_i+\rho_e)$]{
				\includegraphics[width=2.2in, height=1.9in]{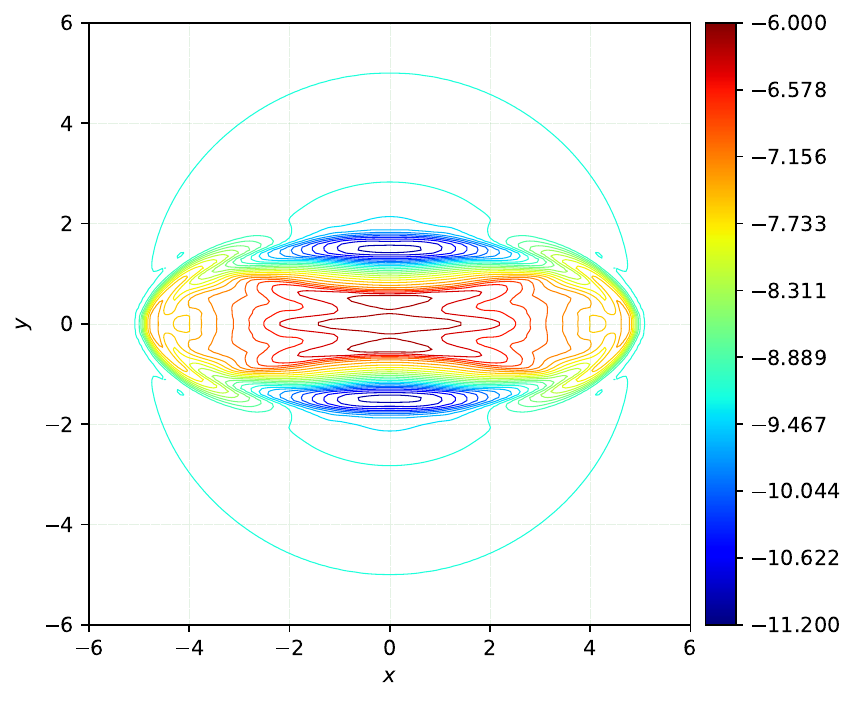}
			}
			\subfigure[Plot of $\log(p_i+p_e)$]{
				\includegraphics[width=2.2in, height=1.9in]{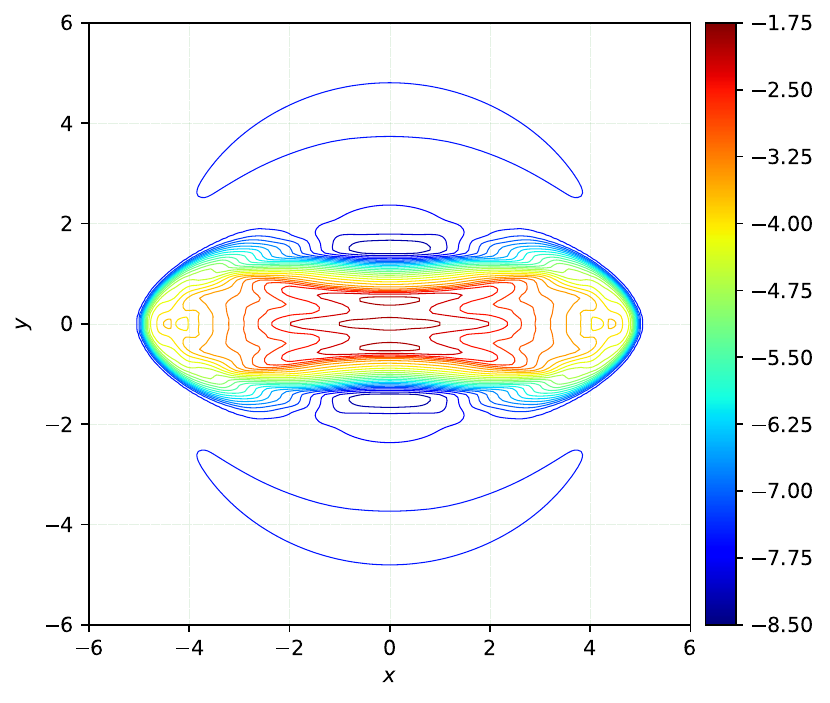}
			}
			\subfigure[Plot of ion Lorentz factor ($\Gamma_i$)]{
				\includegraphics[width=2.2in, height=1.9in]{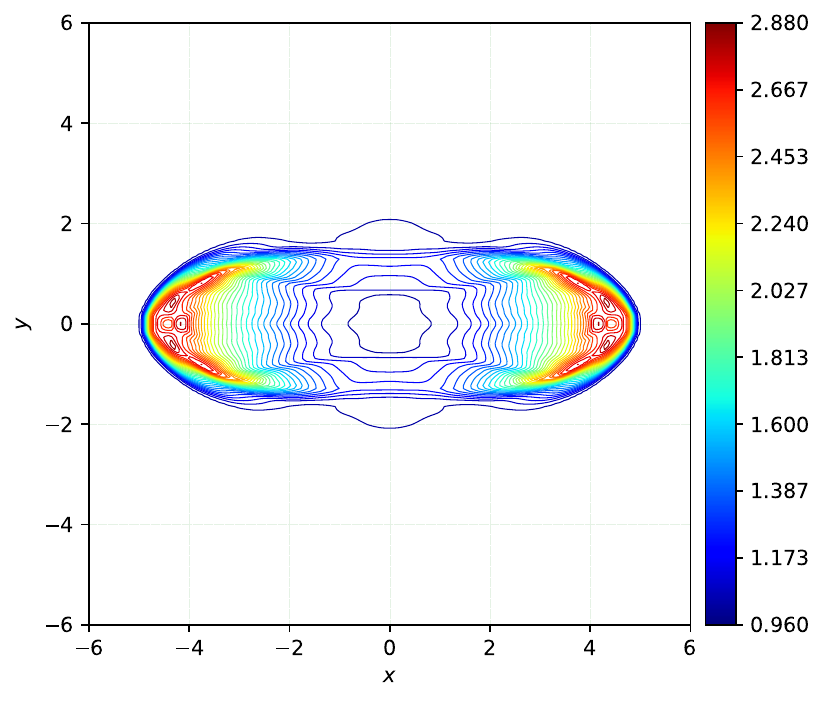}
			}
			\subfigure[Plot of magnitude of the magnetic field ($\dfrac{|\bm{B}|^2}{2}$)]{
				\includegraphics[width=2.2in, height=1.9in]{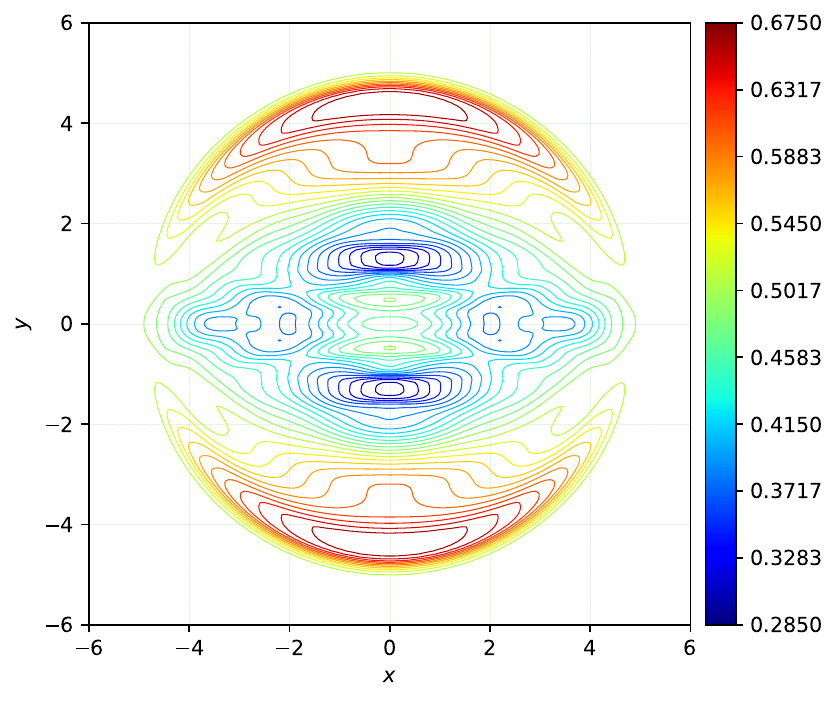}
			}
			\caption{\nameref{test:2d_blast}: Plots for the strongly magnetized medium $B_0=1$, using $\third$ scheme with $200\times200$ cells. }
			\label{fig:blast_o3_1_0}
		\end{center}
	\end{figure}
	
	\begin{figure}[!htbp]
		\begin{center}
			\subfigure[Plot of $\log(\rho_i+\rho_e)$]{
				\includegraphics[width=2.2in, height=1.9in]{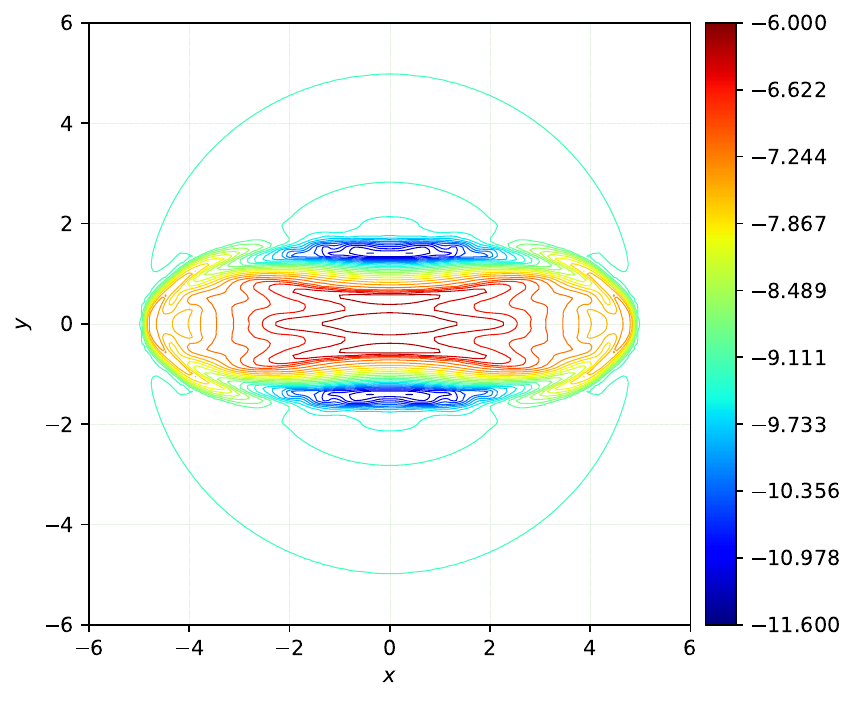}
			}
			\subfigure[Plot of $\log(p_i+p_e)$]{
				\includegraphics[width=2.2in, height=1.9in]{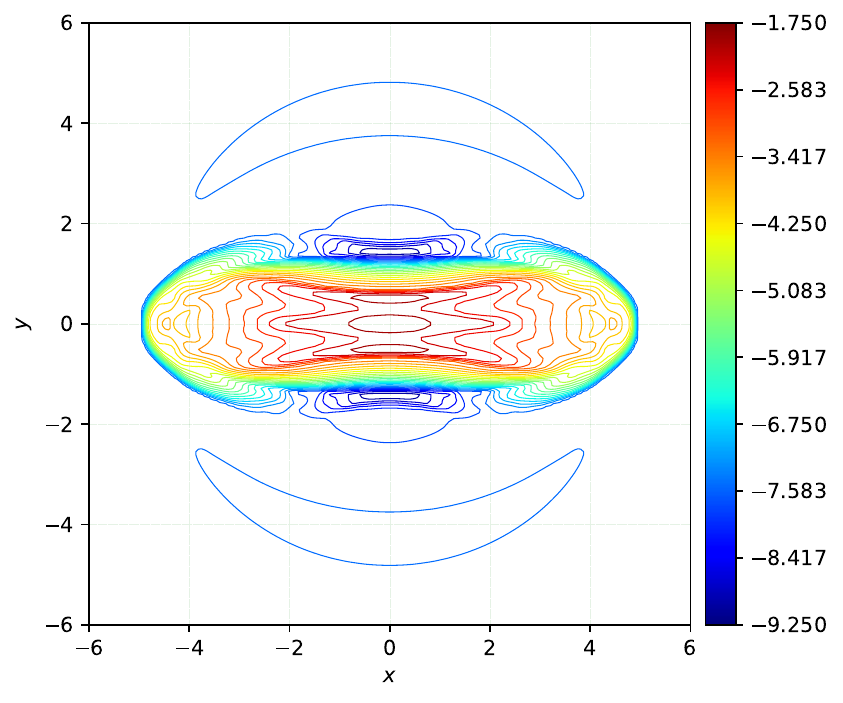}
			}
			\subfigure[Plot of ion Lorentz factor ($\Gamma_i$)]{
				\includegraphics[width=2.2in, height=1.9in]{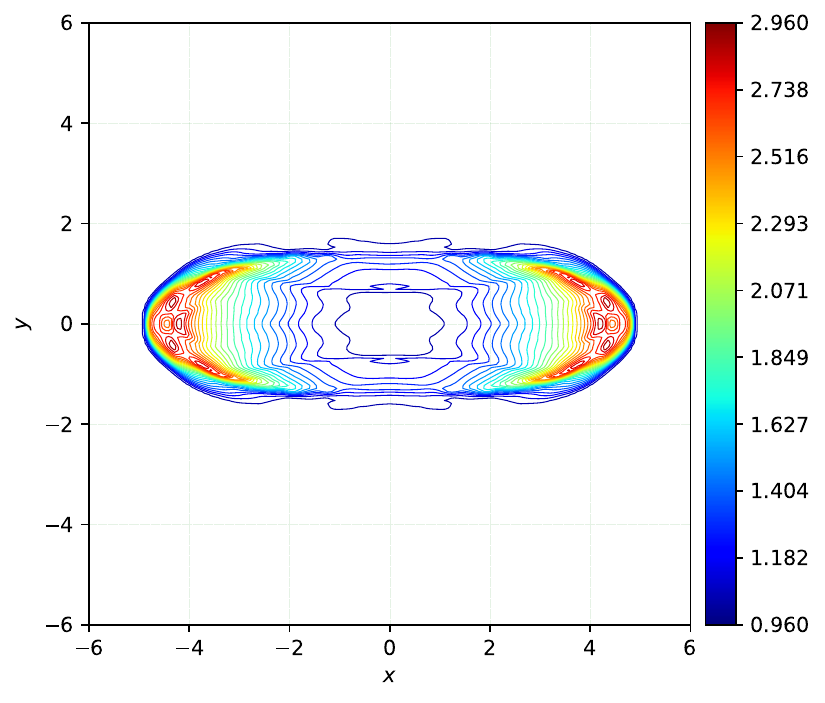}
			}
			\subfigure[Plot of magnitude of the magnetic field ($\dfrac{|\bm{B}|^2}{2}$)]{
				\includegraphics[width=2.2in, height=1.9in]{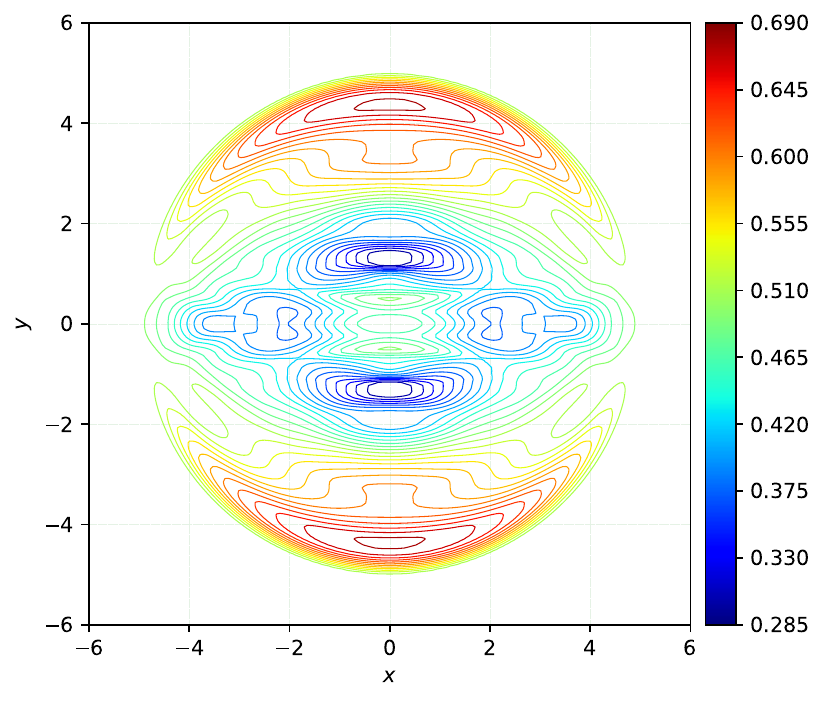}
			}
			\caption{\nameref{test:2d_blast}: Plots for the strongly magnetized medium $B_0=1$, using $\fourth$ scheme with $200\times200$ cells. }
			\label{fig:blast_o4_1_0}
		\end{center}
	\end{figure}
	
	\begin{figure}[!htbp]
		\begin{center}
			\subfigure[Plot of total entropy evolution for weakly magnetized medium, $B_0=0.1$.]{
				\includegraphics[width=2.2in, height=1.6in]{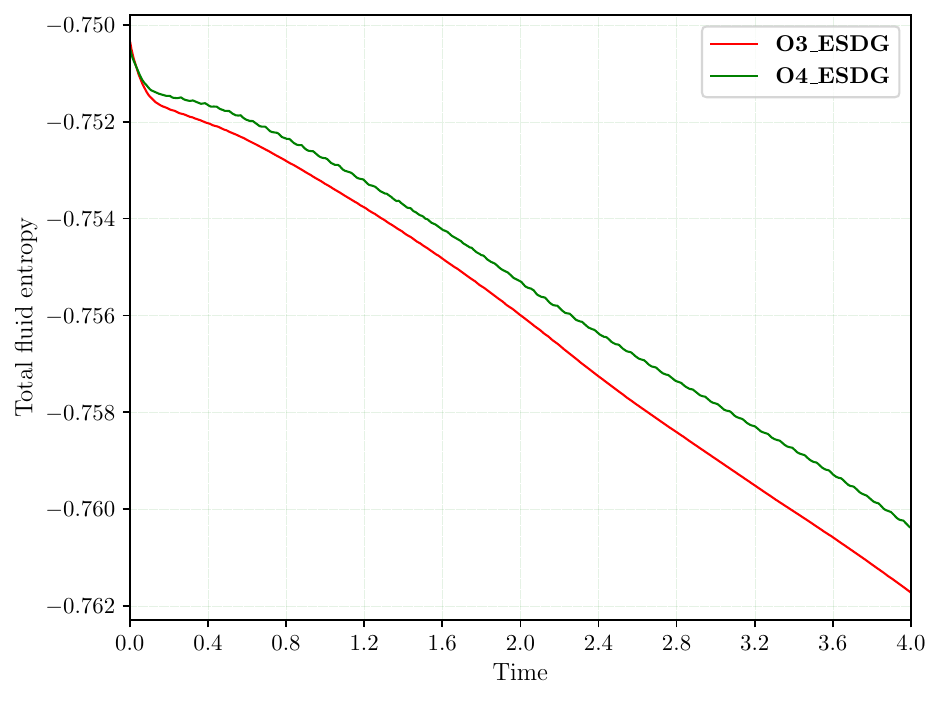}
				\label{fig:blast_0p1_entropy}}
			\subfigure[Plot of total entropy evolution for strongly magnetized medium, $B_0=1$.]{
				\includegraphics[width=2.2in, height=1.6in]{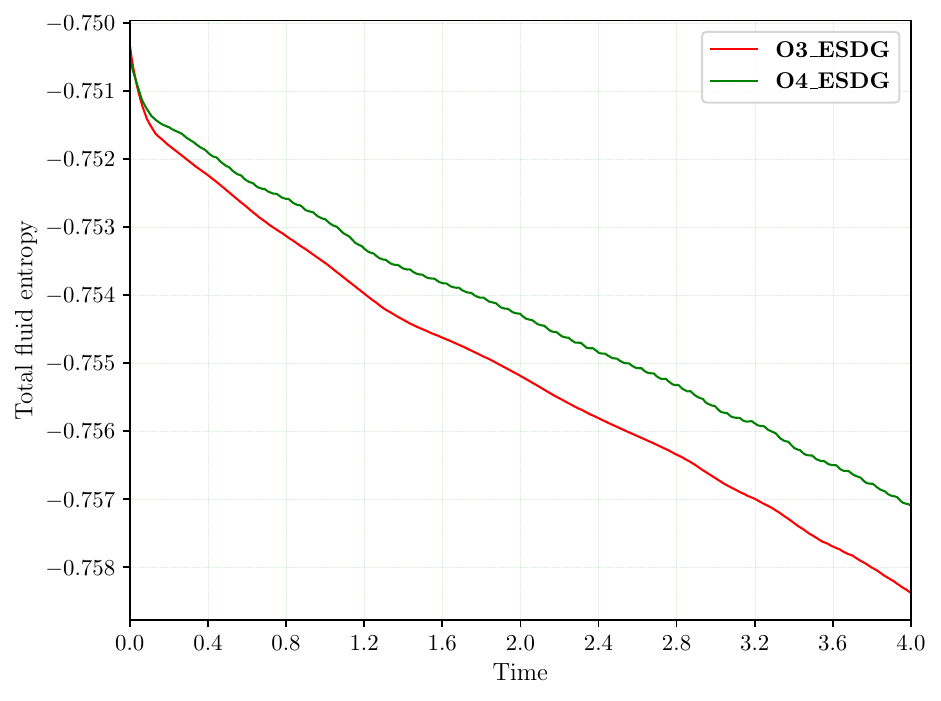}
				\label{fig:blast_1p0_entropy}}
			\caption{\nameref{test:2d_blast}: Plots of total entropy evolution for $\third$ and $\fourth$ schemes with $200\times200$ cells.}
			\label{fig:blast_entropy}
		\end{center}
	\end{figure}	
	
	In Figure \ref{fig:blast_entropy}, we plot the evolution of the total entropy for the $\third$ and $\fourth$ schemes. In both weakly and strongly magnetized cases, we observe that the $\third$ scheme is more entropy diffusive than the $\fourth$ scheme.
	
	\subsubsection{Relativistic two-fluid GEM challenge problem} \label{test:2d_gem}
	In this well-known test case, we consider a two-fluid relativistic extension of the GEM (Geospace Environment Modeling) magnetic reconnection problem. The non-relativistic GEM  magnetic reconnection test was considered in~\cite{Birn2001}, and the extension to two-fluid relativistic case is described in~\cite{Amano2016}. The computational domain is $[-L_x/2,L_x/2] \times [-L_y/2,L_y/2]$, where $L_x=8\pi$ and $L_y=4\pi$. We use periodic boundary conditions at $x=\pm L_x/2$ and conducting wall boundary at $y=\pm L_y/2$ boundary. We set $r_i=1$ and $r_e=-25$ with the ion-electron mass ratio as $m_i/m_e = 25$ where $m_i=1$. The problem is initialized in the $x$-direction with the state given by $B_x(y) = B_0 \tan (y/d)$ with $B_0=1$, where $d=1$ is the thickness of the current sheet. The unperturbed magnetic field components, $B_y$ and $B_z$, are set to zero. \reva{Following~\cite{Amano2016}, we introduce the magnetization parameter, denoted by $\sigma_s$, as the ratio between the cyclotron and plasma frequency squared. For the considered unperturbed magnetic field, the magnetized parameter simplifies to $\sigma_s=B_0^2/m_s$ ($s\in{i,e}$), which gives $\sigma_i=1$ and $\sigma_e=25$. This confirms that the charged particles are strongly magnetized.} The remaining initial conditions are given by
	\begin{align*}
		\begin{pmatrix}
			\rho_i  \\
			v_{z_i} \\
			p_i     \\ \\
			\rho_e  \\
			v_{z_e} \\
			p_e     \\ \\
			B_x     \\
			B_y
		\end{pmatrix} =
		\begin{pmatrix}
			n                                                  \\
			\frac{c}{2d}\frac{B_0 \mathrm{sech}^2(y/d)}{n}              \\
			0.2 + \frac{B_0^2 \mathrm{sech}^2(y/d)}{4} \frac{5}{24 \pi} \\ \\
			\frac{m_e}{m_p} n                                  \\
			-v_{z_i}                                          \\
			- p_i                                              \\ \\
			B_0 \tan (y/d) - B_0 \psi_0 \frac{\pi}{L_y} \cos\left(\frac{\pi x}{L_x}\right) \sin\left(\frac{\pi y}{L_y}\right)
			\\
			B_0 \psi_0 \frac{\pi}{L_x} \sin\left(\frac{\pi x}{L_x}\right) \cos\left(\frac{\pi y}{L_y}\right)
		\end{pmatrix}
	\end{align*}
	where $n=\mathrm{sech}^2(y/d)+0.2$ is the number density. We set the remaining variables to zero. The resistivity constant is given by $\eta=0.01$ and the adiabatic index is taken to be $\gamma = 4/3$. The computations are performed on a grid of $N_x \times N_y$ zones (where $N_x=256$ and $N_y=N_x/2$). We plot the total density ($\rho_i+\rho_e$), the $z$-component of the magnetic field $\bm{B}$, the $x$-component of ion velocity and $x$-component of electron velocity. These plots overlap with the field's magnetic field lines $(B_x,B_y)$.
	
	In Figures \ref{fig:gem_o3_40} and  \ref{fig:gem_o4_40}, we have plotted the results for the $\third$ and $\fourth$ schemes at $t=40$. We observe that magnetic reconnection is already underway. Furthermore, both schemes result in similar results.  In Figures \ref{fig:gem_o3_80} and  \ref{fig:gem_o4_80}, we observe that the magnetic reconnection is significantly increased, and the fluid velocities have jumped significantly. Again both schemes have produced similar results. We also note that at $t=80$, the fluid state is significantly turbulent. We observe that the results are comparable to those presented in~\cite{Amano2016},\cite{Balsara2016}. In Figure \ref{fig:gem_entropy}, we present the time evolution of the total entropy for the $\third$ and $\fourth$ schemes. We observe both schemes have almost the same entropy decay.
 	
	In Figure~\ref{fig:gem_recon}, we have plotted magnetic reconnection flux $\psi(t)$,
	\begin{equation*}
		\psi(t) = \dfrac{1}{2 B_0} \int_{-L_x/2}^{L_x/2}
		| B_y(x,y=0,t) | dx.
	\end{equation*}
	for the $\second$, $\third$ and $\fourth$ schemes. We have compared them with the rates presented in~\cite{Amano2016}. We observe that all the results match the published results.

	\begin{figure}[!htbp]
		\begin{center}
			\subfigure[Total Density ($\rho_i + \rho_e$)]{
				\includegraphics[width=2.2in, height=1.2in]{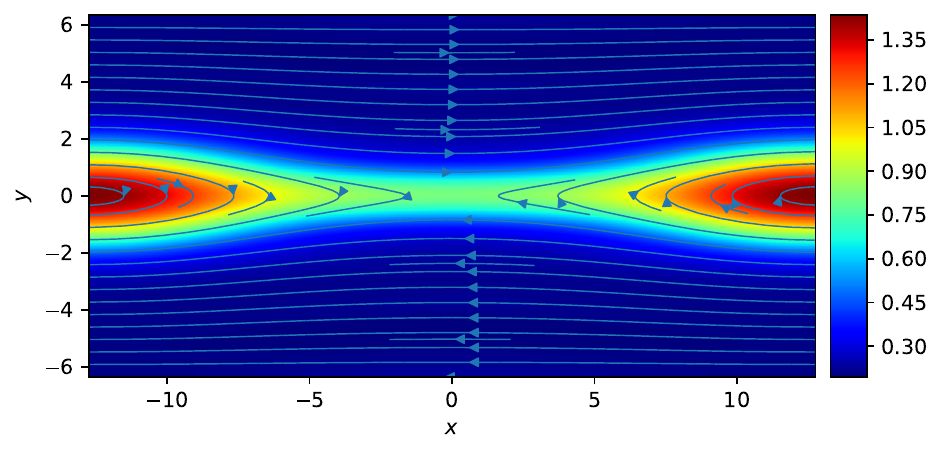}
				\label{fig:gem_o3_40_rho}}
			\subfigure[$B_z$]{
				\includegraphics[width=2.2in, height=1.2in]{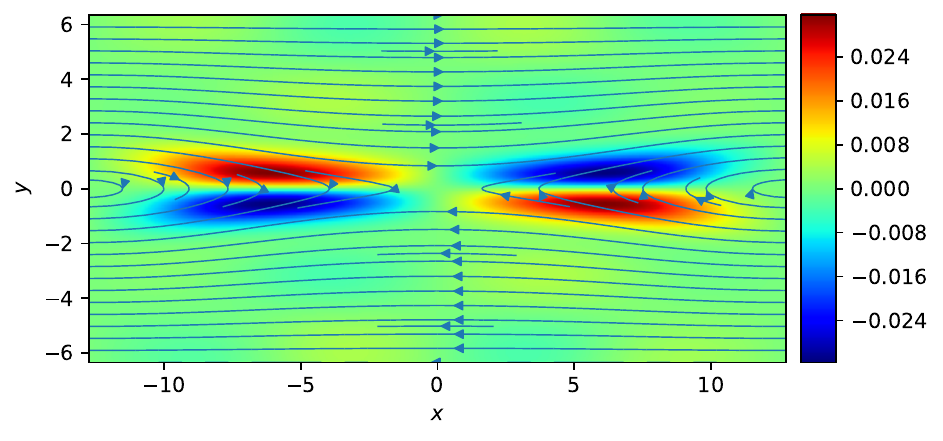}
				\label{fig:gem_o3_40_Bz}}
			\subfigure[Ion $x$-velocity]{
				\includegraphics[width=2.2in, height=1.2in]{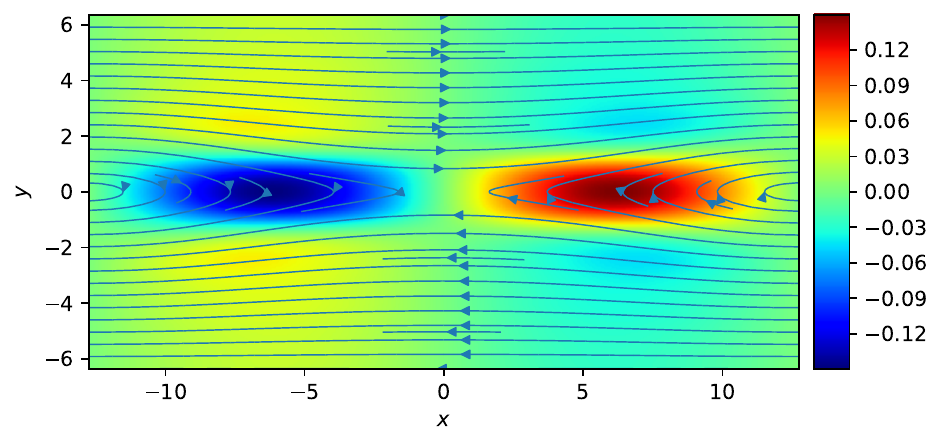}
				\label{fig:gem_o3_40_uxi}}
			\subfigure[Electron $x$-velocity]{
				\includegraphics[width=2.2in, height=1.2in]{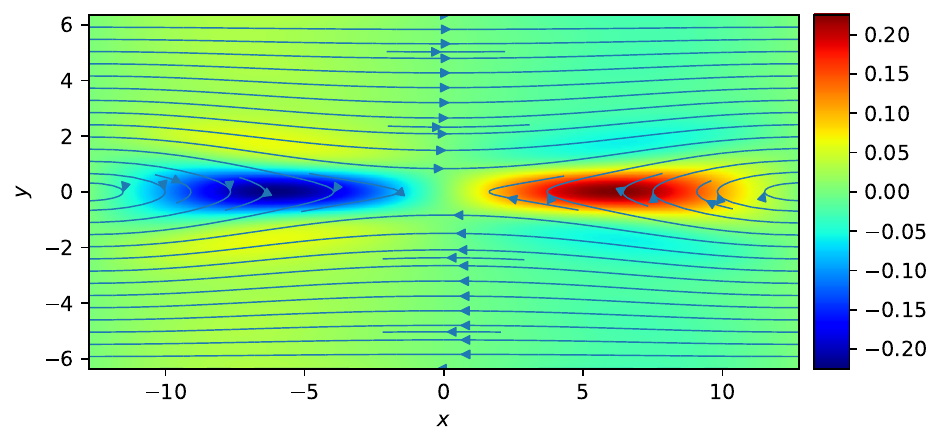}
				\label{fig:gem_o3_40_uxe}}
			\caption{\nameref{test:2d_gem}: Plots for the Total Density, $B_z$-component, Ion $x$-velocity, and Electron $x$-velocity on the rectangular grid of size $256 \times 128$, using the scheme $\third$, at time $t=40$. }
			\label{fig:gem_o3_40}
		\end{center}
	\end{figure}
	
	\begin{figure}[!htbp]
		\begin{center}
			\subfigure[Total Density ($\rho_i + \rho_e$)]{
				\includegraphics[width=2.2in, height=1.2in]{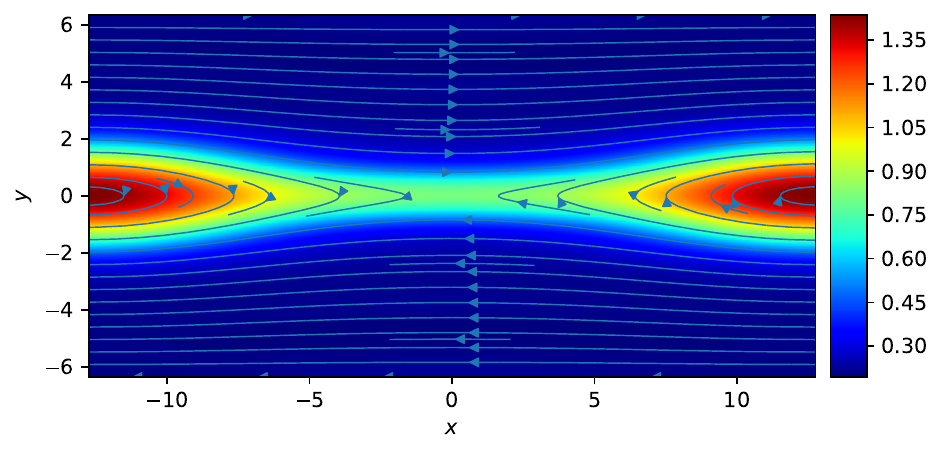}
				\label{fig:gem_o4_40_rho}}
			\subfigure[$B_z$]{
				\includegraphics[width=2.2in, height=1.2in]{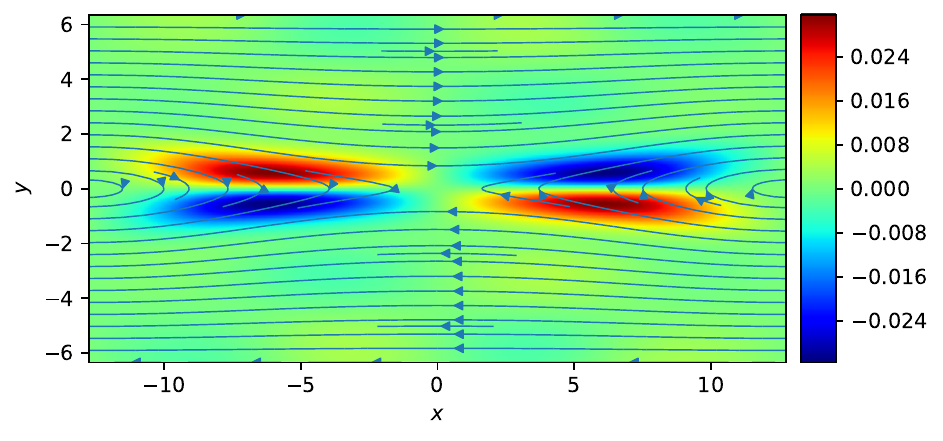}
				\label{fig:gem_o4_40_Bz}}
			\subfigure[Ion $x$-velocity]{
				\includegraphics[width=2.2in, height=1.2in]{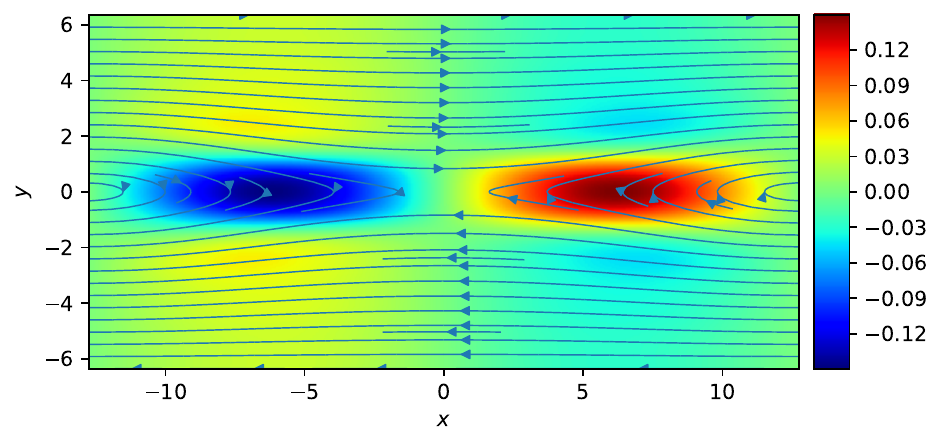}
				\label{fig:gem_o4_40_uxi}}
			\subfigure[Electron $x$-velocity]{
				\includegraphics[width=2.2in, height=1.2in]{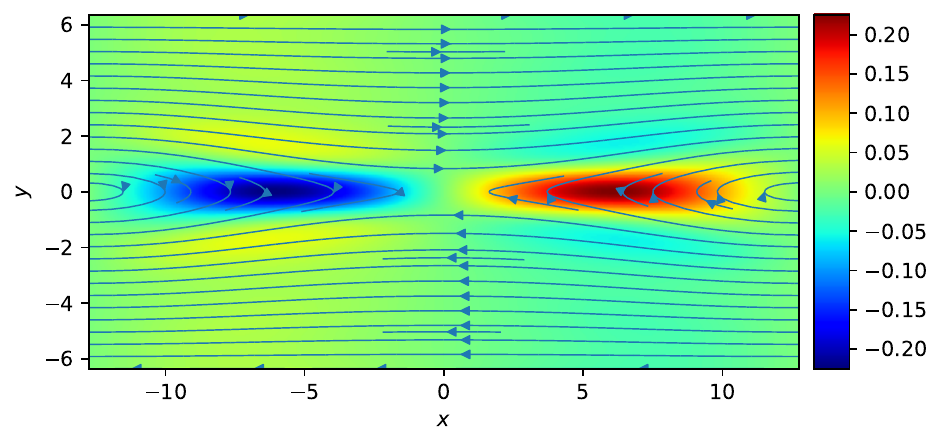}
				\label{fig:gem_o4_40_uxe}}
			\caption{\nameref{test:2d_gem}: Plots for the Total Density, $B_z$-component, Ion $x$-velocity, and Electron $x$-velocity on the rectangular grid of size $256 \times 128$, using the scheme $\fourth$, at time $t=40$. }
			\label{fig:gem_o4_40}
		\end{center}
	\end{figure}

	\begin{figure}[!htbp]
		\begin{center}
			\subfigure[Total Density ($\rho_i + \rho_e$)]{
				\includegraphics[width=2.2in, height=1.2in]{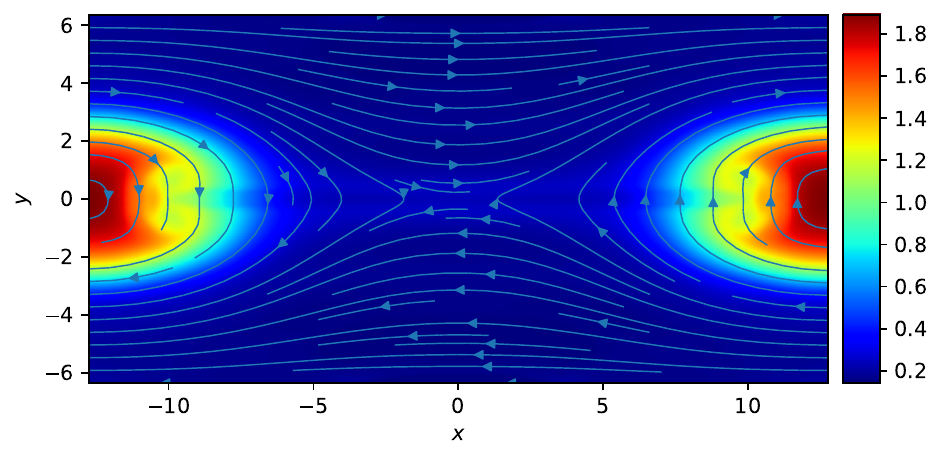}
				\label{fig:gem_o3_80_rho}}
			\subfigure[$B_z$]{
				\includegraphics[width=2.2in, height=1.2in]{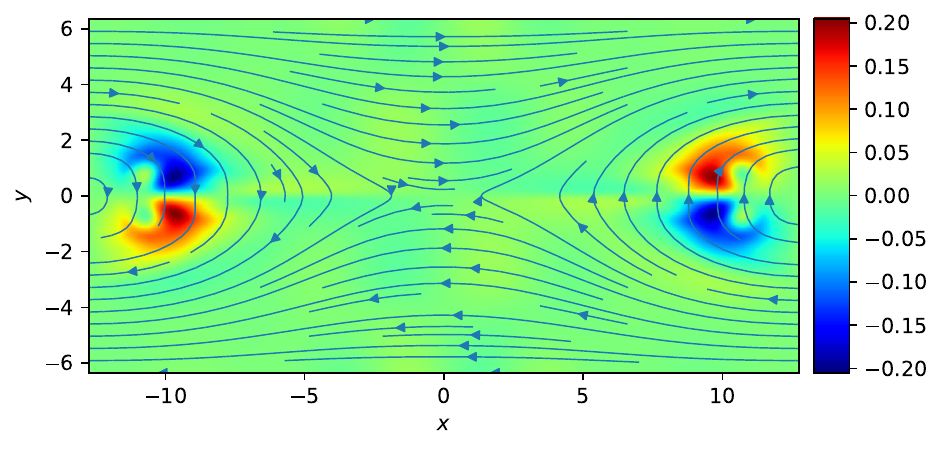}
				\label{fig:gem_o3_80_Bz}}
			\subfigure[Ion $x$-velocity]{
				\includegraphics[width=2.2in, height=1.2in]{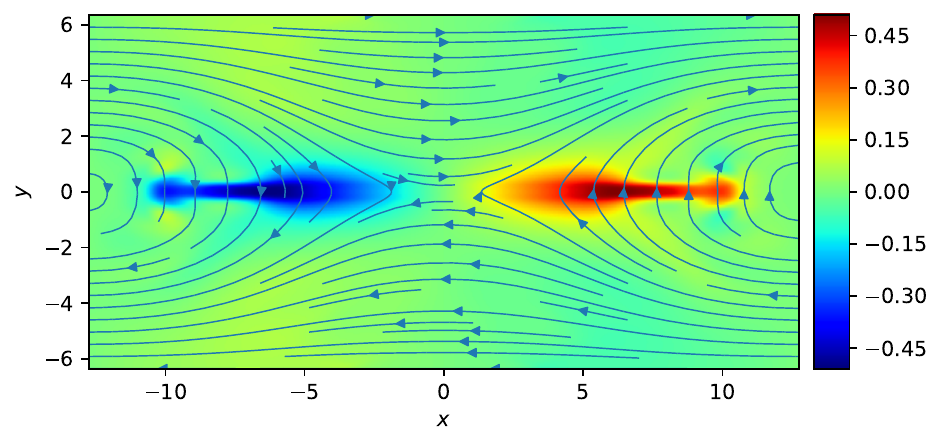}
				\label{fig:gem_o3_80_uxi}}
			\subfigure[Electron $x$-velocity]{
				\includegraphics[width=2.2in, height=1.2in]{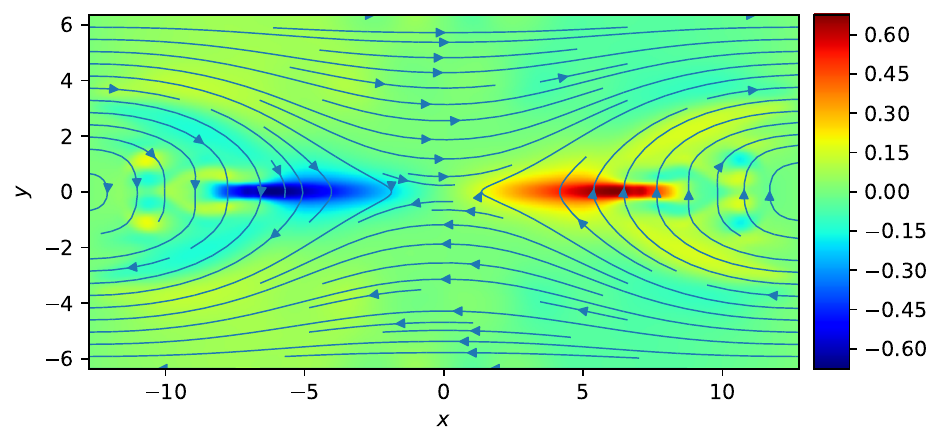}
				\label{fig:gem_o3_80_uxe}}
			\caption{\nameref{test:2d_gem}: Plots for the Total Density, $B_z$-component, Ion $x$-velocity, and Electron $x$-velocity on the rectangular grid of size $256 \times 128$, using the scheme $\third$, at time $t=80$. }
			\label{fig:gem_o3_80}
		\end{center}
	\end{figure}
	
	\begin{figure}[!htbp]
		\begin{center}
			\subfigure[Total Density ($\rho_i + \rho_e$)]{
				\includegraphics[width=2.2in, height=1.2in]{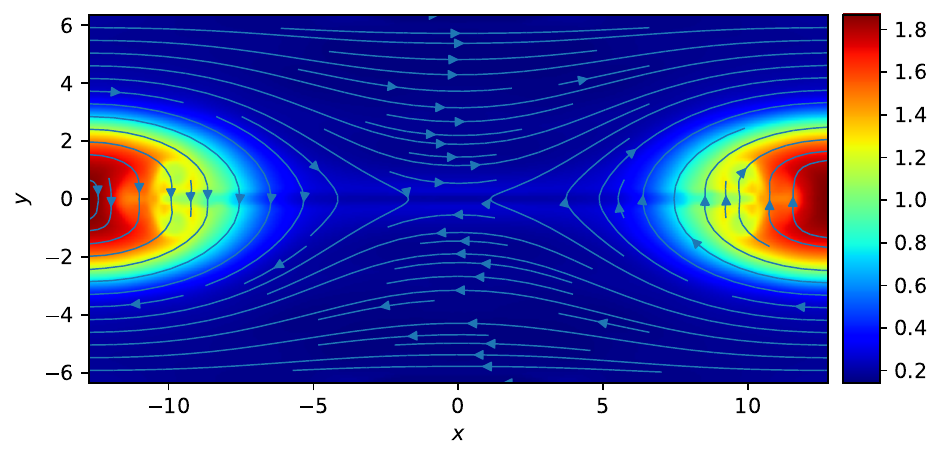}
				\label{fig:gem_o4_80_rho}}
			\subfigure[$B_z$]{
				\includegraphics[width=2.2in, height=1.2in]{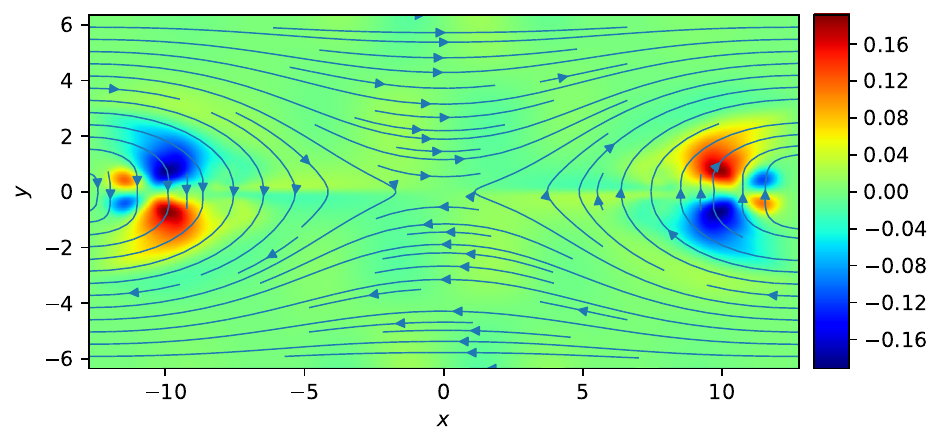}
				\label{fig:gem_o4_80_Bz}}
			\subfigure[Ion $x$-velocity]{
				\includegraphics[width=2.2in, height=1.2in]{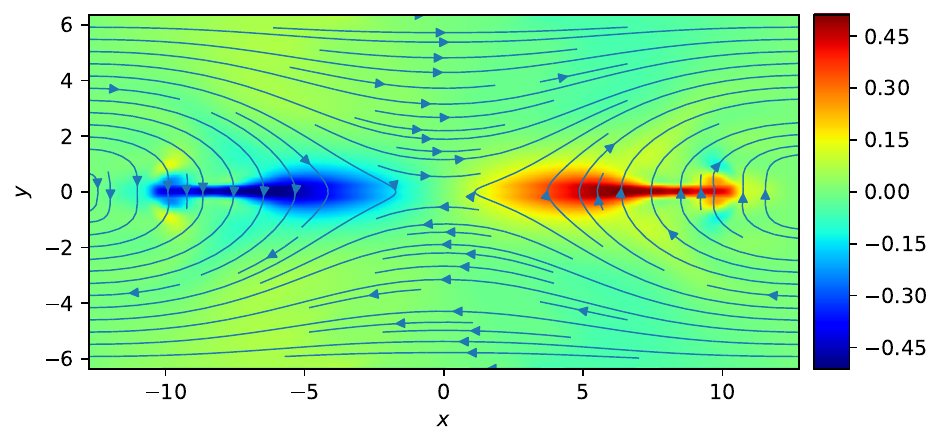}
				\label{fig:gem_o4_80_uxi}}
			\subfigure[Electron $x$-velocity]{
				\includegraphics[width=2.2in, height=1.2in]{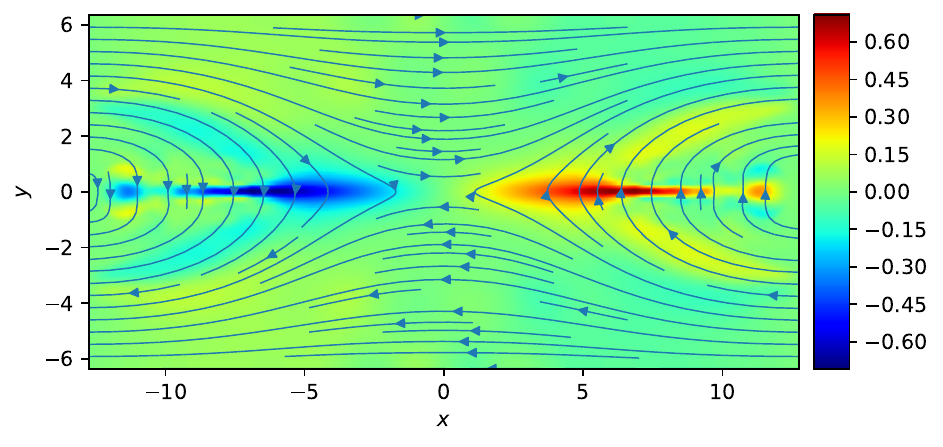}
				\label{fig:gem_o4_80_uxe}}
			\caption{\nameref{test:2d_gem}: Plots for the Total Density, $B_z$-component, Ion $x$-velocity, and Electron $x$-velocity on the rectangular grid of size $256 \times 128$, using the scheme $\fourth$, at time $t=80$. }
			\label{fig:gem_o4_80}
		\end{center}
	\end{figure}
	
        \begin{figure}[!htbp]
		\begin{center}
			\subfigure{
				\includegraphics[width=2.5in, height=1.8in]{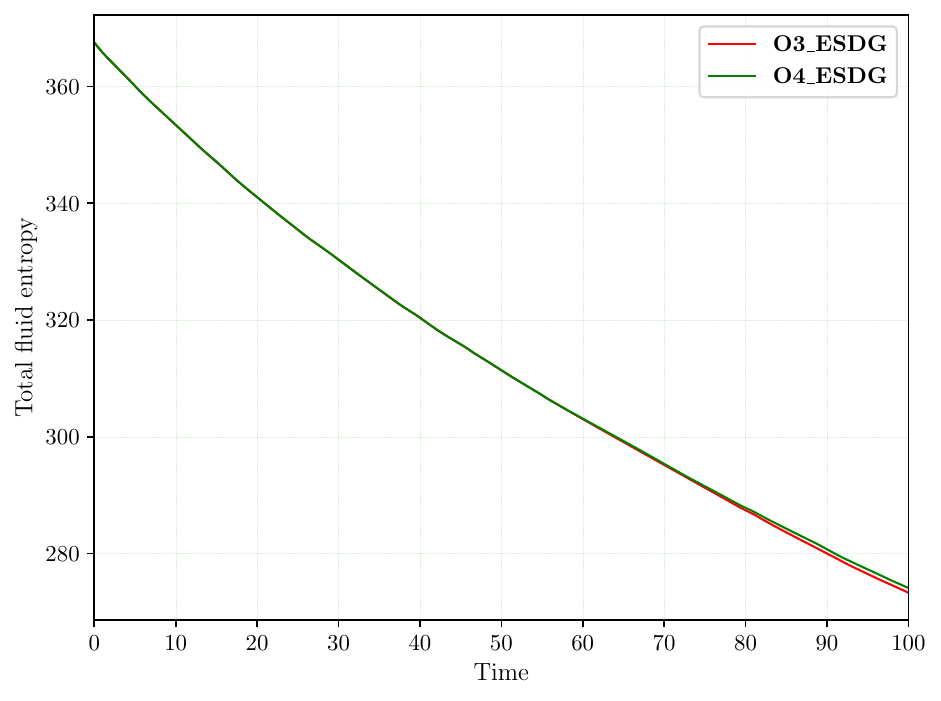}
			}
			\caption{\nameref{test:2d_gem}: Plots of the total entropy evolution using the $\third$ and $\fourth$ scheme with $256\times128$ cells.}
			\label{fig:gem_entropy}
		\end{center}
	\end{figure}
	
	\begin{figure}[!htbp]
		\begin{center}
			\subfigure{
				\includegraphics[width=2.5in, height=1.8in]{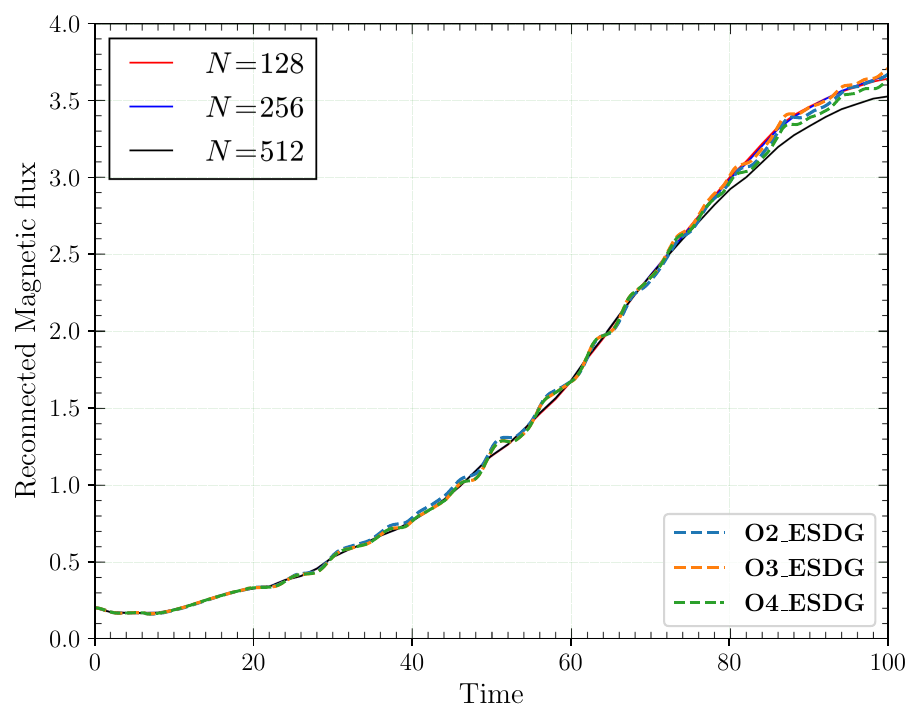}
			}
			\caption{\nameref{test:2d_gem}: Time development of the reconnected magnetic flux with $256\times128$ cells for $\second$, $\third$ and $\fourth$ schemes. We overlay the plot on plot from \cite{Amano2016} (solid lines) for comparison. }
			\label{fig:gem_recon}
		\end{center}
	\end{figure}
	
	\section{Conclusion}
	\label{sec:conclusion}
	In this article, we have designed high-order entropy stable discontinuous Galerkin schemes for the two-fluid relativistic plasma flow equations. The schemes are an extension of the entropy stable DG scheme for RHD equations and are based on the framework presented in \cite{Chen2017}. We prove that the schemes are entropy stable in one and two dimensions. We then use these schemes to simulate several challenging problems in one and two dimensions. We first show that the scheme indeed has desired order of accuracy for smooth solutions. We then test the schemes to simulate the Brio-Wu shock tube problem and demonstrate that the third- and fourth-order schemes can capture finite skin depth effects even at coarser mesh. We also show that the proposed schemes are able to compute resistive RMHD solutions very effectively. In two dimensions, we first compute the Orzag-Tang vortex problem and show that the proposed schemes are highly accurate. We also observe that the entropy decays significantly only when discontinuities are present in the solution. To demonstrate the performance of the schemes on strong shock solutions, we compute the Blast wave problem in both weakly and strongly magnetized mediums. We show that the schemes are suitable for the computation of solutions containing shocks, even in two dimensions. Finally, we simulate the GEM problem and demonstrate that the schemes are suitable for capturing complicated flow features and reconnection phenomenon.

\begin{acknowledgements}
		The work of Praveen Chandrashekar is supported by the Department of Atomic Energy,  Government of India, under project no.~12-R\&D-TFR-5.01-0520. The work of Harish Kumar is supported in parts by DST-SERB, MATRICS grant	with file No. MTR/2019/000380 and VAJRA Grant with file no. VJR-2018-000129
\end{acknowledgements}

	%
	\section*{Conflict of interest}
	The authors declare that they have no conflict of interest.
	\section*{Data Availability Declaration}
	Data will be made available on reasonable request.

\bibliographystyle{spmpsci}      
\bibliography{esdg_tfrhd}

	\appendix
	
	\reva{
    \section{Computation of Primitive Variables from the Conservative variables}
    \label{ap:cons2prim}
    The evolutionary equations for the two-fluid relativistic plasma system~\eqref{eq:TFRP_sys} relies on the conserved quantities, $D_\alpha$, $\bm{M}_\alpha$ and $\mathcal{E}_\alpha$ and primitive quantities, $\rho_\alpha$, $\vel_\alpha$ and $p_\alpha$. We use the numerical procedure given in~\cite{Bhoriya2020},\cite{Schneider1993} for the extraction of primitive variables from the conservative variables. As the procedure for obtaining primitives for the ions and electron fluid quantities is identical, we only focus on the ion-variables. Accordingly, for simplicity, we suppress the subscript $\alpha$. For the ideal equation of state \eqref{EOS}, we obtain the following quartic polynomial for the velocity variable $v=|\vel|$ with real coefficients which only depend on the conservative variables. 
    \begin{equation}
    {v}^4 + c_3 {v}^3 + c_2 {v}^2 + c_1 {v} + c_0=0, \label{polynomial}
    \end{equation}
    where
    \begin{gather*}
    c_3 = -\frac{2 \gamma (\gamma - 1) M \mathcal{E}}{(\gamma - 1)^2({M}^2 + D^2)},\;\;
    c_2= \frac{(\gamma^2 \mathcal{E}^2 + 2(\gamma - 1){M}^2 - (\gamma - 1)^2 D^2)}{(\gamma - 1)^2({M}^2 + D^2)},
    \\
    c_1=\frac{-2 \gamma {M} \mathcal{E}}{(\gamma - 1)^2({M}^2 + D^2)},\;\;
    c_0 = \frac{{M}^2}{(\gamma - 1)^2({M}^2 + D^2)},
    \end{gather*} 
    with $M=(M_x^2+M_y^2+M_z^2)^{1/2}$. The quartic polynomial~\eqref{polynomial} is solved using the Newton's root solver with the number of iterations fixed to ten. We use the following initial guess $v_0$ to initialize the root solver,
    \begin{equation*}
    	v_0=\frac{1}{2}(v_{lb}+v_{ub})+z,
    \end{equation*}
    where
	\begin{align*}
    v_{lb}&=\frac{1}{2M(\gamma - 1)}(\gamma \mathcal{E} - \sqrt{\gamma^2 \mathcal{E}^2 - 4(\gamma - 1)M^2} ),
	\\
    v_{ub}&=\min{\left(1, \frac{M}{\mathcal{E}}+ \delta \right)},
	\\
	z &=
	\begin{cases}
        \dfrac{1}{2}\left(1-\dfrac{D}{\mathcal{E}}\right)(u_{lb}-u_{ub})  &  \text{ if } u_{lb}>10^{-9}
		\\
        0 & \text{ otherwise}
	\end{cases}.
	\end{align*}
    After obtaining the velocity norm, $v$, we can extract the primitive variables using the following expressions.
    \begin{gather*}
    \rho = \frac{D}{\Gamma}, \\
    v_x=\frac{M_x}{M} v, \ \ \ v_y=\frac{M_y}{M} v, \ \ \ v_z=\frac{M_z}{M} v,  \\
    p=(\gamma-1)(\mathcal{E}- M_x v_x- M_y v_y -M_z v_z -\rho ).
    \end{gather*}
	}
	\section{Right eigenvectors for the two-dimensional two-fluid relativistic plasma system}\label{ap:right_eig}
	In this appendix, we provide expressions of the right eigenvectors for the two-dimensional relativistic plasma system~\eqref{conservedform_A}. The set of right eigenvectors of the matrix $\mathbf{A}^x$ corresponding to the eigenvalues $\eigval^x$ of the system~\eqref{conservedform_A} is obtained by setting $d=x$ in the ordered set 
	$
	\mathbf{R}^d_{\Lambda^d} 
	= 
	\left\{
	\left(
	\mathbf{R}_{\Lambda^d}^d
	\right)_n: \ n = 1, \ 2 ,\ 3,\dots, 18
	\right\}$. 
	For $n=1,2,3,\dots,18,$ the vectors $\left(\mathbf{R}_{\Lambda^d}^d\right)_n$ are given by 
	\begin{align} \label{x_eigvec}
		\left(\mathbf{R}_{\Lambda^d}^d\right)_n =
		\begin{cases}
			\Big((\mathbf{R}_{i,k}^{d})_{1 \times 5}, 
			\mathbf{0}_{1 \times 5},
			\mathbf{0}_{1 \times 8}\Big)^\top, 
			&
			1 \leq n \leq 5, \ k = n
			\\
			\Big(\mathbf{0}_{1 \times 5},
			(\mathbf{R}_{e,k}^{d})_{1 \times 5}, 
			\mathbf{0}_{1 \times 8}\Big)^\top, 
			&
			6 \leq n \leq 10, \ k = n-5
			\\
			\Big(\mathbf{0}_{1 \times 5},
			\mathbf{0}_{1 \times 5}, 
			(\mathbf{R}_{m,k}^{d})_{1 \times 8}\Big)^\top, 
			&
			11 \leq n \leq 18, \ k = n-10,
		\end{cases}
	\end{align}		
	where $\mathbf{R}^{d}_{\alpha,k}$, $\alpha \in \{i,e\}$, is the $k^{th}$ column vector of the $5 \times 5$ right eigenvector matrices $\mathbf{R}^d_\alpha$ of the flux jacobians $\dfrac{\partial\mathbf{f}^d_\alpha}{\partial \mathbf{U}_\alpha}$, and $\mathbf{R}^{d}_{m,k}$ is the $k^{th}$ column vector of the right eigenvector matrix $\mathbf{R}^d_m$ of the flux jacobian matrix $\dfrac{\partial\mathbf{f}^d_m}{\partial \mathbf{U}_m}$. The matrices $\mathbf{R}^{k}_{\alpha}$ and $\mathbf{R}^{k}_{m}$ has the following expressions. 
	
	\begin{itemize}
		\item For $d =x,y$, $\alpha \in \{i,e\}$, the right eigenvector matrix $\mathbf{R}^d_\alpha$ is given by the equation 
		\begin{equation*} 
			\mathbf{R}_{\alpha}^{d}= \left(\dfrac{\partial \mathbf{U}_\alpha}{\partial \mathbf{W}_\alpha} \right) \mathbf{R}_{\alpha,\mathbf{W}}^d
		\end{equation*}
		where $\mathbf{R}_{\alpha,\mathbf{W}}^d$ is the matrix of right eigenvectors of system~\eqref{conservedform_A} written in primitive form. For $d=x$, the matrix $\mathbf{R}_{\alpha,\mathbf{W}}^d$ has the expression 
		\begin{gather*}
			\mathbf{R}_{\alpha,\mathbf{W}}^x = 
			\begin{pmatrix*}
				\frac{1}{c_\alpha^2 h_\alpha} & 1 & 0 & 0 & \frac{1}{c_\alpha^2 h_\alpha}  
				\\
				\frac{-\sqrt{Q^x_\alpha}}{c_\alpha h_\alpha \Gamma_\alpha \rho_\alpha} & 0 & 0 & 0 &  \frac{+\sqrt{Q^x_\alpha}}{c_\alpha h_\alpha \Gamma_\alpha \rho_\alpha}
				\\
				\frac{\left(c_\alpha-\Gamma_\alpha \sqrt{Q^x_\alpha} {v_{x_\alpha}}\right) {v_{y_\alpha}}}{c_\alpha h_\alpha \Gamma^2_\alpha \rho_\alpha \left(v_{x_\alpha}^2-1\right)} & 0 & 1 & 0 & \frac{\left(c_\alpha+\Gamma_\alpha \sqrt{Q^x_\alpha} {v_{x_\alpha}}\right) {v_{y_\alpha}}}{c_\alpha h_\alpha \Gamma^2_\alpha \rho_\alpha \left(u_{x_\alpha}^2-1\right)}  
				\\
				\frac{\left(c_\alpha-\Gamma_\alpha \sqrt{Q^x_\alpha} {v_{x_\alpha}}\right) {v_{z_\alpha}}}{c_\alpha h_\alpha \Gamma^2_\alpha \rho_\alpha \left(v_{x_\alpha}^2-1\right)} & 0 & 1 & 0 & \frac{\left(c_\alpha+\Gamma_\alpha \sqrt{Q^x_\alpha} {v_{x_\alpha}}\right) {v_{z_\alpha}}}{c_\alpha h_\alpha \Gamma^2_\alpha \rho_\alpha \left(v_{x_\alpha}^2-1\right)}  
				\\
				1 & 0 & 0 & 0 & 1
			\end{pmatrix*}. 
		\end{gather*}
		\item The eigenvector matrix $\mathbf{R}^{d}_{m}$ for $d=y$ is given by 
		\begin{gather*}
			\mathbf{R}_{m}^{x}=
			\begin{pmatrix}
				0 &-1 & 0 & 0 & 0 & 0 & 1 & 0 \\
				0 & 0 & 0 & 1 & 0 &-1 & 0 & 0 \\
				0 & 0 &-1 & 0 & 1 & 0 & 0 & 0 \\
			   -1 & 0 & 0 & 0 & 0 & 0 & 0 & 1 \\
				0 & 0 & 1 & 0 & 1 & 0 & 0 & 0 \\
				0 & 0 & 0 & 1 & 0 & 1 & 0 & 0 \\
				1 & 0 & 0 & 0 & 0 & 0 & 0 & 1 \\
				0 & 1 & 0 & 0 & 0 & 0 & 1 & 0 
			\end{pmatrix}.  
		\end{gather*}
	\end{itemize}
	\begin{remark}
		For the $y-$directional flux, $\mathbf{f}^y$, we proceed similarly to obtain the set of eigenvalues $\eigval^y$ of the jacobian matrix $\dfrac{\partial \mathbf{f}^y}{\partial \mathbf{U}}$ as
		\begin{align*} 
			\eigval^y 
			= 
			\biggl\{
			& 	\frac{(1-c_i^2)v_{y_i}-(c_i/\Gamma_i) \sqrt{Q_i^y}}{1-c_i^2 |\vel|_i^2},\
			v_{y_i}, \ v_{y_i}, \ v_{y_i},
			\frac{(1-c_i^2)v_{y_i}+(c_i/\Gamma_i) \sqrt{Q_i^y}}{1-c_i^2 |\vel|_i^2}, 
			\nonumber
			\\  
			& 	\frac{(1-c_e^2)v_{y_e}-(c_e/\Gamma_e) \sqrt{Q_e^y}}{1-c_e^2 |\vel|_e^2},\
			v_{y_e}, \ v_{y_e}, \ v_{y_e},
			\frac{(1-c_e^2)v_{y_e}+(c_e/\Gamma_e) \sqrt{Q_e^y}}{1-c_e^2 |\vel|_e^2}, 
			\nonumber
			\\
			& -\chi, \ -\kappa, \ -1, \ -1, \ 1, \ 1, \ \kappa, \ \chi
			\biggr\},
		\end{align*}
		where, 
		$
		Q_\alpha^y=1-u_{y_\alpha}^2-c_\alpha^2 (u_{x_\alpha}^2+u_{z_\alpha}^2), \ \alpha \in \{ i, \ e \}
		$. 
		The corresponding right eigenvectors can be obtained by taking $d=y$ in the ordered set $\mathbf{R}^d_{\Lambda^d} = \left\{\left(\mathbf{R}_{\Lambda^d}^d\right)_n: \ n = 1, \ 2 ,\ 3,\dots, \ 18\right\}$ where $\left(\mathbf{R}_{\Lambda^d}^d\right)_n$ is defined by Eqn.~\eqref{x_eigvec} and the updated $y-$directional matrices $\mathbf{R}_{\alpha,\mathbf{W}}^y$ and $\mathbf{R}^{y}_{m}$ are given by
		\begin{gather*}
			\mathbf{R}_{\alpha,\mathbf{W}}^y = 
			\begin{pmatrix}
				\frac{1}{c_\alpha^2 h_\alpha} & 1 & 0 & 0 & \frac{1}{c_\alpha^2 h_\alpha}  
				\\
				\frac{\left(c_\alpha-\Gamma_\alpha \sqrt{Q^y_\alpha} {v_{y_\alpha}}\right) {v_{x_\alpha}}}{c_\alpha h_\alpha \Gamma^2_\alpha \rho_\alpha \left(v_{y_\alpha}^2-1\right)} & 0 & 1 & 0 & \frac{\left(c_\alpha+\Gamma_\alpha \sqrt{Q^y_\alpha} {v_{y_\alpha}}\right) {v_{x_\alpha}}}{c_\alpha h_\alpha \Gamma^2_\alpha \rho_\alpha \left(v_{y_\alpha}^2-1\right)}  
				\\
				\frac{-\sqrt{Q^y_\alpha}}{c_\alpha h_\alpha \Gamma_\alpha \rho_\alpha} & 0 & 0 & 0 &  \frac{+\sqrt{Q^y_\alpha}}{c_\alpha h_\alpha \Gamma_\alpha \rho_\alpha}
				\\
				\frac{\left(c_\alpha-\Gamma_\alpha \sqrt{Q^y_\alpha} {v_{y_\alpha}}\right) {v_{z_\alpha}}}{c_\alpha h_\alpha \Gamma^2_\alpha \rho_\alpha \left(v_{y_\alpha}^2-1\right)} & 0 & 0 & 1 & \frac{\left(c_\alpha+\Gamma_\alpha \sqrt{Q^y_\alpha} {v_{y_\alpha}}\right) {v_{z_\alpha}}}{c_\alpha h_\alpha \Gamma^2_\alpha \rho_\alpha \left(v_{y_\alpha}^2-1\right)}  
				\\
				1 & 0 & 0 & 0 & 1
			\end{pmatrix}, 
			\\
			\mathbf{R}_{m}^{y}=
			\begin{pmatrix*}
				0 & 0 & 0 &-1 & 0 & 1 & 0 & 0 \\
				0 &-1 & 0 & 0 & 0 & 0 & 1 & 0 \\
				0 & 0 & 1 & 0 &-1 & 0 & 0 & 0 \\
				0 & 0 & 1 & 0 & 1 & 0 & 0 & 0 \\
			   -1 & 0 & 0 & 0 & 0 & 0 & 0 & 1 \\
				0 & 0 & 0 & 1 & 0 & 1 & 0 & 0 \\
				1 & 0 & 0 & 0 & 0 & 0 & 0 & 1 \\
				0 & 1 & 0 & 0 & 0 & 0 & 1 & 0
			\end{pmatrix*}  
		\end{gather*}
	\end{remark}

\end{document}